\documentclass[11pt]{article}
\usepackage{graphicx}
\usepackage{amsfonts}
\usepackage{amsmath}
\usepackage{amssymb}
\usepackage{amsthm}
\usepackage{stmaryrd}
\usepackage{hyperref}
\usepackage{cite}
\usepackage[all]{xy}
\usepackage[utf8]{inputenc}
\usepackage{tikz-cd}

\usepackage[letterpaper, margin=1in]{geometry}

\theoremstyle{plain}
\newtheorem{thm}{Theorem}[section]
\newtheorem{prop}[thm]{Proposition}
\newtheorem{corr}[thm]{Corollary}
\newtheorem{lem}[thm]{Lemma}

\theoremstyle{remark}
\newtheorem{rem}[thm]{Remark}

\theoremstyle{definition}
\newtheorem{defn}[thm]{Definition}
\newtheorem{ex}[thm]{Example}

\newcommand{\inv}{^{-1}}
\newcommand{\units}{^{(0)}}
\newcommand{\ra}{\rightarrow}

\newcommand{\calA}{\mathcal A}
\newcommand{\calG}{\mathcal G}
\newcommand{\calH}{\mathcal H}
\newcommand{\calK}{\mathcal K}
\newcommand{\calL}{\mathcal L}

\newcommand{\ca}{\mathcal}

\newcommand{\Z}{\mathbb Z}

\newcommand{\cst}{$C^*$}

\newcommand{\inter}{\mathrm{int}}

\newcommand{\frgr}{\mathfrak{Gr}}
\newcommand{\kk}{\mathrm{KK}}
\newcommand{\ind}{\mathrm{Ind}}
\newcommand{\E}{\underline{E}}

\title{A going-down principle for \'etale groupoids and the Baum--Connes conjecture}
\author{Kai Mao}

\begin{document}

\maketitle

\begin{abstract}
    We study a going-down principle for \'etale groupoids and its applications, extending previous results for locally compact groups by Chabert, Echterhoff and Oyono-Oyono, and for ample groupoids by B\"onicke and by B\"onicke--Dell'Aiera. The proof in the general \'etale groupoid setting is based on a more detailed study of groupoid simplicial complexes. For the most commonly considered kind of going-down functors, we recover the result of B\"onicke and Proietti, which they proved via a categorical approach and used to establish the split injectivity of the Baum--Connes assembly map for \'etale groupoids that are strongly amenable at infinity. We also study a bicategorical functoriality, involving the induction functors from \'etale groupoid correspondences introduced by Miller. This yields a bicategorical interpretation of the induction-restriction adjunction. The going-down principle is also applied to the proof of continuity of topological K-theory of \'etale groupoids and the study of the scope of validity of K\"unneth formulas.
\end{abstract}

\tableofcontents

\section{Introduction}
If $\calG$ is a $\sigma$-compact locally compact Hausdorff groupoid with a Haar system and $A$ is a $\calG$-\cst-algebra, the Baum--Connes conjecture states that, the assembly map
\[\mu_{\calG,A}:K^{top}_*(\calG,A):=\varinjlim_{X\subseteq \E\calG}\kk^\calG_*(C_0(X),A)\ra K_*(A\rtimes_r\calG)\]
is an isomorphism, where $K^{top}_*(\calG,A)$ is called the topological K-theory of $\calG$ with coefficients in $A$, $X$ runs over all $\calG$-compact subsets of a classifying space $\E\calG$ of proper actions of $\calG$. The conjecture is proved to be true for groupoids satisfying the Haagerup property, for example amenable groupoids, in \cite{tu1999conjecture-bc} by Tu. It is known to be false in general in \cite{HigsonLafforgueSkandalis} by Higson, Lafforgue and Skandalis even with trivial coefficients.

When $G$ is a locally compact Hausdorff group which is amenable at infinity, i.e. admits an amenable action on some compact space, the Baum--Connes assembly map is split injective with any coefficients $A$. One of the proofs is using a going-down principle in \cite{chabert2004going} by Chabert, Echterhoff and Oyono-Oyono, where a key technical ingredient can be transformed as the following version in the setting of \'etale groupoids.

\begin{thm}\label{main thm A}
    Let $\calG$ be a second countable locally compact Hausdorff \'etale groupoid, $A,B$ be separable $\calG$-\cst-algebra. If $x\in \kk^\calG_0(A,B)$ such that
    \[\kk^\calH_*(C_0(\calH\units),A|_\calH)\xrightarrow{-\otimes res^\calG_\calH(x)}\kk^\calH_*(C_0(\calH\units),B|_\calH)\]
    is an isomorphism for any proper open subgroupoid $\calH\subseteq \calG$. Then
    \[-\otimes x: K^{top}_*(\calG;A)\ra K^{top}_*(\calG;B)\]
    is an isomorphism.
\end{thm}

This theorem was already proved in \cite[Theorem 4.4]{bonicke2024categorical} by B\"onicke and Proietti in a categorical approach. Our main purpose is to prove this theorem by following the strategy in \cite{chabert2004going} for locally compact groups and \cite{bonicke2020going} for ample groupoids and overcoming some topological difficulties. In fact, we will obtain a generalization for going-down functors (i.e a collection of cohomological functors that are compatible with suspension and induction like $(\kk_*^\calH(-,A|_\calH))_{\calH\subseteq \calG}$, where $\calH$ runs over all proper open subgroupoids and $\calG$ itself). Here is our main theorem.

\begin{thm}\label{main thm B}
    Let $F, G$ be two going-down functors for a second countable locally compact Hausdorff \'etale groupoid $\calG$, and $\Lambda:F\ra G$ be a going-down transformation. Suppose that for all proper open subgroupoids $\calH\subseteq \calG$ and for all $n\in \mathbb Z$,
    \[
    \Lambda^n_\calH(C_0(\calH\units)): F^n_\calH(C_0(\calH\units))\ra G^n_\calH(C_0(\calH\units))
    \]
    is an isomorphism. 
    
    Then for all $n\in \mathbb Z$, 
    \[\varinjlim_{Z\subseteq \E\calG} \Lambda^n_\calG(C_0(Z)): \varinjlim_{Z\subseteq \E\calG} F_\calG^n(C_0(Z))\ra \varinjlim_{Z\subseteq \E\calG} G_\calG^n(C_0(Z))\]
    is an isomorphism, where $Z$ runs over all $\calG$-compact subsets of $\E\calG$.
\end{thm}

This is a generalization from the ample groupoid case, which was proved in \cite{bonicke2018going} by B\"onicke.

Then we represent a proof of the following result about the Baum--Connes conjecture by following the same ideas as in \cite{bonicke2020going}, which was firstly proved in \cite{bonicke2024categorical} by B\"onicke and Proietti.

\begin{thm}
    Let $\calG$ be a second countable locally compact Hausdorff \'etale groupoid, $A$ be a separable $\calG$-\cst-algebra. If $\calG$ is strongly amenable at infinity, then the assembly map
    \[\mu_{\calG,A}:K^{top}_*(\calG;A)\ra K_*(A\rtimes_r\calG)\]
    is split injective.
\end{thm}

Besides, same as in \cite{chabert2004going} and \cite{bonicke2019going}, we apply this going-down principle to the study of the continuity of topological K-theory and the scope of validity of K\"unneth formula.

One of our improvements in details is the study about groupoid simplicial complexes. Here is a summary of our results.

\begin{thm}
    Let $\calG$ be a locally compact Hausdorff \'etale groupoid, $(Y,\Delta)$ be a $\calG$-simplicial complex satisfying hypotheses $(H_1)$ and $(H_2)$ (see definition \ref{defn groupoid simplicial complex} and definition \ref{defn H1 H2}).
    \begin{enumerate}
        \item Its geometric realization $|\Delta|$ is a locally compact Hausdorff $\calG$-space.
        \item If $Y$ is a proper $\calG$-space, then so is $|\Delta|$. If $|\Delta^0|$ is $\calG$-compact, then so is $|\Delta|$.
        \item Rips complexes $(\calG,\Delta_K(\calG))$ (definition \ref{defn Rips complex}) are proper $\calG$-compact $\calG$-simplicial complexes satisfying hypotheses $(H_1)$ and $(H_2)$.
        \item If $(Y,\Delta)$ is typed (definition \ref{defn typed}), then for any $1\leqslant m\leqslant\dim\Delta$, we have a $\calG$-equivariant homeomorphism
        \[|\Delta^m|\setminus |\Delta^{m-1}|\cong center(m,\Delta)\times \mathbb R^m,\]
        where $center(m,\Delta)$ is the closed subset of $|\Delta|$ consisting of centers of all $m$-simplices, which is a $\calG$-invariant closed subset of some \'etale $\calG$-space.
        \item The barycentric subdivision $(Y',\Delta')$ (definition \ref{defn barycentric subdivision}) of $(Y,\Delta)$ is a typed $\calG$-simplicial complex satisfying hypotheses $(H_1)$ and $(H_2)$. Its geometric realization $|\Delta'|$ is $\calG$-equivariantly homeomorphic to $|\Delta|$.
    \end{enumerate}
\end{thm}

Another improvement is that, motivated by the induction functor $\ind_\Omega$ from an \'etale groupoid correspondence $\Omega$, defined by Miller in \cite{miller2024functors}, we complete the functoriality as a bicategorical functoriality.

\begin{thm}
Let $\frgr$ be the bicategory of second countable \'etale groupoids and second countable locally compact Hausdorff correspondences, $\mathfrak{KK}$ be the bicategory of equivariant Kasparov categories and functors. Then there exists a well-defined pseudofunctor $\frgr\ra \mathfrak{KK}$, which assigns a groupoid $\calG$ to the Kasparov category $\kk^\calG$, and assigns a correspondence $\calG\curvearrowright \Omega \curvearrowleft \calH$ to the induction functor $\ind_\Omega\in \mathrm{Fun}(\kk^\calH,\kk^\calG)$ (as defined in \cite{miller2024functors}).
\end{thm}

A more precise statement should be found in theorem \ref{pseudofunctor}.

In section 2, we start with some fundamental topological properties of \'etale groupoids and actions. The local structure of proper actions of an \'etale groupoid will be essential in a descent technique used in the proof of theorem \ref{main thm B}. Then we recall the notion of $C_0(X)$-algebras and upper-semicontinuous bundles. The induction functor from an \'etale groupoid correspondence defined by Miller in \cite{miller2022k} and \cite{miller2024functors} is introduced. Some other technical lemmas are placed at the end of this section.

In section 3 we introduce the definition and properties of groupoid simplicial complexes. It was mentioned in \cite{bonicke2020going} that there are topological difficulties to overcome to develop a going-down principle for \'etale groupoids. One of the reasons is that the definition of groupoid simplicial complex in \cite[Definition 7.5]{bonicke2020going} works for Rips complexes of ample groupoids but not for Rips complexes of \'etale groupoids in general. We point out that the two hypotheses $(H_1)$ and $(H_2)$ (see definition \ref{H_1 H_2}), which appear firstly in \cite{bessi2023}, will ensure groupoid simplicial complexes to have good topological properties. Rips complexes $(\calG,\Delta_K(\calG))$ are proper $\calG$-compact $\calG$-simplicial complexes satisfying hypotheses $(H_1)$ and $(H_2)$ in our definition, which provide a sufficient family of proper $\calG$-spaces to approximate $\E\calG$. 

Section 4 is dedicated to a bicategorical proof of the induction-restriction adjunction. Let $\calH$ be a proper open subgroupoid of $\calG$, $A$ be an $\calH$-\cst-algebra and $B$ be a $\calG$-\cst-algebra, it is proved in \cite[Theorem 6.2]{bonicke2020going} and \cite[Theorem 2.3]{bonicke2024categorical} that there is a compression isomorphism
\[comp^\calG_\calH:\kk^\calG(\ind^\calG_\calH A,B)\ra \kk^\calH(A,B|_\calH).\]
In \cite{antunes2021bicategory} it was shown that \'etale groupoids, correspondences and bi-equivariant continuous maps form a bicategory $\frgr$. We will construct a pseudofunctor that sends a groupoid $\calG$ to the category $\kk^\calG$ and sends a correspondence $\calG\curvearrowright \Omega\curvearrowleft \calH$ to an induction functor $\ind_\Omega\in \mathrm{Fun}(\kk^\calH,\kk^\calG)$ as defined by Miller in \cite{miller2024functors} where the conditions on the level of 1-cells are already proved in that paper. And we will see that the induction-restriction adjunction arises from an internal adjunction in the bicategory $\frgr$ of \'etale groupoids and correspondences.

In section 5 we prove our main theorem \ref{main thm B}. The strategy is the same as in \cite{chabert2004going} and \cite{bonicke2020going}: we approximate $\E\calG$ by proper $\calG$-compact $\calG$-simplicial complexes $(Y,\Delta)$ satisfying $(H_1)$ and $(H_2)$, and apply induction on the dimension of $\calG$-simplicial complexes. Two facts will play important roles in the proof. Firstly the geometrical realization of a 0-dimensional proper $\calG$-compact $\calG$-simplicial complex satisfying $(H_1)$ and $(H_2)$ is a $\calG$-invariant closed subset of an \'etale $\calG$-space; secondly if a proper $\calG$-compact $\calG$-simplicial complexes has dimension $n$, the centers of all $n$-simplices form a 0-dimensional simplicial complex like this. One of the topological difficulties was that if $(Y,\Delta)$ is of dimension 0, the anchor map may be not a local homeomorphism but the composition of a closed inclusion with a local homeomorphism. This will be remedied by considering one more exact sequence (see the diagram \ref{diagram 4}) compared to the proof in \cite{bonicke2020going}. We conclude this paper with several applications, including the split injectivity of the Baum--Connes assembly map for \'etale groupoids that are strongly amenable at infinity, continuity of topological K-theory and validity of K\"unneth formulas.

\section*{Acknowledgment}
The content of this paper covers part of the doctoral thesis of the author. I would like to sincerely thank Herv\'e Oyono-Oyono, my advisor, who gives me a lot of guides and inspirations about this topic. Some essential ideas (for example, the notion of hypotheses $(H_1)$ and $(H_2)$) are in debt to him. I would like to thank Georges Skandalis for his comment in remark \ref{remark of Skandalis}. I would also like to thank an anonymous referee who reminds me that theorem \ref{main thm A} was already proved in \cite{bonicke2024categorical} as an equivalent form of \cite[Theorem 4.4]{bonicke2024categorical}. This work gets help from the project OpART (ANR-23-CE40-0016) of the Agence Nationale de la Recherche.

\section*{Notations}
If $A$ is a subset of a topological space $X$, we denote the interior of $A$ by $\inter_X(A)$ or $\inter(A)$.

If $X,Y,Z$ are topological spaces, $f:Y\ra X$ and $g:Z\ra X$ are continuous maps, we define the fiber product $Y\times_{f,X,g}Z$ as $\{(y,z)\in Y\times Z:f(y)=g(z)\}$, equipped with topology as a subspace of $Y\times Z$. When there is no ambiguity we can write it as $Y\times_X Z$ or $Y\times_{f,g}Z$.

\section{Preliminaries}

\subsection{Groupoids and actions}

Let $\calG$ be a locally compact Hausdorff groupoid (see \cite{williams2019tool} for detailed definitions). We denote its unit space by $\calG\units$, which is a closed subspace of $\calG$ (see \cite[lemma 1.6]{williams2019tool}). We can also denote the groupoid $\calG$ with unit space $\calG\units=X$ by $\calG\rightrightarrows X$. Let $r_\calG,s_\calG$ be respectively the range and source maps from $\calG$ to $\calG\units$, and when there is no ambiguity we write simply $r$ and $s$. The set of composable pairs is therefore the fiber product $\calG^{(2)}:=\calG\times_{s,\calG\units,r}\calG=\{(\gamma_1,\gamma_2)\in \calG\times \calG:s(\gamma_1)=r(\gamma_2)\}$. For any subsets $A,B$ of $\calG\units$ and any points $x,y\in \calG\units$, we use the following notations
\[\calG_A=s\inv(A),\quad \calG^B=r\inv(B),\quad \calG^B_A=\calG_A\cap \calG^B,\]
\[\calG_x=s\inv(x),\quad \calG^y=r\inv(y),\quad \calG^y_x=\calG_x\cap \calG^y.\]
For any two subsets $U,V$ of $\calG$, the set $UV$ is defined to be $\{uv:u\in U,v\in V,(u,v)\in \calG^{(2)}\}$, and $U\inv$ is defined to be $\{u\inv:u\in U\}$.

A strict morphism of locally compact Hausdorff groupoids $f:\calG\ra \calH$ is a continuous map such that for any $(\gamma_1,\gamma_2)\in \calG^{(2)}$, $(f(\gamma_1),f(\gamma_2))\in \calH^{(2)}$ and $f(\gamma_1)f(\gamma_2)=f(\gamma_1\gamma_2)$, and for any $\gamma\in \calG$, $f(\gamma\inv)=f(\gamma)\inv$.

We say that $\calG$ is proper, if the map $(r,s):\calG\ra \calG\units\times \calG\units$ is proper. We say that $\calG$ is principal, if $(r,s)$ is injective. And we say that $\calG$ is $r$-discrete, if $\calG\units$ is open in $\calG$.

We say that $\calG$ is \'etale if $r$ is a local homeomorphism. That is, for any $\gamma\in \calG$, there exists an open neighborhood $U$ of $\gamma$ in $\calG$ such that $r|_U$ is a homeomorphism onto an open of $\calG\units$. We say that a subset $B\subseteq \calG$ is a bisection if $r|_B$ and $s|_B$ are homeomorphisms onto opens of $\calG\units$. If an $r$-discrete locally compact Hausdorff groupoid is given, it is \'etale if and only if $r$ is open, if and only if it has a topological basis consisting of bisections (see \cite[proposition 1.29]{williams2019tool}). \textbf{Throughout this paper, all \'etale groupoids are assumed to be locally compact Hausdorff.}

Recall that $\calG$ acts on (the left of) a topological space $X$, or $X$ is a (left) $\calG$-space, if there is a continuous map (called anchor map) $\rho: X\ra \calG\units$ and a continuous map $\calG\times_{s,\calG\units,\rho}X\ra X, (\gamma,x)\mapsto \gamma x$, such that for any $x\in X$, $\rho(x)x=x$, and for any $(\gamma_1,\gamma_2)\in \calG^{(2)}$ such that $s(\gamma_2)=\rho(x)$, we have $\rho(\gamma_2x)=r(\gamma_2)$ and $\gamma_1(\gamma_2x)=(\gamma_1\gamma_2)x$. Similarly, we can define the right action. And when $U$ is a subset of $\calG$, $W$ is a subset of $\calG$, $UW$ is defined to be
\[\{uw:u\in U, w\in W, s(u)=\rho(w)\}.\]
When there is no ambiguity, for any $x\in \calG\units$ and $A\subseteq \calG\units$, we write $X_x=\rho\inv(x)$, $X_A=\rho\inv(A)$.

We can construct a groupoid out of an action of a groupoid. If $X$ is a left $\calG$-space, for two elements of 
 $\calG\times_{s,\calG\units,\rho}X$, we say that $(\gamma_1,x_1)$ and $(\gamma_2,x_2)$ are composable if $x_1=\gamma_2x_2$. And in this case we define
\[(\gamma_1,x_1)(\gamma_2,x_2)=(\gamma_1\gamma_2,x_1)\]
and the inverse map
\[i_{\calG\ltimes X}: \calG\times_{s,\calG\units,\rho}X\ra \calG\times_{s,\calG\units,\rho}X, (\gamma,x)\mapsto (\gamma\inv, \gamma x).\]
These maps make $\calG\times_{s,\calG\units,\rho}X$ a topological groupoid, which is denoted as $\calG\ltimes X$. Assume that $X$ is a locally compact Hausdorff space; then $\calG\ltimes X$ is a locally compact Hausdorff groupoid. When $\calG$ is \'etale, the groupoid $\calG\ltimes X$ is also \'etale. The unit space of $\calG\ltimes X$ is $\{(\rho(x),x):x\in X\}$, which is canonically identified with $X$. Thus, we define the source and range maps as
\[s_{\calG\ltimes X}:\calG\ltimes X\ra X, (\gamma,x)\mapsto x,\]
\[r_{\calG\ltimes X}:\calG\ltimes X\ra X, (\gamma,x)\mapsto \gamma x.\]
Similarly, we can define the groupoid $X\rtimes \calG$ when $X$ is a right $\calG$-space.

Given a $\calG$-space $X$ with anchor map $\rho:X\ra \calG\units$, we say that $X$ is a free $\calG$-space, if $\calG\ltimes X$ is principal; we say that $X$ is a proper $\calG$-space, if $\calG\ltimes X$ is proper; we say that $X$ is an \'etale $\calG$-space, if $\rho$ is a local homeomorphism.

\begin{ex}
    The action of a locally compact Hausdorff groupoid on itself is free and proper. This action is \'etale when the groupoid is \'etale. 
\end{ex}

\begin{prop}\label{groupoid action-proper action}
    Let $\calG$ be a locally compact Hausdorff groupoid and $X$ be a locally compact Hausdorff $\calG$-space with anchor map $\rho: X\ra \calG\units$. The following conditions are equivalent:
    \begin{enumerate}
        \item The action of $\calG$ on $X$ is proper.
        \item For any compact subset $K$ of $X$, $\{\gamma\in \calG: \gamma K\cap K\neq \emptyset\}$ is a compact subset of $\calG$.
        \item If $(x_\lambda)_{\lambda\in \Lambda}$ is a convergent net in $X$, $(\gamma_\lambda)_{\lambda\in \Lambda}$ is a net in $\calG$, such that for every $\lambda\in \Lambda$ we have $s(\gamma_\lambda)=\rho(x_\lambda)$ and $(\gamma_\lambda x_\lambda)_{\lambda\in \Lambda}$ is a convergent net, then $(\gamma_\lambda)_{\lambda\in \Lambda}$ has a convergent subnet.
    \end{enumerate}
\end{prop}

\begin{proof} See proposition 2.14 of \cite{2006Non} and proposition 1.84 of \cite{goehle2009groupoid}.\end{proof}

\begin{prop} \label{groupoid action-proper action base change}
    \textnormal{\cite[Proposition 2.20]{2006Non}}
    Let $\calG$ be a locally compact groupoid acting on two spaces $Y$ and $Z$. Suppose that the action of $\calG$ on $Z$ is proper and that $Y$ is Hausdorff, then $\calG$ acts properly on $Y\times_{\calG\units}Z$.
\end{prop}

Take $Z$ to be $\calG\units$, we can conclude that any action of a proper locally compact Hausdorff groupoid is proper.

Let $X$ be a $\calG$-space (not necessarily locally compact Hausdorff). The action defines an equivalence relation on $X$ as $x\sim y$ if and only if there exists $\gamma\in \calG^{\rho(x)}_{\rho(y)}$ such that $x=\gamma y$. We denote the quotient space by $X/\calG$, equipped with quotient topology. We summarize the properties that we need about the orbit space $X/\calG$ in the following two propositions.

\begin{prop}\label{orbit space of free proper action of etale grpd}
    If $\calG$ is an \'etale groupoid and $X$ is a free proper $\calG$-space, the quotient map $q: X\ra X/\calG$ is a local homeomorphism and $X/\calG$ is Hausdorff.
\end{prop}
\begin{proof} See \cite[Proposition 2.19, Lemma 2.12]{antunes2021bicategory}.\end{proof}

\begin{prop}\label{orbit space of LCH by proper action}
    If $\calG$ is a locally compact Hausdorff groupoid such that the source map or the range map of $\calG$ is open, $X$ is a locally compact Hausdorff proper $\calG$-space, then $X/\calG$ is also a locally compact Hausdorff space, the quotient map $q:X\ra X/\calG$ is open. Moreover, if $X$ is second countable, so is $X/\calG$.
\end{prop}
\begin{proof} See \cite[Proposition 1.69, Proposition 1.85]{goehle2009groupoid}.\end{proof}

When there exists a compact subset $K\subseteq X$ such that $X=\calG K$, we say that $X$ is $\calG$-compact or cocompact.

\begin{lem}\label{G-compact iff quotient compact}
    Let $\calG$ be a locally compact Hausdorff groupoid with open source or range maps, $X$ is a locally compact Hausdorff proper $\calG$-space with anchor map $\rho$, then $X$ is $\calG$-compact if and only if $X/\calG$ is compact.
\end{lem}
\begin{proof}
    Let $q:X\ra X/\calG$ be the quotient map. By proposition \ref{orbit space of LCH by proper action}, $q$ is a continuous open map between two locally compact Hausdorff spaces. If $X$ is $\calG$-compact, assume that $X=\calG K$ for some compact subset $K$, then $X/\calG=q(K)$ is compact. Conversely, assume that $X/\calG$ is compact. For every $x\in X$, let $V_x$ be a relatively compact open neighborhood of $x$. Since $q$ is open, $(q(V_x))_x$ form an open cover of $X/\calG$. So we can select finitely many $x_1,\cdots, x_n\in X$ such that $X/\calG=\cup_{i=1}^nq(V_{x_i})$. Take $K=\cup_{i=1}^n \overline{V_{x_i}}$, we will have $X=\calG K$.
\end{proof}

\begin{lem}\label{G-compact subsets are closed}
    Let $\calG$ be a locally compact Hausdorff groupoid, $X$ is a locally compact Hausdorff proper $\calG$-space with anchor map $\rho$, $Z\subseteq X$ is a $\calG$-compact subset, then $Z$ is closed in $X$.
\end{lem}

\begin{proof}
    Let $q:X\ra X/\calG$ be the quotient map. Assume that $Z=\calG K$ where $K$ is a compact subset of $Z$, then $Z=q\inv(q(K))$ is closed since $q(K)$ is closed in $X/\calG$.
\end{proof}

\begin{lem}\label{maps between G-compact sets are proper}
    \textnormal{\cite[Lemma 4.2.1]{bonicke2018going}}
    Let $\calG$ be a locally compact Hausdorff groupoid, $X$ be a locally compact Hausdorff $\calG$-compact $\calG$-space and $Y$ be a locally compact Hausdorff proper $\calG$-space. Then any $\calG$-equivariant continuous map $X\ra Y$ is proper.
\end{lem}

The properness of a groupoid allows us to perform an averaging process by using a cutoff function.

\begin{defn}
    Let $\calG$ be an \'etale groupoid. A cutoff function for $\calG$ is a continuous function $c:\calG\units\ra \mathbb R_{\geqslant 0}$ such that
    \begin{enumerate}
        \item the map $s:supp(c\circ r)\ra \calG\units$ is proper;
        \item for any $x\in \calG\units$, $\sum_{\gamma\in \calG_x}c(r(\gamma))=1$.
    \end{enumerate}
\end{defn}

\begin{prop}\label{cuttoff function exist}
    \textnormal{\cite[Proposition 6.11]{tu1999conjecture-n}}
    If $\calG$ is a proper \'etale groupoid such that $\calG\units/\calG$ is $\sigma$-compact, then there exists a cutoff function. Moreover, if $\calG\units/\calG$ is compact, the cutoff function can be chosen to be compactly supported.
\end{prop}

\subsection{Local structure of proper actions of \'etale groupoids}

Recall that, when $\Gamma$ is a discrete group and $Z$ is a proper $\Gamma$-space, for every point $z$ of $Z$ and an open neighborhood $U$ of it, there is an open neighborhood $V$ in $Z$ of $z$, such that $V\subseteq U$ and $V$ is invariant under action of the stabilizer $\Gamma_z$ of $z$ (which is a finite subgroup), and for $\gamma\in \Gamma\setminus\Gamma_z$, $\gamma V\cap V=\emptyset$. We will show an analog for proper actions of étale groupoids, where finite subgroups are replaced by proper open subgroupoids. The proof here is adapted from theorem 3.3.4 of \cite{bessi2023}. This is also proved in \cite[Proposition 3.2]{bonicke2024categorical}.

\begin{lem}\label{two section coincide}
    Let $f: Z\ra X$ be a local homeomorphism between locally compact Hausdorff spaces. If $s_1,s_2: X\ra Z$ are two sections of $f$, and for $x_0\in X$, $s_1(x_0)=s_2(x_0)$, then there exists an open neighborhood $U$ of $x_0$ in $X$, such that $s_1|_U=s_2|_U$.
\end{lem}
\begin{proof} Let $V$ be an open neighborhood of $s_1(x_0)=s_2(x_0)$ such that $f(V)$ is open and $f|_V$ is a homeomorphism from $V$ to $f(V)$. Let $U=f(V)$. So for any $x\in U$,
\[s_1(x)=f|_V\inv \circ f|_V\circ s_1(x)=f|_V\inv(x)=f|_V\inv \circ f|_V\circ s_2(x)=s_2(x).\]
\end{proof}

\begin{prop}\label{local structure of proper cocompact action of étale groupoid}
    Let $\calG\rightrightarrows X$ be an étale groupoid, $Z$ be a locally compact Hausdorff space with left proper action of $\calG$ with anchor map $\rho: Z\ra X$. Then for any $z_0\in Z$ and a neighborhood $U$ of $z_0$, there exists an open neighborhood $V$ of $z_0$ in $U$, and a proper open subgroupoid $\calH\subseteq \calG$, such that
    \begin{enumerate}
        \item $V\subseteq \overline{V}\subseteq U$, $\overline{V}$ is compact;
        \item $\rho(V)\subseteq \calH\units$ and $\calH\units$ is a relatively compact open of $X$;
        \item $V$ is $\calH$-invariant;
        \item For any $\gamma \in\calG\setminus \calH$, $\gamma \overline{V}\cap \overline{V}=\emptyset$.
    \end{enumerate}
    Moreover, the stabilizer $F=\{\gamma\in \calG^{\rho(z)}:\gamma z=z\}$ is a finite group, it admits an action on a relatively compact open neighborhood $A$ of $\rho(z)$ in $X$, such that $\calH$ is isomorphic to $F\ltimes A$.
\end{prop}
\begin{proof} The action of $\calG$ on $Z$ is proper, hence $F$ is finite. For every $\gamma\in F$, take an open bisection $U_\gamma$ containing $\gamma$ such that $s|_{U_\gamma}, r|_{U_\gamma}$ are homeomorphisms onto opens of $X$. By definition, we have $s|_{U_\gamma}\inv(\rho(z))=\gamma$. Without loss of generality, let all the $U_\gamma$ be pairwisely disjoint.

After replacing each $U_\gamma$ by $U_\gamma\cap (U_{\gamma\inv})^{-1}$, we can choose these $U_\gamma$ such that $U_\gamma=(U_{\gamma\inv})^{-1}$.

Let $A_1=\cap_{\gamma\in F} s(U_\gamma)=\cap_{\gamma\in F} r(U_{\gamma\inv})=\cap_{\gamma\in F}r(U_\gamma)$. It is an open neighborhood of $\rho(z_0)$. For $\gamma$ in $F$ and $a$ in $A_1$, define $t_\gamma(a)=r\circ s|_{U_\gamma}\inv(a)$. Particularly, $t_\gamma(\rho(z_0))=\rho(z_0)$. And $t_{\gamma\inv}a=r|_{U_{\gamma\inv}}\circ s|_{U_{\gamma\inv}}\inv(a)=s|_{U_{\gamma}}\circ r|_{U_{\gamma}}\inv(a)$, so $t_{\gamma}(t_{\gamma\inv}a)=a$.

Let $\gamma_1,\gamma_2\in F$. Consider the following two maps:
\[\phi_1:A_1\ra \calG: a\mapsto s|_{U_{\gamma_1\gamma_2}}\inv(a); \]
and 
\[\phi_2: A_1\ra \calG :a\mapsto s|_{U_{\gamma_1}}\inv(t_{\gamma_2}a) s|_{U_{\gamma_2}}\inv(a).\]

They are both sections of the local homeomorphism $s$, since for all $a$ in $A_1$,
\[s(\phi_2(a))=s(s|_{U_{\gamma_1}}\inv(t_{\gamma_2}a) s|_{U_{\gamma_2}}\inv(a))=s(s|_{U_{\gamma_2}}\inv(a))=a=s(\phi_1(a)).\]

And $\phi_1(\rho(z_0))=\gamma_1\gamma_2=\phi_2(\rho(z_0))$. Hence, by lemma \ref{two section coincide}, there exists an open neighborhood $A_{2,\gamma_1,\gamma_2}$ of $\rho(z_0)$ in $A_1$, such that $\phi_1|_{A_{2,\gamma_1,\gamma_2}}=\phi_2|_{_{A_{2,\gamma_1,\gamma_2}}}$. In particular, for $a\in A_{2,\gamma_1,\gamma_2}$,
\[t_{\gamma_1\gamma_2}a=r(\phi_1(a))=r(\phi_2(a))=r(s|_{U_{\gamma_1}}\inv(t_{\gamma_2}a))=t_{\gamma_1}(t_{\gamma_2}a).\]

Let $A_2=\cap_{\gamma_1,\gamma_2\in F\times F} A_{2,\gamma_1,\gamma_2}$. It is a neighborhood of $\rho(z_0)$, and on $A_2$, $t_{\gamma_1\gamma_2}=t_{\gamma_1}t_{\gamma_2}$. Let $A_3=\cap_{\gamma\in F}t_\gamma(A_2)$. So $A_3$ has a well-defined left action of $F$. Then let $A$ be an $F$-invariant relatively compact open neighborhood of $\rho(z_0)$ in $A_3$.

Define such a continuous map 
\[\Phi: F\ltimes A\ra \calG, (\gamma,a)\ra s|_{U_\gamma}\inv(a).\]
The map $\Phi$ is a strict morphism of groupoids. This is because
\begin{align*}
    \Phi(\gamma_1,t_{\gamma_2}a)\Phi(\gamma_2,a) & = s|_{U_{\gamma_1}}\inv(t_{\gamma_2}a) s|_{U_{\gamma_2}}\inv(a) \\
    & = s|_{U_{\gamma_1\gamma_2}}\inv(a)\\
    & = \Phi(\gamma_1\gamma_2,a).
\end{align*}

For every $\gamma\in F$, $\Phi|_{\{\gamma\}\times A}=s|_{U_\gamma}\inv$ is homeomorphism, hence open. And since all $U_\gamma$ are disjoints, $\Phi$ is an open inclusion. So $\calH:=\Phi(F\ltimes A)$ is open subgroupoid of $\calG$.

The map $\Phi$ establishes a continuous open bijection between $F\ltimes A$ and $\calH$, and it is also an algebraic morphism. Hence, $F\ltimes A$ is isomorphic to $\calH$, which is proper since $F$ is a finite group.

Claim: there exists an open neighborhood $V_1$ of $z_0$ in $Z$, such that $V_1$ for all $\gamma\notin\calH$, $\gamma V_1\cap V_1=\emptyset$. Consider the following map
\[\alpha:\calG\times_{s,X,\rho} Z\ra  Z\times Z,(\gamma,z)\mapsto (z,\gamma z).\]
It is proper since the action of $\calG$ on $Z$ is proper, and therefore closed. We have a closed subset $C:=(\calG\times_X Z)\setminus (\calH\times_X Z)$ of $\calG\times_X Z$, and $\alpha(C)$ does not contain $(z_0,z_0)$. Therefore, there exists an open neighborhood $V_1$ of $z_0$, such that $V_1\times V_1$ does not intersect $\alpha(C)$. This implies that for all $\gamma\notin\calH$, $\gamma V_1\cap V_1=\emptyset$.

Now let $V_2=V_1\cap \rho\inv(A)\cap U$, let $V_3$ be a relatively compact open neighborhood of $z_0$ such that $z_0\in V_3\subseteq \overline{V_3}\subseteq V_2$, $\overline{V_3}$ is compact. Canonically $\rho\inv(A)$ admits an action of $F$ as, for any $\gamma\in F$,
\[\gamma.z=\Phi(\gamma, \rho(z))z, \forall z\in \rho\inv(A).\]
Notice that $z_0$ is $F$-invariant, and we can take $V=\cap_{\gamma\in F}\gamma.V_3$ as an open neighborhood of $z_0$. Now $V_3$ is relatively compact, so $V$ is also relatively compact, and $\overline{V}\subseteq \overline{V_3}\subseteq U$. We have also $\rho(V)\subseteq\rho(\rho\inv(A))\subseteq A=\calH\units$. By definition $V$ is also $\calH$-invariant. Finally, $\overline{V}\subseteq V_1$, hence for all $\gamma\in \calG\setminus\calH$, $\gamma\overline{V}\cap \overline{V}=\emptyset$.\end{proof}

\subsection{\texorpdfstring{$C_0(X)$-algebras and modules}{C0(X)-algebras and modules}}

We make a summary about $C_0(X)$-algebras and upper-semicontinuous bundles. The usual references are \cite{blanchard1996deformations} and \cite[Appendix C]{williams2007crossed}.  Here our standard reference is \cite{miller2022k}, where almost all details can be found.

\textbf{Banach bundles}: let $X$ be a locally compact Hausdorff space. A Banach bundle over $X$ is a topological space $\calA$ equipped with a continuous open surjection $p: \calA\ra X$, and every fiber $\calA_x:=p\inv(x)$ is equipped with complex Banach space structure, such that
\begin{enumerate}
    \item the map $a\mapsto \|a\|:=\|a\|_{\calA_{p(a)}}$ is upper-semicontinuous from $\calA$ to $\mathbb R$;
    \item the addition $\calA\times_X \calA\ra \calA, (a,b)\mapsto a+b$ is continuous;
    \item for every $\lambda\in \mathbb C$, $\calA\ra \calA, a\mapsto \lambda a$ is continuous;
    \item if $(a_\lambda)_\lambda$ is a net in $\calA$ such that $p(a_\lambda)\ra x$ and $\|a_\lambda\|\ra 0$, then $a_\lambda\ra 0_x$.
\end{enumerate}

\begin{prop}\label{convergence in banach bundle}
    \textnormal{\cite[Proposition 1.38]{miller2022k}}
    Let $p:\ca A\ra X$ be a Banach bundle and let $(a_\lambda)_\lambda$ be a net in $\ca A$ such that $p(a_\lambda)\ra p(a)$ for some $a\in \ca A$. Suppose that for all $\epsilon>0$, there is a net $(u_\lambda)_\lambda$ in $\ca A$ and $u\in \ca A$ with $p(u_\lambda)=p(a_\lambda)$, $p(u)=p(a)$, such that
    \begin{enumerate}
        \item $u_\lambda$ converges to $u$ in $\ca A$;
        \item $\|a-u\|<\epsilon$;
        \item $\|a_\lambda-u_\lambda\|<\epsilon$ for $\lambda$ large enough.
    \end{enumerate}
    Then $a_\lambda \ra a$.
\end{prop}

If $\calA$ is a Banach bundle over $X$, $U$ is a locally closed subset of $X$, the spaces of continuous sections $U\ra p\inv(U)\subseteq \calA$ that are bounded, vanish at infinity or are compactly supported are denoted respectively as $\Gamma_b(U,\calA), \Gamma_0(U,\calA), \Gamma_c(U,\calA)$, equipped with the norm $\|f\|=\sup_{x\in U}\|f(x)\|$. Clearly $\Gamma_c(U,\calA)$ is dense in $\Gamma_0(U,\calA)$.

We define the morphisms between two Banach bundles $\calA, \ca B$ over $X$ as a continuous map $\varphi: \ca A\ra \ca B$ such that for any $x\in X$, $\varphi(\ca A_x)\subseteq \ca B_x$, and every fiber $\varphi_x: \ca A_x\ra \ca B_x$ is a bounded linear map.

\textbf{\cst-bundles}: given a Banach bundle $\calA\ra X$, we say that it is a \cst-bundle if every fiber is equipped with \cst-algebra structure and the multiplication map 
\[ \calA\times_X\calA=\sqcup_{x\in X}(\calA_x\times \calA_x)\ni (a,b)\mapsto ab\in \calA\]
and the involution map
\[\calA=\sqcup_{x\in X}\calA_x\ni a\mapsto a^*\in \calA\]
are continuous. Then $\Gamma_0(X,\ca A)$ is a $C_0(X)$-algebra. The morphisms between two \cst-bundles $\ca A, \ca B$ over $X$ are morphisms of Banach bundles $\varphi:\ca A\ra \ca B$ such that every fiber $\varphi_x$ is a *-homomorphism.

\textbf{$C_0(X)$-algebras}: a $C_0(X)$-algebra is a \cst-algebra with a non-degenerate *-homomorphism $C_0(X)\ra ZM(A)$. A $C_0(X)$-linear *-homomorphism between two $C_0(X)$-algebras $\phi:A\ra B$ is a *-homomorphism $\phi$ such that
\[h\cdot\phi(a)=\phi(h\cdot a),\quad \forall h\in C_0(X), a\in A.\]

If $\calA\ra X$ is a \cst-bundle, $\Gamma_0(X,\ca A)$ is a $C_0(X)$-algebra. 

Conversely, assume that $A$ is a $C_0(X)$-algebra. We define its fiber at $x\in X$ as the \cst-algebra $A_x:=A/\overline{C_0(X\setminus\{x\})A}$. Let $\ca A$ be the set $\sqcup_{x\in X}A_x$. For any $a\in A$, we use $a(x)$ to denote its image in the quotient $A_x$. Then $\ca A$ can be equipped with the weakest topology such that $x\mapsto a(x)$ is continuous for all $a\in A$, and $\ca A$ becomes a \cst-bundle over $X$. We say that $\ca A$ is the \cst-bundle associated with $A$. Moreover, $\Gamma_0(X,\ca A)$ is $C_0(X)$-linearly isomorphic to $A$. (See \cite[Theorem C.26]{williams2007crossed}, or the Dauns--Hofmann representation theorem \cite{Hofmann}.)

\textbf{Hilbert bundles}: for a \cst-bundle $\calA\ra X$, a Hilbert $\calA$-bundle is a Banach bundle $\ca E\ra X$ such that every fiber $\ca E_x$ is equipped with Hilbert $\calA_x$-module structure and the inner product $\ca E\times_X\ca E\ra \ca A$ and $\ca A$-action $\ca E\times_X\calA\ra \ca E$ are continuous. Then $\Gamma_0(X,\ca E)$ is a Hilbert $\Gamma_0(X,\ca A)$-module. The morphisms between two Hilbert $\ca A$-bundles $\ca E, \ca F$ are morphisms of Banach bundles $\varphi: \ca E\ra \ca F$ such that every fiber $\varphi_x$ is an adjointable operator between the two Hilbert $\ca A_x$-modules.

Assume that $A$ is a $C_0(X)$-algebra and $E$ is a Hilbert $A$-module. Since $EA=E$ (see \cite[Lemma 1.3]{blanchard1996deformations}), we can define the action of $C_0(X)$ on $E$ as $f\cdot (ea)=e(f\cdot a)$, where $f\in C_0(X), e\in E$ and $a\in A$. Let $E_x=E/\overline{C_0(X\setminus\{x\})E}$ and for any $e\in E$, $x\in X$, let $e(x)$ be the image of $e$ under the quotient map $E\ra E_x$. Similarly, $\ca E:=\sqcup_{x\in X}E_x$ can be equipped with a topology to become a Hilbert $\ca A$-bundle, and for every $e\in E$, $x\mapsto e(x)$ is continuous. We say that $\ca E$ is the Hilbert $\ca A$-bundle associated to $E$. Similarly, $\Gamma_0(X,\ca E)$ is a Hilbert $A$-module which is isomorphic to $E$.

From now on if we use $A,B,\cdots$ to denote $C_0(X)$-algebras (use $E,F,\cdots$ to denote a Hilbert module over some $C_0(X)$-algebras, respectively), we will let the corresponding calligraphic letter $\ca A,\ca B,\cdots$ ($\ca E,\ca F,\cdots$, respectively) be the associated \cst-bundles (Hilbert bundles, respectively). We have an equivalence in the canonical way between the category of \cst-bundles over $X$ and the category of $C_0(X)$-algebras by a pair of functors
\[\ca A\mapsto \Gamma_0(X,\ca A), \quad A\mapsto \sqcup_{x\in X}A_x.\]
Similarly, for a fixed $C_0(X)$-algebra $A$ and its associated \cst-bundle $\ca A$, we have an equivalence between the category of Hilbert $\ca A$-bundles and the category of Hilbert $A$-modules
\[\ca E\mapsto \Gamma_0(X,\ca E),\quad E\mapsto\sqcup_{x\in X}E_x.\]
If $E,F$ are two Hilbert modules over a $C_0(X)$-algebra $A$, an adjointable operator $T\in \mathcal L(E,F)$ induces an adjointable operator $T_x\in \mathcal L(E_x,F_x)$ for every $x\in X$ by
\[T_x: E_x\ra F_x,T_x(e(x))=(Te)(x)\]
since the map $E\ra F\ra F_x$ factors through $E_x$.

\textbf{Commutative cases}: let $\rho: Y\ra X$ be a continuous map between two locally compact Hausdorff spaces. Let $\Phi: C_0(X)\ra C_b(Y)\cong ZM(C_0(Y))$ be the *-homomorphism defined as $f\mapsto f\circ \rho$. Obviously $\Phi$ gives a $C_0(X)$-algebra structure of $C_0(Y)$. We see that $\overline{\Phi(C_0(X\setminus\{x\}))C_0(Y)}$ equals to $\{f\in C_0(Y): f|_{\rho\inv(x)}=0\}$ as an ideal of $C_0(Y)$. Therefore, the fiber of $C_0(Y)$ at $x$ can be identified with $C_0(Y_x)$, where $Y_x=\rho\inv(x)$. The \cst-bundle associated to $C_0(Y)$ can be set-theoretically defined as $\sqcup_{x\in X}C_0(Y_x)$.

\textbf{Banach bundles of compact operators}: let $A$ be a $C_0(X)$-algebra and let $E,F$ be two Hilbert $A$-modules. The $C_0(X)$-action gives a *-homomorphism $\Phi: C_0(X)\ra Z\calL(E)$ such that $\overline{\Phi(C_0(X))E}=E$. Since we can identify $\calL(E)$ with $M(\calK(E))$, $\Phi$ induces a $C_0(X)$-algebra structure of $\calK(E)$. By \cite[Proposition 1.51]{miller2022k}, fiber of $\calK(E)$ at $x\in X$ can be identified with $\calK_{A_x}(E_x)$. The \cst-bundle associated to $\calK(E)$ can be set-theoretically defined as $\sqcup_{x\in X}\calK(E_x)$.

\textbf{Bundles of adjointable operators}: let $A$ be a $C_0(X)$-algebra and let $E,F$ be Hilbert $A$-modules. In general $\calL(E)$ fails to be a $C_0(X)$-algebra. We define $\ca L(\ca E):=\sqcup_{x\in X}\calL(E_x)$ and the surjective map $p:\calL(\ca E)\ra X$ such that $p(T)=x$ if $T\in \calL(E_x)$. We can define the strict topology on $\calL(\ca E)$ as the weakest topology such that, for any $e\in E$, the following maps
\[\calL(\ca E)\ra \ca E, T \mapsto Te(p(T)),\]
\[\calL(\ca E)\ra \ca E, T\mapsto T^*e(p(T))\]
are continuous. The space of bounded strictly continuous sections $\Gamma_b(X,\calL(\ca E))$ is isomorphic to $\ca L(E)$. (See \cite[Proposition 1.54]{miller2022k}.)

\textbf{Representations}: let $A, B$ be $C_0(X)$-algebras and $E$ be a Hilbert $B$-module, $\pi$ be a *-homomorphism $A\ra \calL(\ca E)$. We say that $\pi$ is non-degenerate, if $\overline{\pi(A)E}=E$. We say that $\pi$ is a $C_0(X)$-representation (or $C_0(X)$-linear), if for any $a\in A$, $e\in E$, $f\in C_0(X)$, $\pi(a)(fe)=\pi(fa)(e)$. This implies that $A\xrightarrow{\pi}\calL_B(\ca E)\xrightarrow{T\mapsto T_x} \calL_{B_x}(E_x)$ factories through $A_x$. Let $\pi_x: A_x\ra \calL_{B_x}(E_x)$ be the induced *-homomorphism. We say that $\pi_x$ is the fiber of $\pi$ at $x$. For any $a\in A$, $e\in E$, $x\in X$, we have $\pi_x(a(x))(e(x))=(\pi(a)e)(x)$.

\textbf{Pullback}: let $\rho: Y\ra X$ be a continuous map between two locally compact Hausdorff spaces, let $\ca A$ be a Banach bundle over $X$. There is a Banach bundle structure on $\rho^*\ca A:=\ca A\times_X Y$ over $Y$. Its fiber $(\rho^*\ca A)_y$ at $y\in Y$ can be identified with $\ca A_{\rho(y)}$. When $\ca A$ is the \cst-bundle associated to a \cst-algebra $A$, we define the pullback of $A$ by $\rho$ as $\rho^*A:=\Gamma_0(X,\rho^*\ca A)$, which is therefore a $C_0(Y)$-algebra. Canonically, $\rho^*A$ is isomorphic to the tensor product of $C_0(X)$-algebras $C_0(Y)\otimes_{X}A$ (as in definition 2.1.5 of \cite{legal1994}) as $C_0(Y)$-algebras.

When $E$ is a Hilbert $A$-module, we define the pullback of $E$ by $\rho$ as $\rho^*E=\Gamma_0(X,\rho^*\ca E)$, which is a Hilbert $\rho^*A$-module. We have identifications
\[\calK(\rho^*E)\cong \rho^*(\calK(E)), \calL(\rho^*E)\cong \Gamma_b(Y,\calL(\rho^*\ca E)).\]
When $\pi: A\ra \calL_B(E)$ is a $C_0(X)$-representation, we can define the pullback of $\pi$ by $\rho$ as, for any $a\in \rho^*A$ and $e\in \rho^*E$, $\rho^*\pi(a)e\in \rho^*E=\Gamma_0(Y,f^*\ca E)$ is a section defined by
\[[\rho^*\pi(a)e](y)=\pi_{\rho(y)}(a(y))(e(y))\in E_{\rho(y)}.\]
And one can check that $\rho^*\pi: \rho^*A\ra \calL_{\rho^*B}(\rho^*E)$ is a well-defined $C_0(Y)$-representation. Its fiber $(\rho^*\pi)_y$ at $y\in Y$ can be identified with $\pi_{\rho(y)}$. If $\pi$ is non-degenerate then so is $\rho^*\pi$ (we can obtain this from \cite[Proposition 1.94]{miller2022k}).

Especially, for a $C_0(X)$-linear *-homomorphism $\phi:A\ra B$, 
\[\rho^*\phi:\rho^*A\ra \rho^*B, [(\rho^*\phi)a](y)=\phi_y(a(\rho(y)))\]
is a $C_0(Y)$-linear *-homomorphism.

\begin{lem}
    A $C_0(X)$-representation $\pi$ is non-degenerate if and only if for any $x\in X$, $\pi_x$ is non-degenerate.
\end{lem}
\begin{proof} Suppose that for any $x\in X$, $\pi_x$ is non-degenerate. Let $\Gamma=\pi(A)(E)\subseteq E=\Gamma_0(X,\mathcal E)$, where $\mathcal E=\sqcup E_x$ is the associated Hilbert bundle. For any $a\in A, e\in E, \phi\in C_0(X)$, since it is a $C_0(X)$-representation, $(\pi(a)e)\phi=\pi(a)(e\phi)=\pi(a\phi)e\in \pi(A)(E)$. So $\Gamma$ is stable under action of $C_0(X)$. And for any $x\in X$, $\pi_x$ is non-degenerate, so $\{\gamma_x:\gamma\in \Gamma\}=\pi_x(A_x)(E_x)$ has dense span in $E_x$. By proposition 2.12 of \cite{miller2024functors}, $\Gamma$ has dense span in $E$. So $\pi$ is non-degenerate.

Conversely, any pullback of a non-degenerate representations is again non-degenerate, $\pi_x$ can be seen as the pullback of $\pi$ by the inclusion of a singleton. \end{proof}

\textbf{Pushout}: let $g: Z\ra X$ be a continuous map between two locally compact Hausdorff spaces. Let $A$ be a $C_0(Z)$-algebra, $E$ be a Hilbert $A$-module and $\Phi$ be the structure map $\Phi: C_0(Z)\ra ZM(A)$. Since $\Phi$ is non-degenerate, it has a strictly continuous extension $\tilde\Phi: C_b(Z)\ra ZM(A)$. We define the pushout of $A$ by $g$ as the same \cst-algebra $A$ but equipped with $C_0(X)$-algebra structure
\[C_0(X)\xrightarrow{g^*}C_b(Z)\xrightarrow{\tilde\Phi} ZM(A).\]
The \cst-algebra $A$ is written as $g_*A$ to be seen as a $C_0(X)$-algebra. We define the pushout $g_*E$ as the same Hilbert $g_*A$-module. We have canonical identifications $(g_*A)_x=\Gamma_0(Z_x,\ca A)$ and $(g_*E)_x=\Gamma_0(Z_x,\ca E)$. (See \cite[Proposition 3.5]{bonicke2020going}).

The following lemma will be useful.

\begin{lem}\label{sum of fiber}
    \textnormal{\cite[Lemma 2.19]{miller2024functors}}
    Let $\rho:Y\ra X$ be a local homeomorphism between locally compact Hausdorff spaces and $\ca A\ra X$ be a Banach bundle over $X$. Then for any $\xi\in \Gamma_b(Y,\rho^*\calA)$ such that $\rho|_{supp(\xi)}$ is proper,
    \[\rho_*\xi:x\mapsto \sum_{y\in Y_x}\xi(y)\]
    is well-defined and continuous bounded section of $\ca A\ra X$. If $\xi$ is compactly supported then so is $\rho_*\xi$.
\end{lem}

\subsection{Groupoid \texorpdfstring{\cst}{C*}-algebras and Hilbert modules}

Let $\calG\rightrightarrows X$ be a locally compact Hausdorff groupoid, with source map $s$ and range map $r$. 

\textbf{Banach $\calG$-bundles}: a Banach $\calG$-bundle is a Banach bundle $p: \ca A\ra X$ equipped with an action of $\calG$ on the topological space $\ca A$ with anchor map $p$, such that for each $\gamma\in \calG$,
\[\ca A_{s(\gamma)}\ra \ca A_{r(\gamma)},\quad a\mapsto g.a\]
is an isomorphism of Banach spaces.

\textbf{$\calG$-\cst-bundles}: if $\ca A$ is a \cst-bundle and a Banach $\calG$-bundle, we say that $\ca A$ is a $\calG$-\cst-bundle. 

\textbf{$\calG$-\cst-algebras}: we say $(A,\calG,\alpha)$ is a groupoid dynamic system or $A$ is a $\calG$-\cst-algebra (or $(A,\alpha)$ is a $\calG$-\cst-algebra), if $A$ is a $C_0(X)$-algebra and $\alpha: s^* A\ra r^*A$ is an isomorphism of $C_0(\calG)$-algebras such that for any $(\gamma_1,\gamma_2)\in \calG^{(2)}$, $\alpha_{\gamma_1}\circ\alpha_{\gamma_2}=\alpha_{\gamma_1\gamma_2}:A_{s(\gamma_2)}\ra A_{r(\gamma_1)}$. The isomorphism $\alpha$ induces an isomorphism of two \cst-bundles over $\calG$
\[\sqcup_{\gamma\in \calG}\alpha_\gamma: \calG\times_{s,\calG\units} \ca A=\sqcup_{\gamma\in \calG}A_{s(\gamma)}\ra \sqcup_{\gamma\in \calG}A_{r(\gamma)}=\calG\times_{r,\calG\units}\ca A,\]
after composing with the projection $\calG\times_{r,\calG\units}\ca A\ra \calA$, we have a continuous map
\[\calG\times_{s,\calG\units}\calA\ra\calA,\quad (\gamma,a_{s(\gamma)})\mapsto \alpha_\gamma (a_{s(\gamma)}).\]
This map defines a left action of $\calG$ on $\calA$. Hence, the associated \cst-bundle $\ca A$ is a $\calG$-\cst-bundle.

\begin{lem}\label{continuity of groupoid action}
    \textnormal{\cite[Lemma 3.9]{bonicke2020going}}
    Let $\calG\rightrightarrows X$ be a locally compact Hausdorff groupoid, $A$ be a $C_0(X)$-algebra, $\alpha=(\alpha_\gamma)_{\gamma\in \calG}$ be a family of *-isomorphisms $\alpha_\gamma: A_{s(\gamma)}\ra A_{r(\gamma)}$ such that $\alpha_{\gamma_1\gamma_2}=\alpha_{\gamma_1}\circ \alpha_{\gamma_2}$ for every $(\gamma_1,\gamma_2)\in \calG^{(2)}$. Then $(A,\calG,\alpha)$ is a groupoid dynamic system if and only if for every $a\in A$, $\gamma\mapsto \alpha_{\gamma}(a(s(\gamma)))
    $ is a continuous section $\calG \ra r^*\calA$.
\end{lem}

\textbf{$\calG$-Hilbert $\ca A$-bundle}: let $\ca A$ be a $\calG$-\cst-bundle and $p:\ca E\ra X$ be a Hilbert $\ca A$-bundle. If $\ca E$ is also a Banach $\calG$-bundle and the action of $\calG$ is compatible, that is for each $\gamma\in \calG$, $a\in \ca A_{s(\gamma)}$, $e_1,e_2\in \ca E_{s(\gamma)}$,
\[\gamma(e_1a)=(\gamma e_1)(\gamma a), \quad \gamma\langle e_1,e_2\rangle=\langle \gamma e_1,\gamma e_2\rangle,\]
we say that $\ca E$ is a $\calG$-Hilbert $\ca A$-bundle.

Let $(A,\alpha)$ be a $\calG$-\cst-algebra. A $\calG$-Hilbert $A$-module is a Hilbert $C_0(X)$-module $E$ and a unitary $V\in \calL_{s^*B}(s^*E,r^*E)$, where the Hilbert $s^*A$-module structure of $r^*E$ is given by identifying $s^*A$ and $r^*A$ through $\alpha$, such that for any $(\gamma_1,\gamma_2)\in \calG^{(2)}$, we have $V_{\gamma_1}\circ V_{\gamma_2}=V_{\gamma_1\gamma_2}$. Similarly, its associated Hilbert $\ca A$-bundle $\ca E$ is a $\calG$-Hilbert $\ca A$-bundle.
\begin{lem}\label{continuity of groupoid action of Hilbert module}
    Let $\calG\rightrightarrows X$ be a locally compact Hausdorff groupoid, $(A,\alpha)$ be a $\calG$-\cst-algebra and $E$ be a Hilbert $A$-module. Suppose that $(V_\gamma)_{\gamma\in \calG}$ be a family of unitary isomorphisms $V_\gamma\in \calL(E_{s(\gamma)},E_{r(\gamma)})$, such that $V_{\gamma_1\gamma_2}=V_{\gamma_1}\circ V_{\gamma_2}$ for every $(\gamma_1,\gamma_2)\in \calG^{(2)}$. Then this action of $\calG$ makes $E$  be a $\calG$-Hilbert $A$-module if and only if for every $e\in E$, $\gamma\mapsto V_\gamma(e(s(\gamma)))$ is a continuous section $\calG\ra r^*\ca E$.
\end{lem}

\textbf{$\calG$-equivariant representations}: for two $\calG$-\cst-algebras $(A,\alpha),(B,\beta)$ and a $\calG$-Hilbert $B$-module $(E,V)$, a *-homomorphism $\pi: A\ra \calL(E)$ is called $\calG$-equivariant, if it is $C_0(\calG\units)$-linear and for any $\gamma\in \calG$ and $a\in A_{s(\gamma)}$,
\[\pi_{r(\gamma)}(\alpha_\gamma(a))=V_{\gamma}\circ \pi_{s(\gamma)}(a)\circ V_{\gamma\inv}.\]

\textbf{Quotient Banach bundles}: let $\calG\rightrightarrows X$ be a proper principal \'etale groupoid. If $\ca A$ is a right Banach $\calG$-bundle, then the quotient space $\ca A/\calG$ still have a structure of a Banach bundle over $X/\calG$. (See \cite[Definition 1.85, Proposition 1.86]{miller2022k}). Moreover, if $\ca A$ is a $\calG$-\cst-bundle, $\ca A/\calG$ is a \cst-bundle. If $\ca E$ is a $\calG$-Hilbert $\ca A$-bundle, then $\ca E/\calG$ is a Hilbert $\ca A/\calG$-bundle.

Let $(A,\alpha)$ be $\calG$-\cst-algebra.
Let $q: X\ra X/\calG, x\mapsto [x]$ be the quotient map. The associated \cst-bundle $\ca A$ is a \cst-bundle over $X$ and $\calA/\calG$ is a \cst-bundle over $X/\calG$. The fiber of $A/\calG\ra X/\calG$ at $[x]$ can be identified with $A_x$. We define $A^{\calG}$ as $\Gamma_0(X/\calG,\calA/\calG)$, which can be identified with the set of bounded continuous sections $\xi\in \Gamma_b(X,\calA)$ such that
\begin{enumerate}
    \item for any $\gamma\in \calG$, $\alpha_\gamma(\xi(s(\gamma))=\xi(r(\gamma))$;
    \item the map $X/\calG\ra\mathbb R:[x]\mapsto \|\xi(x)\|_{A_x}$ vanishes at infinity.
\end{enumerate}
Similarly, if $(E,V)$ is a $\calG$-Hilbert $A$-module, we define
$E^{\calG}=\Gamma_0(X/\calG,\ca E/\calG)$, it is a Hilbert $A^{\calG}$-module that can be identified with the set of bounded continuous sections $\eta\in \Gamma_b(X,\ca E)$ that
\begin{enumerate}
    \item for any $\gamma\in \calG$, $V_\gamma(\eta(s(\gamma))=\eta(r(\gamma))$;
    \item the map $X/\calG\ra \mathbb R: [x]\mapsto \|\eta(x)\|_{E_x}$ vanishes at infinity.
\end{enumerate}

Let $(A,\alpha),(B,\beta)$ be two $\calG$-\cst-algebras and let $(E,V)$ a $\calG$-Hilbert $B$-module, let $\pi: A\ra \calL(E)$ be a $\calG$-equivariant representation. We can define
\[\pi^\calG:A^\calG\ra \calL(E^\calG)\]
as for any $\xi\in A^\calG$, $\eta\in E^\calG$, $x\in X$,
\[(\pi^\calG(\xi)\eta)(x)=\pi_x(\xi(x))\eta(x).\]
In this way $\pi^\calG_{[x]}$ can be identified with $\pi_x$. Hence, if $\pi$ is non-degenerate, so is $\pi^\calG$.

\begin{prop}\label{invariant algebra keeps sigma unity}
    Let $\calG\rightrightarrows X$ be a proper principle \'etale groupoid such that $X/\calG$ is $\sigma$-compact, $(A,\alpha)$ be a $\calG$-\cst-algebra. If $A$ is $\sigma$-unital, then so is $A^\calG$.
\end{prop}
\begin{proof}
    By proposition \ref{cuttoff function exist}, the groupoid $\calG$ admits a cutoff function $c$. Let $(a_n)_{n\in \mathbb N_+}$ be a countable approximate unit of $A$, $\|a_n\|_A\leqslant 1$ for all $n\in \mathbb N_+$. 
    We define $\xi_n: X\ra \calA$ as
    \[\xi_n(x)=\sum_{\gamma\in \calG_x}c(r(\gamma))\alpha_{\gamma\inv}(a_n(r(\gamma))).\]
    By lemma \ref{sum of fiber}, $\xi_n$ is well-defined continuous section. And for any $x\in X$,
    \[\|\xi_n(x)\|_{A_x}\leqslant \sum_{\gamma\in \calG_x}c(r(\gamma))\|a_n(r(\gamma))\|_{A_{r(\gamma)}}\leqslant \sum_{\gamma\in \calG_x}c(r(\gamma))=1.\]
    So $\xi_n\in \Gamma_b(X,\calA)$. And any $\gamma\in \calG$,
    \begin{align*}
        \xi_n(r(\gamma)) & =\sum_{h\in \calG_{r(\gamma)}}c(r(h))\alpha_{h\inv}(a_n(r(h)))\\
        & = \sum_{h\in \calG_{r(\gamma)}} c(r(h\gamma))\alpha_\gamma[\alpha_{(h\gamma)\inv}(a_n(r(h\gamma)))]\\
        & = \alpha_\gamma(\xi_n(s(\gamma))).
    \end{align*}
    Since $X/\calG$ is locally compact Hausdorff and $\sigma$-compact, there exists a sequence of compact subsets $(K_n)_{n\in \mathbb N_+}$, such that $X/\calG=\cup_{n\in \mathbb N_+}K_n$ and for any $n\in \mathbb N_+$, $K_n\subseteq \inter({K_{n+1}})$. For each $n\in \mathbb N_+$, we can select $\phi_n\in C_c(X/\calG,[0,1])$ such that $\phi_n|_{K_n}=1$ and $supp(\phi_n)\subseteq \inter({K_{n+1}})$. We define $\xi'_n\in \Gamma_b(X,\ca A)$ as $\xi'_n(x)=\phi_n([x])\xi_n(x)$. It is easy to see that for each $n\in \mathbb N_+$, $\xi'_n$ is a well-defined element of $A^\calG$. We will show that $(\xi_n')_n$ is a countable approximate unit.

    If $\eta\in \Gamma_c(X/\calG,\calA/\calG)$, that is $\eta$ can be seen as an element of $\Gamma_b(X,\calA)$ such that
    \begin{enumerate}
       \item for any $\gamma\in \calG$, $\alpha_\gamma(\eta(s(\gamma))=\eta(r(\gamma))$;
    \item the map $X/\calG\ra\mathbb R:[x]\mapsto \|\xi(x)\|_{A_x}$ is compactly supported.
    \end{enumerate}
    Then there exists $n_0\in \mathbb N_+$ such that $supp(\eta)\subseteq K_{n_0}$. Then for any $n\geqslant n_0$, $\xi_n'\eta=\xi_n\eta$, $\eta\xi'_n=\eta\xi_n$.
    
    By lemma \ref{G-compact iff quotient compact}, there exists a compact subset $K$ of $q\inv(supp(\eta))$ such that $q\inv(supp(\eta))=\calG K$. Let $K'=\{\gamma\in \calG_K:\gamma\in supp(c\circ r)\}$. By the definition of cutoff function $K'$ is a compact subset of $\calG$. So there exists finitely many open bisections $V_1,\cdots, V_m$ and $W_1,\cdots, W_m$ of $\calG$, such that
    \begin{enumerate}
        \item $K'\subseteq \cup_{j=1}^m V_j$;
        \item for every $1\leqslant j\leqslant m$, $V_j\subseteq \overline{V_j}\subseteq W_j$, $\overline{V_j}$ is compact.
    \end{enumerate}
    By Urysohn's lemma, for each $1\leqslant j\leqslant m$, there exists $f_j\in C_c(X)$ such that $f_j|_{\overline{r(V_j)}}=1$ and $f_j$ is supported in $r(W_j)$. For each $1\leqslant j\leqslant m$, $x'\mapsto f_j(x')\alpha_{r|_{W_j}\inv(x')}(\eta(s(r|_{W_j}\inv(x')))$ is an element of $\Gamma_0(r(W_j),\calA)$, which can be extended to an element of $\Gamma_0(X,\calA)=A$ through extension by zero. Therefore, since $(a_n)_{n\in \mathbb N_+}$ is approximate unit of $A$,
    \begin{align*}
        & \lim_{n\ra\infty}\sup_{\gamma\in V_j}\|\alpha_{\gamma\inv}(a_n(r(\gamma)))\eta(x)-\eta(x)\|_{A_{s(\gamma)}}\\
        & =\lim_{n\ra\infty}\sup_{x'\in r(V_j)}\|a_n(x')\alpha_{r|_{V_j}\inv(x')}(\eta(s(r|_{V_j}\inv(x')))-\alpha_{r|_{V_j}\inv(x')}(\eta(s(r|_{V_j}\inv(x')))\|=0.
    \end{align*}

    Therefore, by the following inequality,
    \begin{align*}
        \|\xi_n\eta-\eta\|_{A^\calG} & = \sup_{[x]\in X/\calG}\|\xi_n(x)\eta(x)-\eta(x)\|_{A_x}\\
        & = \sup_{x\in K}\|\xi_n(x)\eta(x)-\eta(x)\|_{A_x}\\
        & = \sup_{x\in K}\|\sum_{\gamma\in \calG_x}c(r(\gamma))(\alpha_{\gamma\inv}(a_n(r(\gamma)))\eta(x)-\eta(x))\|_{A_x}\\
        & = \sup_{x\in K}\|\sum_{\gamma\in K'_x}c(r(\gamma))(\alpha_{\gamma\inv}(a_n(r(\gamma)))\eta(x)-\eta(x))\|_{A_x}\\
        & \leqslant \sum_{j=1}^m \sup_{\gamma\in V_j}\|\alpha_{\gamma\inv}(a_n(r(\gamma)))\eta(x)-\eta(x)\|_{A_{s(\gamma)}},
    \end{align*}
    we have $\|\xi_n\eta-\eta\|_{A^\calG}\ra 0$. Similarly, we have also $\|\eta\xi_n-\eta\|_{A^\calG}\ra 0$. Therefore, $\lim_{n\ra\infty}\xi_n'\eta=\lim_{n\ra\infty}\xi_n\eta=\eta=\lim_{n\ra\infty}\eta\xi_n=\lim_{n\ra\infty}\eta\xi_n'$.

    Moreover, $\Gamma_c(X/\calG,\calA/\calG)$ is dense in $A^\calG$, so $(\xi_n')_{n\in \mathbb N_+}$ is a countable approximate unit of $A^\calG$.
\end{proof}

\subsection{\'Etale groupoid correspondences}

Let $\calG,\calH$ be two \'etale groupoids. Recall that, a $\calG,\calH$-bibundle $\Omega$ is a topological space $\Omega$, equipped with a left action of $\calG$ on $\Omega$ with anchor map $\rho_\Omega:\Omega\ra \calG\units$ and a right action of $\calH$ on $\Omega$ with anchor map $\sigma_\Omega: \Omega\ra \calH\units$, such that the $\calG$-action commutes with the $\calH$-action. We say that $\Omega:\calG\leftarrow \calH$ is a correspondence, if the right $\calH$-action is free proper and \'etale. Details can be found in \cite{antunes2021bicategory}. For convenience of writing, we write a correspondence as an arrow from the right to the left. \textbf{Throughout this paper all correspondences are assumed to be locally compact Hausdorff}.

For convenience of writing, for $x\in \calG\units$ and $y\in \calH\units$, we write $\Omega^x=\rho_\Omega\inv(x)$ and $\Omega_y=\sigma_\Omega\inv(y)$. 

\begin{ex}\label{Correspondence associated to a strict morphism}
We say that a strict morphism $f:\calG\ra \calH$ is \'etale if $f$ is a local homeomorphism. For an \'etale strict morphism $f:\calG\ra \calH$ between two \'etale groupoids, the space
\[\Omega_f:=\calG\units\times_{f\units,\calH\units,r_{\calH} }\calH\]
admits a left action by $\calG$ defined as $\gamma.(s_\calG(\gamma),h)=(r_\calG(\gamma), f(\gamma)h)$, whose  anchor map is
$\rho_{\Omega_f}:(x,h)\mapsto x$, 
and a right action by $\calH$ defined as $(x,h).h'=(x,hh')$, whose anchor map is
$\sigma_{\Omega_f}:(x,h)\mapsto s_{\calH}(h)$. 
It is easy to check that $\Omega_f$ is a $\calG,\calH$-bibundle. If $(x,h).h'=(x,h)$, $hh'=h$, this implies that $h'\in \calH\units$, so the $\calH$-action is free. The anchor map of the $\calH$-action is a local homeomorphism. And $\calH$ acts properly on itself, then by proposition \ref{groupoid action-proper action base change}, $\Omega_f$ is proper $\calH$-space. Hence, $\Omega_f$ is a correspondence $\calG\leftarrow \calH$.

\end{ex}

Another example comes from relatively clopen subgroupoids, which were firstly introduced in \cite{oyono2023groupoids}.

\begin{defn}
    Let $\calG$ be a locally compact Hausdorff groupoid. A relatively clopen subgroupoid of $\calG$ is an open subgroupoid $\calH\subseteq \calG$ such that $\calH$ is closed in $\calG_{\calH\units}$.
\end{defn}

\begin{lem}\label{relative clopen grpd has natural correspondece}
    Let $\calG$ be an \'etale groupoid and $\calH\subseteq \calG$ an open subgroupoid. Take $\Omega=\calG_{\calH\units}$, equipped with the left action of $\calG$ and the right action of $\calH$ by multiplication of groupoid. Then $\Omega$ is a correspondence $\calG\leftarrow \calH$ if and only of $\calH$ is relatively clopen.
\end{lem}
\begin{proof} If $\omega\in \Omega$, $h\in \calH$ such that $\omega h=\omega$, then $h=s(\omega)\in \calH\units$, hence the action of $\calH$ on $\Omega$ is always free. The two actions are given by multiplication of $\calG$, hence they always commute. The anchor map for $\calH$ action $s|_{\calG_{\calH\units}}:\calG_{\calH\units}\ra \calH\units$ is a local homeomorphism. Therefore, $\Omega$ is a well-defined correspondence if and only if $\Omega$ is a proper $\calH$-space.

If $\calH$ is clopen in $\Omega$, for any compact subset $K\subseteq \Omega$, 
\[\{h\in \calH:Kh\cap K\neq \emptyset\}=K\inv K \cap\calH\]
is also a compact subset in $\calH$, so by proposition \ref{groupoid action-proper action}, $\Omega$ is proper $\calH$-space.

Conversely, if $\Omega$ is proper $\calH$-space, for any compact $K\subseteq \Omega$, let $K_1=K\cup r(K)$, which is again a compact subset of $\Omega$. The properness of $\calH$-action implies that
\[\{h\in \calH:K_1h\cap K_1\neq \emptyset\}=K_1\inv K_1\cap \calH\]
is compact subset of $\calH$. Obviously $K\subseteq K_1\inv K_1$, and hence $K\cap \calH$ is closed subset of $\calH$ contained in $K_1\inv K_1\cap \calH$. We conclude that $K\cap \calH$ is compact. So we proved that the open inclusion $\calH\hookrightarrow\Omega$ is proper, that is $\calH$ is clopen in $\Omega$.\end{proof}

\begin{corr}\label{proper subgroupoid is rel clopen}
    A proper open subgroupoid is relatively clopen.
\end{corr}

Suppose that $\Omega:\calG\leftarrow\calH$ is a correspondence and let $Z$ be a left $\calH$-space with anchor map $\rho_Z$. The space $\Omega\times_{\calH\units}Z$ has a right action by $\calH$ defined as $(\omega,z).h=(\omega h,h\inv z)$ whenever $\sigma_\Omega(\omega)=r_\calH(h)=\rho_Z(z)$. We define the balance product
\[\Omega\times_\calH Z:=(\Omega\times_{\calH\units}Z)/\calH\]
to be the orbit space and let $[\omega,z]$ be the image of $(\omega,z)$ under the canonical quotient map. The balance product $\Omega\times_\calH Z$ admits a left $\calG$-action defined by $\gamma[\omega,z]=[\gamma\omega,z]$ whenever $s_\calG(\gamma)=\rho_\Omega(\omega)$ and $\sigma_\Omega(\omega)=\rho_Z(z)$. When $\Omega$ and $Z$ are locally compact Hausdorff and second countable, so is $\Omega\times_\calH Z$.

\begin{defn}
    If $\Omega:\calG\leftarrow \calH$ and $\Lambda:\calH\leftarrow \ca K$ are groupoid correspondences, we define the composition correspondence as the balance product $\Omega\circ\Lambda:=\Omega\times_\calH\Lambda$, with left action of $\calG$ and right action of $\calK$ defined as following: let $\rho_{\Omega\circ \Lambda}$ be the map $[\omega,\lambda]\mapsto \rho_\Omega(\omega)$, and let $\sigma_{\Omega\circ \Lambda}$ be the map $[\omega,\lambda]\mapsto \sigma_\Lambda(\lambda)$,
let the left $\calG$-action on $\Omega\circ \Lambda$ be
\[\calG\times_{r_\calG,\rho_{\Omega\circ \Lambda}}(\Omega\circ \Lambda)\ra \Omega\circ \Lambda, (\gamma, [\omega,\lambda])\mapsto [\gamma\omega,\lambda]\]
and let the right $\ca K$-action on $\Omega\circ \Lambda$ be
\[(\Omega\circ \Lambda)\times_{\sigma_{\Omega\circ \Lambda,s_{\ca K}}}\ca K\ra \Omega\circ \Lambda, ([\omega,\lambda],k)\mapsto [\omega,\lambda k].\]
\end{defn}
We can check that $\Omega\circ \Lambda$ is a well-defined correspondence $\calG\leftarrow \ca K$. (See \cite[section 6]{antunes2021bicategory}.)
\begin{rem}
    The composition $\Omega\circ \Lambda$ here is same as $\Omega\circ_{\calH} \Lambda$ in \cite{antunes2021bicategory}, or $\Lambda\circ \Omega$ in \cite{miller2022k} and \cite{miller2024functors}.
\end{rem}

Morphisms between correspondences are bi-equivariant continuous maps. Given three \'etale groupoids $\calG_1,\calG_2,\calG_3$, let $\Omega_1,\Omega_1'$ be correspondences $\calG_1\leftarrow \calG_2$, $\Omega_2,\Omega_2'$ be correspondences $\calG_2\leftarrow \calG_3$. If $f:\Omega_1\ra \Omega_1'$ is a $\calG_1,\calG_2$-equivariant continuous map, $g:\Omega_2\ra \Omega_2'$ is a $\calG_2,\calG_3$-equivariant continuous map, we define the map $[f,g]$ as
\[[f,g]:\Omega_1\circ \Omega_1'\ra \Omega_2\circ \Omega_2', [\omega_1,\omega_2]\mapsto [f(\omega_1),g(\omega_2)],\]
which is a well-defined $\calG_1,\calG_3$-equivariant continuous map.

\subsection{Induction}

We make a brief introduction to the  induction functor associated to a correspondence which is defined in \cite{miller2022k} and \cite{miller2024functors} by Miller. Suppose that $\Omega:\calG\leftarrow\calH$ is a correspondence. Let $\bar \rho$ be the map $\Omega/\calH\ra \calG\units, \omega\calH\mapsto \rho_\Omega(\omega)$ (which is well-defined because the $\calG$-action and the $\calH$-action commutes). For convenience of writing, let $\rho=\rho_\Omega$ and $\sigma=\sigma_{\Omega}$. Let $s,r$ be the source and range maps of $\calG$.

\begin{defn}
    Let $(A,\alpha)$ be an $\calH$-\cst-algebra. $\ind_\Omega A$ is defined as the sub *-algebra of $\Gamma_b(\Omega,\sigma^* \ca A)$ consisting of continuous sections $\xi:\Omega\ra \sigma^*\calA$ such that
\begin{enumerate}
    \item for any $(\omega,h)\in \Omega\times_{\sigma,\calH\units, r_\calH}\calH$, $\xi(\omega h)=\alpha_{h\inv}(\xi(\omega))$,
    \item the map $\Omega/\calH\ra \mathbb R: \omega\calH\mapsto \|\xi(\omega)\|$ vanishes at infinity.
\end{enumerate}
For $\xi\in \ind_\Omega A$, we define $supp(\xi)=\overline{\{\omega\calH\in \Omega/\calH:\xi(\omega)\neq 0\}}$.
\end{defn}

The induced \cst-algebra $\ind_\Omega A$ can be canonically identified with $\bar \rho_*((\sigma^*A)^{\Omega\rtimes \calH})$, which gives the $C_0(\calG\units)$-algebra of $\ind_\Omega A$ by the map
\[C_0(\calG\units)\ra ZM(\ind_\Omega A), (f\cdot \xi)(\omega)=f(\rho(\omega))\xi(\omega), \quad \forall f\in C_0(\calG\units), \xi\in \ind_\Omega A, \omega\in \Omega.\]
The fiber of $\ind_\Omega A$ at $x\in \calG\units$ can be identified with the sub *-algebra $\ind_{\Omega^x} A$ of $\Gamma_b(\Omega^x,\sigma^* \ca E)$ consisting of continuous sections $\xi:\Omega^x\ra \sigma^*\ca A$ such that
\begin{enumerate}
    \item for any $(\omega,h)\in \Omega^x\times_{\sigma,\calH\units, r_\calH}\calH$, $\xi(\omega h)=\alpha_{h\inv}(\xi(\omega))$,
    \item the map $\Omega^x/\calH\ra \mathbb R:\omega\calH\mapsto \|\xi(\omega)\|$ vanishes at infinity.
\end{enumerate}
For every $\gamma\in\calG$, let $\beta_\gamma$ be the isomorphism
\[\beta_\gamma:\ind_{\Omega^{s(\gamma)}}A\ra \ind_{\Omega^{r(\gamma)}}A, \xi\mapsto \xi(\gamma\inv-).\]
We can prove that $\beta=(\beta_\gamma)_{\gamma\in \calG}$ makes $\ind_\Omega A$ a $\calG$-\cst-algebra.

\begin{defn}
    Let $(E,V)$ be an $\calH$-Hilbert $A$-module. $\ind_\Omega E$ is defined as the closed subspace of $\Gamma_b(\Omega,\sigma^*\ca E)$ consisting of continuous sections $\eta:\Omega\ra \sigma^* \ca E$ such that
    \begin{enumerate}
    \item for any $(\omega,h)\in \Omega\times_{\sigma,\calH\units, r_\calH}\calH$, $\eta(\omega h)=V_{h\inv}(\eta(\omega))$,
    \item the map $\Omega/\calH\ra \mathbb R: \omega\calH\mapsto \|\eta(\omega)\|$ vanishes at infinity.
\end{enumerate}
The Banach space $\ind_\Omega E$ admits a Hilbert $\ind_\Omega A$-module structure as,
\[(\eta\xi)(\omega)=\eta(\omega)\xi(\omega),\quad\forall\eta\in \ind_\Omega E,\xi\in \ind_\Omega A, \omega\in\Omega,\]
\[\langle \eta_1,\eta_2\rangle(\omega)=\langle \eta_1(\omega),\eta_2(\omega)\rangle_{E_{\sigma(\omega)}},\quad \forall \xi_1,\xi_2\in \ind_\Omega E, \omega\in \Omega.\]
For $\eta\in \ind_\Omega E$, we define $supp(\eta)=\overline{\{\omega\calH\in \Omega/\calH:\eta(\omega)\neq 0\}}$.
\end{defn}

For $x\in \calG\units$, we can similarly define $\ind_{\Omega^x}E$. The fiber of $\ind_\Omega E$ at $x$ can be identified with $\ind_{\Omega^x}E$ and the $\calG$-action on $\ind_\Omega E$ is given by the unitary isomorphisms
\[W_\gamma:\ind_{\Omega^{s(\gamma)}}E\ra \ind_{\Omega^{r(\gamma)}}E, \eta\mapsto \eta(\gamma\inv-).\]
We can prove that $W=(W_\gamma)_{\gamma\in \calG}$  makes $\ind_\Omega E$ a $\calG$-Hilbert $\ind_\Omega A$-module. See details in the section 2 of \cite{miller2022k}.

\begin{ex}
    Let $f:\calG\ra \calH$ be an \'etale strict morphism between two \'etale groupoids, $\Omega_f: \calG\leftarrow \calH$ be the associated correspondence as defined in example \ref{Correspondence associated to a strict morphism}, $(A,\alpha)$ be an $\calH$-\cst-algebra. We construct such a *-homomorphism
    \[\psi:f^*A=\Gamma_0(\calG\units,({f\units})^*\ca A) \ra \ind_{\Omega_f}A,\]
    that for all $\xi\in f^*A$, $(x,h)\in \Omega_f$,
    \[\psi(\xi)(x,h)=\alpha_{h\inv}(\xi(x)).\]
    It is easy to check that $\psi$ is a $\calG$-equivariant isomorphism. Hence, as $\calG$-\cst-algebra, $\ind_{\Omega_f}A$ is canonically isomorphic to $f^*A$ (defined as in section 3.1.3 of \cite{legal1994}).
\end{ex}
\begin{ex}\label{subgroupoid induction functor}
    When $\calH$ is a relatively clopen subgroupoid of $\calG$, $\calG_{\calH\units}:\calG\leftarrow\calH$ is a well-defined correspondence by lemma \ref{relative clopen grpd has natural correspondece}. Then $\ind_{\calG_{\calH\units}}$ coincides with B\"onicke's subgroupoid induction functor $\ind^{\calG_{\calH\units}}_{\calH}$ in section 3 of \cite{bonicke2020going}.
\end{ex}
\begin{ex}[Commutative case]
    Let $\Omega:\calG\leftarrow \calH$ be a correspondence, $Z$ be a locally compact Hausdorff left $\calH$-space, $\ind_\Omega C_0(Z)$ can be identified with $C_0(\Omega\times_\calH Z)$ by isomorphism
    \[C_0(\Omega\times_{\calH}Z)\ra \ind_\Omega C_0(Z), f\mapsto [\omega\mapsto f([\omega,-])\in C_0(Z_{\sigma(\omega)})].\]
\end{ex}

When $A,B$ are $\calH$-\cst-algebras, $E$ is an $\calH$-Hilbert $B$-module, $\pi:A\ra \calL(E)$ is an $\calH$-equivariant representation, we define $\ind_\Omega\pi:\ind_\Omega A\ra \calL(\ind_\Omega E)$ such that for every $\xi\in \ind_\Omega A$, $\eta\in \ind_\Omega E$, $\omega\in \Omega$,
\[((\ind_\Omega\pi)(\xi)\eta)(\omega)=\pi_{\sigma(\omega)}(\xi(\omega))\eta(\omega).\]
We can check that $\ind_\Omega \pi$ is a well-define $\calG$-equivariant representation. If $\pi$ is non-degenerate then so is $\ind_\Omega\pi$.

Especially, when $f:A\ra B$ is an $\calH$-equivariant *-homomorphism,
\[\ind_\Omega(f):\ind_\Omega A\ra \ind_\Omega B,\]
\[\ind_\Omega(f)(\xi)(\omega)=f_{\sigma(\omega)}(\xi(\omega))\in B_{\sigma(\omega)}, \forall \xi\in \ind_\Omega A, \omega \in \Omega\]
is a $\calG$-equivariant *-homomorphism.
The data above constitute a functor from the category of $\calH$-\cst-algebras to the category of $\calG$-\cst-algebras.

\begin{rem}\label{remark on ind_c}
    If we see $\ind_\Omega A$ as a $\calG\ltimes(\Omega/\calH)$-\cst-algebra as in \cite[Proposition 3.17]{bonicke2020going}, $\ind_\Omega A$ is identified with $(\sigma^*A)^{\Omega\rtimes\calH}$. Here the $C_0(\Omega/\calH)$-algebra structure of $\ind_\Omega A$ is given by
    \[C_0(\Omega/\calH)\ra ZM(\ind_\Omega A),\]
    \[(f\cdot \xi)(\omega)=f(\omega\calH)\xi(\omega), \quad\forall f\in C_0(\Omega/\calH),\xi\in \ind_\Omega A, \omega\in \Omega.\]
    
    The $\calG\ltimes(\Omega/\calH)$-\cst-bundle associated to $(\sigma^*A)^{\Omega\rtimes\calH}$ is $(\sigma^*\ca A)/(\Omega\rtimes \calH)=\Omega\times_\calH\ca A$.

    We define
    \[\ind_{\Omega,c}A:=\{\xi\in \ind_\Omega A:\omega\calH\mapsto \|\xi(\omega)\| \text{ is compactly supported}\}.\]
    Therefore, $\ind_{\Omega,c}A$ is canonically identified with $\Gamma_c(\Omega/\calH,\Omega\times_\calH \ca A)\subseteq (\sigma^*A)^{\Omega\rtimes\calH}$ and is dense in $\ind_\Omega A$.
\end{rem}

The following proposition describes compact operators and adjointable operators on the induced module.

\begin{prop}\label{operators on induced modules}
    \textnormal{\cite[Proposition 4.11, Proposition 4.12]{miller2024functors}}
    Let $\Omega:\calG\leftarrow \calH$ be a correspondence, $A$ be an $\calH$-\cst-algebra and $E$ be an $\calH$-Hilbert $A$-module. Then
    \begin{enumerate}
        \item The map $\calK(\ind_\Omega E)\ra \ind_\Omega \calK(E)$, $T\mapsto [\omega\mapsto T_\omega]$ is an *-isomorphism.
        \item The map $\calL(\ind_\Omega E)\ra \Gamma_b(\Omega, \calL(\sigma_\Omega^*\ca E))$, $T\mapsto [\omega\mapsto T_\omega]$ is an embedding, whose image is the $\calH$-equivariant sections.
    \end{enumerate}
\end{prop}

The following construction provides a sufficient family of elements in induced algebras or induced modules.

\begin{prop}\label{diamond construction}
    Let $\Omega:\calG\leftarrow \calH$ be a correspondence, $(A,\alpha)$ be an $\calH$-\cst-algebra and $(E,V)$ be an $\calH$-Hilbert $A$-module. For every $\varphi\in C_c(\Omega), a\in A$ and $e\in E$, we define
    \[\varphi\diamond a(\omega)=\sum_{h\in\calH^{\sigma_\Omega(\omega)}} \varphi(\omega h)\alpha_h(a(s_\calH(h)))\in A_{\sigma_\Omega(\omega)},\]
    \[\varphi\diamond e(\omega)=\sum_{h\in\calH^{\sigma_\Omega(\omega)}} \varphi(\omega h)V_h(e(s_\calH(h)))\in E_{\sigma_\Omega(\omega)}.\]
    Then, $\varphi\diamond a$ is a well-defined element of $\ind_\Omega A$, $\varphi\diamond e$ is a well-defined element of $\ind_\Omega E$. Moreover, $C_c(\Omega)\diamond A:= span\{\varphi\diamond a:\varphi\in C_c(\Omega),a\in A\}$ is dense in $\ind_\Omega A$, $C_0(\Omega)\diamond E:= span\{\varphi\diamond e:\varphi\in C_c(\Omega),a\in E\}$ is dense in $\ind_\Omega E$.
\end{prop}
\begin{proof} Use the exactly same method in Lemma 3.13, Proposition 3.14 and Remark 3.15 of \cite{bonicke2020going}.\end{proof}

Finally, we are going to prove that, if $\Omega$ is second countable and $A$ is separable, $\ind_\Omega A$ should also be separable. Before that we need a lemma which is trivial but turns out to be very useful.

\begin{lem}\label{uniformly finite fiber}
    If $\rho:Y\ra X$ is a local homeomorphism between two locally compact Hausdorff spaces, $K\subseteq Y$ a compact subset of $Y$, then $\sup_{x\in X}\#(K\cap \rho\inv(x))<\infty$.
\end{lem}
\begin{proof}
    For every $y\in K$, there exists an open neighborhood $V_y$ of $y$ such that $\rho|_{V_y}$ is a homeomorphism onto an open of $X$. Since $K$ is compact, there exists finitely many point $y_1,\cdots, y_n$, such that $K\subseteq \cup_{i=1}^nV_{y_i}$. Since every $\rho|_{V_{y_i}}$ is injective, we have $\#(K\cap \rho\inv(x))\leqslant n$ for every $x\in X$.
\end{proof}

\begin{lem}\label{inequalities involving diamond}
    Let $\Omega:\calG\leftarrow \calH$ be a correspondence, $(A,\alpha)$ be an $\calH$-\cst-algebra and $(E,V)$ be an $\calH$-Hilbert $A$-module.
    \begin{enumerate}
        \item For every compact subset $K$ of $\Omega$, there exists a constant $M_K>0$, such that for any $\varphi\in C_c(\Omega)$ that is supported in $K$ and for any $a\in A$, $e\in E$, we have
        \[\|\varphi\diamond a\|_{\ind_\Omega A}\leqslant M_K \|\varphi\|_\infty \|a\|_A, \quad \|\varphi\diamond e\|_{\ind_\Omega E}\leqslant M_K\|\varphi\|_\infty \|e\|_E\]
        \item If $(a_\lambda)_\lambda$ is a convergent net in $A$ with limit $a$, $(e_\lambda)_\lambda$ is a convergent net in $E$ with limit $e$, $\varphi\in C_c(\Omega)$. Then $\varphi\diamond a_\lambda\ra \varphi\diamond a$, $\varphi\diamond e_\lambda\ra \varphi\diamond e$.
        \item Suppose that $(\varphi_\lambda)_\lambda$ is a net in $C_c(\Omega)$, $K$ is a compact subset of $\Omega$, $\varphi\in C_c(\Omega)$, $a\in A$, $e\in E$, such that
        \begin{enumerate}
            \item for any $\lambda$, $supp(\varphi_\lambda)\subseteq K$;
            \item $\varphi_\lambda\xrightarrow{\|\cdot\|_\infty}\varphi$.
        \end{enumerate}
        Then $\varphi_\lambda\diamond a\ra \varphi\diamond a$, $\varphi_\lambda\diamond e\ra \varphi\diamond e$.
    \end{enumerate}
\end{lem}
\begin{proof}
    (1) The action of $\calH$ on $\Omega$ is free, and $q:\Omega\ra \Omega/\calH$ is a local homeomorphism (proposition \ref{orbit space of free proper action of etale grpd}), by lemma \ref{uniformly finite fiber}, we define
\begin{align*}
    M_K & :=\sup_{\omega\in \Omega}\#\{h\in\calH^{\sigma_\Omega(\omega)}:\omega h\in K\}\\
    & =\sup_{\omega\calH\in \Omega/\calH}\#(q\inv(\omega\calH)\cap K)<\infty.
\end{align*}
So we have
\begin{align*}
    \|\varphi\diamond a\|_{\ind_\Omega A} & = \sup_{\omega\in \Omega}\|(\varphi\diamond a)(\omega)\|_{A_{\sigma_\Omega(\omega)}}
    \\& = \sup_{\omega\in \Omega}\|\sum_{h\in\calH^{\sigma_\Omega(\omega)}}\varphi(\omega h)\alpha_h(a(s_\calH(h)))\|_{A_{\sigma_\Omega(\omega)}}\\
    & \leqslant M_K \|\varphi\|_{\infty}\|a\|_A.
\end{align*}
Another inequality can be proved in the same way.

(2) Using the result of (1),
\[\|\varphi\diamond a_\lambda -\varphi\diamond a\|_{\ind_\Omega A}\leqslant M_{supp(\varphi)}\|\varphi\|_\infty \|a_\lambda-a\|_A.\]
Therefore, $\varphi\diamond a_\lambda\ra\varphi\diamond a$. Similarly, $\varphi\diamond e_\lambda\ra \varphi\diamond e$.

(3) Using the result of (1),
\[\|\varphi_\lambda\diamond a-\varphi\diamond a\|\leqslant M_K\|\varphi_\lambda-\varphi\|_{\infty}\|a\|_A.\]
Hence, $\varphi_\lambda\diamond a\ra \varphi\diamond a$. Similarly, $\varphi_\lambda\diamond e\ra \varphi\diamond e$.
\end{proof}

\begin{corr}\label{ind transfer pre-hilb mod}
    Let $\Omega:\calG\leftarrow\calH$ be an \'etale groupoid correspondence, $A$ be an $\calH$-\cst-algebra, $E$ be a $\calG$-Hilbert $A$-module, $E'$ be a dense subset of $E$, then
    $C_c(\Omega)\diamond E'=span\{\phi\diamond e:\phi\in C_c(\Omega),e\in E'\}$ is dense in $\ind_\Omega E$.
\end{corr}

\begin{prop}
    Let $\Omega:\calG\leftarrow \calH$ be an \'etale groupoid correspondence that is second countable, $A$ be an $\calH$-\cst-algebra and $E$ be an $\calH$-Hilbert $A$-module.
    \begin{enumerate}
        \item If $A$ is separable, then so is $\ind_\Omega A$.
        \item If $A$ is $\sigma$-unital, then so is $\ind_\Omega A$.
        \item If $E$ is a countably generated Hilbert $A$-module, then $\ind_\Omega E$ is a countably generated Hilbert $\ind_\Omega A$-module.
    \end{enumerate}
\end{prop}
\begin{proof}
     (1) Before all, let $S_1$ be a countable dense subset of $C_c(\Omega)$. Let $(K_i)_{i\in \mathbb N_+}$ be a sequence of compact subsets of $\Omega$ such that $\Omega=\cup_{i\in \mathbb N_+}K_i$ and for any $i\in \mathbb N_+$, $K_i\subseteq \inter({K_{i+1}})$.
    By Urysohn's lemma, for any $i\in \mathbb N_+$ we select $\varphi_i\in C_c(\Omega)$ such that $\varphi_i|_{K_i}=1$ and $supp(\varphi_i)\subseteq \inter({K_{i+1}})$.
    
    Let $S_2$ be a countable dense subset of $A$. We define $S=\{(f\varphi_i)\diamond a:f\in S_1,i\in \mathbb N_+, a\in S_2\}$. Then $S$ is a countable subset of $\ind_\Omega A$. It suffices to show that the span of $S$ is dense in $\ind_\Omega A$. By proposition \ref{diamond construction}, it suffices to show that for any $f\in C_c(\Omega)$ and $a\in A$, $f\diamond a$ is in the closure of $S$. Let $(f_n)_n$ be a sequence in $S_1$ such that $f_n\xrightarrow{\|\cdot\|_\infty}f$ and $(a_n)_n$ be a sequence in $S_2$ such that $a_n\ra a$. We can fix $i\in \mathbb N_+$ such that $supp(f)\subseteq K_i$. Hence, $f_n\varphi_i\xrightarrow{\|\cdot\|_\infty} f\varphi_i=f$ as $n\ra +\infty$. Use lemma \ref{inequalities involving diamond},
    \begin{align*}
        \|(f_n\varphi_i)\diamond a_n-f\diamond a\|_{\ind_\Omega A} & \leqslant \|(f_n\varphi_i)\diamond (a-a_n)\|_{\ind_\Omega A}+\|(f_n\varphi_i-f)\diamond a\|_{\ind_\Omega A}\\
        & \leqslant M_{K_i}(\|f_n\|_{\infty}\|a-a_n\|_A+\|f_n\varphi_i-f\|_\infty \|a\|_A).
    \end{align*}
    Therefore, $(f_n\varphi_i)\diamond a_n\ra f\diamond a$ as $n\ra +\infty$.

    (2) If $A$ is $\sigma$-unital, then so is $\sigma^*A\cong C_0(\Omega)\otimes_{\calH\units}A$. Then apply proposition \ref{invariant algebra keeps sigma unity} (where $\Omega/\calH$ is second countable because of proposition \ref{orbit space of free proper action of etale grpd}), $\ind_\Omega A\cong (\sigma^*A)^{\Omega\rtimes \calH}$ is also $\sigma$-unital.

    (3) We know that $E$ is countably generated if and only if $\calK(E)$ is $\sigma$-unital. By proposition \ref{operators on induced modules}, $\calK(\ind_\Omega E)\cong \ind_\Omega (\calK(E))$, which is also $\sigma$-unital by the previous result. Hence, $\ind_\Omega E$ is also countably generated.
\end{proof}

\subsection{\texorpdfstring{Homomorphism $\tau$}{Homomorphism tau}}

Recall that, if $\calG$ is a locally compact Hausdorff groupoid, $A,B,D$ are $\sigma$-unital $\calG$-\cst-algebras, there is a homomorphism (see definition 4.2.1 of \cite{legal1994})
\[\tau_D^\calG:\kk^\calG_0(A,B)\ra \kk^\calG_0(A\otimes_{\calG\units}D, B\otimes_{\calG\units}D),\]
\[[E,\pi,T]\mapsto [E\otimes_{\calG\units}D, \pi\otimes id_D, T\otimes id_D].\]
We will need an analog of proposition 17.8.6 of \cite{blackadar1998k}.

\begin{prop}\label{functoriality of tau}
    Let $A,B,D_1,D_2$ be $\sigma$-unital $\calG$-\cst-algebras, and let $h:D_1\ra D_2$ be a $\calG$-equivariant *-homomorphism. The following diagram commutes:
    \[
    \xymatrix{
    \kk^\calG_0(A,B) \ar[r]^{\hspace{-1.2cm}\tau_{D_1}^\calG} \ar[d]^{\tau_{D_2}^\calG} & \kk^{\calG}_0(A\otimes_{\calG\units}D_1, B\otimes_{\calG\units}D_1) \ar[d]^{(id_B\otimes h)_*} \\
    \kk^\calG_0(A\otimes_{\calG\units}D_2, B\otimes_{\calG\units}D_2) \ar[r]_{(id_A\otimes h)^*} &
    \kk^\calG_0(A\otimes_{\calG\units}D_1,B\otimes_{\calG\units} D_2)
    }
    \]
\end{prop}
\begin{proof} Let $(E,\pi,F)\in \mathbb E^\calG(A,B)$, $J=D_1\otimes_h D_2$, so $J$ is a closed right ideal of $D_2$ generated by $D_1$. Meanwhile, $J$ is also a $\calG$-\cst-algebra and the inclusion $J\hookrightarrow D_2$ is $\calG$-equivariant. Let $\pi_h:D_1\ra \calL_{D_2}(J)$ be defined as $\pi_h(d_1)j=h(d_1)j$ for any $d_1\in D_1, j\in J$.
\begin{align*}
    (id_B\otimes h)_*\circ \tau_{D_1}(E,\pi,F) & =(id_B\otimes h)_*(E\otimes_{\calG\units}D_1,\pi\otimes id,F\otimes id)\\
    & = (E\otimes_{\calG\units}J,\pi\otimes \pi_h, F\otimes id),
\end{align*}
\[(id_A\otimes h)^*\circ \tau_{D_2}(E,\pi,F)=(E\otimes_{\calG\units}D_2,\pi\otimes h,F\otimes id).\]
Let $\tilde E=\{f\in D_2[0,1]:f(1)\in J\}$. Let $\tilde \pi:D_1\ra \calL_{D_1[0,1]}(\tilde E)$ be defined as
\[\tilde\pi(d)(f)(t)=h(d)f(t), \forall d\in D_1,f\in \tilde E,t\in [0,1].\]
So $(E\otimes_{\calG\units}\tilde E,\pi\otimes\tilde \pi, F\otimes id)\in \mathbb E^\calG(A\otimes_{\calG\units}D_1, B\otimes_{\calG\units}D_2[0,1])$ is the desired homotopy.\end{proof}

\begin{corr}\label{kk-equiv for inclusion}
    Let $D_1,D_2$ be $\sigma$-unital $\calG$-\cst-algebras, $h: D_1\ra D_2$ be a $\calG$-equivariant *-homomorphism, $J=D_1\otimes_h D_2\subseteq D_2$, that is $J$ is the  closed right ideal generated by $h(D_1)$ in $D_2$. Let $\pi_h$ be the left action of $D_1$ on $J$ through $h$. Then
    \[[h]= [J,\pi_h,0]\in \kk^\calG_0(D_1,D_2).\]
\end{corr}
\begin{proof} In proposition \ref{functoriality of tau}, take $A=B=C_0(\calG\units)$. Consider $1_{C_0(\calG\units)}=[C_0(\calG\units),id,0]\in \kk^\calG_0(A,B)$. We can easily compute that $(id_B\otimes h)_*\circ \tau_{D_1}(1_{C_0(\calG\units)})=[J,\pi_h,0]$ and $(id_A\otimes h)^*\circ \tau_{D_2}(1_{C_0(\calG\units)})=[D_2,h,0]=[h]$. \end{proof}

\subsection{Groupoid pre-Hilbert modules}

For some technical reasons we need to handle with non-closed submodules of groupoid Hilbert modules. The following definition is borrowed from \cite{Paterson2009TheST}.

\begin{defn}
    Let $\calG$ be a locally compact Hausdorff groupoid with source map $s$ and range map $r$, $A$ be a $\calG$-\cst-algebras. Let $E'$ be a pre-Hilbert $A$-module and let $E$ be its completion. For $x\in \calG\units$, we let $E'_x$ be the image of $E'$ under the quotient map $E\ra E_x$. We say that $E'$ is a pre-$\calG$-Hilbert $A$-module, if $E$ is equipped with a $\calG$-Hilbert $A$-module structure $V\in \calL(s^*E,r^*E)$, such that for every $\gamma\in \calG$, $V_\gamma$ sends $E'_{s(\gamma)}$ onto $E'_{r(\gamma)}$.
\end{defn}

\begin{prop}
    Let $\calG,\calH$ be \'etale groupoids, $\Omega:\calG\leftarrow \calH$ be an \'etale groupoid correspondence, $A$ be an $\calH$-\cst-algebra and $E$ be an $\calH$-Hilbert $A$-module, then $\ind_{\Omega,c}E$ is a pre-$\calG$-Hilbert $A$-module. 
\end{prop}
\begin{proof}
    Claim: for every $x\in \calG\units$, $(\ind_{\Omega,c}E)_x=\ind_{\Omega^x,c}E$. For every $\xi\in \ind_{\Omega,c}E$, clearly $\xi|_{\Omega^x}\in \ind_{\Omega^x,c}E$. Conversely, if $\eta\in \ind_{\Omega^x,c}E$, firstly there exists $\xi\in \ind_{\Omega,c}E$ such that $\eta=\xi|_{\Omega^x}$. Since $supp(\eta)$ is a compact subset of $\Omega^x/\calH$, there exists $\varphi\in C_c(\Omega/\calH)$ such that $\varphi|_{supp(\eta)}=1$. Then $\eta=(\varphi\cdot\xi)|_{\Omega^x}$ is an element of $(\ind_{\Omega,c}E)_x$. We proved our claim.

    Now it is easy to see that for any $\gamma\in \calG$, $\ind_{\Omega^{s(\gamma)},c}E\ra \ind_{\Omega^{r(\gamma)},c}E, \eta\mapsto \eta(\gamma\inv-)$ is a bijection.
\end{proof}

\begin{lem}\label{to be a unitary}
    Let $A$ be a \cst-algebra and $E,F$ be two Hilbert $A$-modules. Let $E'$ be a dense $A$-submodule of $E$, $T:E'\ra F$ be a $\mathbb C$-linear map such that
    \begin{enumerate}
        \item for any $e_1,e_2\in E'$, $\langle Te_1,Te_2\rangle_F=\langle e_1,e_2\rangle_E$;
        \item $TE'$ is dense in $F$.
    \end{enumerate}
    Then $T$ can be extended to a unitary element of $\calL(E,F)$.
\end{lem}
\begin{proof}
    By (1), $T$ is bounded, therefore $T$ can be extended to an isometric $\mathbb C$-linear map $E\ra F$ since $E$ is the completion of $E'$. We still denote it by $T$. Since $TE'$ is isometric to $E'$, we have $TE=F$, so $T$ is a bijection. Then by (1) and continuity of the inner product, for any $e\in E$, $f\in F$, we have $\langle Te,f\rangle_F=\langle e,T\inv f\rangle$, hence $T$ is adjointable and $T^*=T\inv$.
\end{proof}

The following trivial results means that to check the $\calG$-equivariance of an operator or a representation, it suffices to check on a groupoid pre-Hilbert module, which is often the case in practice.

\begin{lem}\label{to be equivariant adjointable op}
    Let $\calG$ be a locally compact Hausdorff groupoid, $A$ be a $\calG$-\cst-algebra, $(E,V)$ and $(F,W)$ be $\calG$-Hilbert $A$-modules, $E'\subseteq E$ be a dense pre-$\calG$-Hilbert $A$-module. If $T\in \calL(E,F)$ such that for any $\gamma\in\calG$, $e\in E_{s(\gamma)}'$, $T_{r(\gamma)}V_\gamma e=W_\gamma T_{s(\gamma)}e$, then $T$ is $\calG$-equivariant.
\end{lem}
\begin{proof}
    For each $\gamma\in \calG$, $T_{r(\gamma)}V_\gamma$ and $W_\gamma T_{s(\gamma)}$ are two bounded $A_{s(\gamma)}$-linear maps $E_{s(\gamma)}\ra F_{r(\gamma)}$ that coincide on a dense subspace $E'_{s(\gamma)}$. Then for any $e\in E_{s(\gamma)}$, $T_{r(\gamma)}V_\gamma e=W_\gamma T_{s(\gamma)}e$.
\end{proof}

\begin{corr}\label{to be equivariant representation}
    Let $\calG$ be a locally compact Hausdorff groupoid, $A,B$ be $\calG$-\cst-algebras, $(E,V)$ be $\calG$-Hilbert $B$-modules, $E'\subseteq E$ be a dense pre-$\calG$-Hilbert $B$-module, $\pi:A\ra \calL(E)$ be a $C_0(\calG\units)$-representation. If for any $a\in A$, $\gamma\in \calG$ and $e\in E'_{s(\gamma)}$, $\pi_{r(\gamma)}(a(r(\gamma)))V_{\gamma}e=V_\gamma \pi_{s(\gamma)}(a(s(\gamma)))e$, then $\pi$ is $\calG$-equivariant.
\end{corr}

\section{Groupoid simplicial complexes}

For an étale groupoid $\calG$, we will use Rips complexes as a sufficient family of proper cocompact $\calG$-spaces. Therefore, we need a suitable definition of groupoid simplicial complexes, such that they have enough good topological properties and Rips complexes are well-defined groupoid simplicial complexes.

\subsection{Space of measures}\label{section-space of measures}

In this section we fix a continuous map between two locally compact Hausdorff spaces $\rho:Y \ra X$. 
\begin{defn}
    We denote the linear space of Radon measures on $Y$ by $R(Y)$, and let
\[P(Y):=\{\mu\in R(X): \exists x\in X, supp(\mu)\subseteq Y_x, \mu(Y)=1, \mu \text{ is positive}\},\]
be equipped with weak-*-topology defined by $C_c(Y,\mathbb R)$, that is a net $(\mu_\lambda)_\lambda$ in $P(Y)$ converges to $\mu$ if and only if for any $f\in C_c(Y,\mathbb R)$, $(\mu_\lambda(f))_\lambda$ converges to $\mu(f)$.

We define $P_{fin}(Y)$ as the subspace of elements in $P(Y)$ that have finite supports.
\end{defn}

We are going to study about this space $P(Y)$. Some ideas are from the section 3.1.4 of \cite{lassagne2013k}, but the setting is slightly different. It is easy to see that $P(Y)$ is Hausdorff.

From the definition, opens in form of
\[W(\mu_0, f,\epsilon)=\{\mu\in P(Y): |\mu(f)-\mu_0(f)|<\epsilon\},\]
where $\mu_0\in P(Y)$, $f\in C_c(Y,\mathbb R)$, $\epsilon>0$ form a topological base of $P(Y)$.

The following lemma says that we can canonically identify $Y$ with a subspace of $P(Y)$.

\begin{lem}\label{prob mes dim 0}
    \[Y\ra \{\delta_y: y\in Y\}, y\mapsto \delta_y\]
    is a homeomorphism, where $\{\delta_y: y\in Y\}$ is equipped with subspace topology of $P(Y)$.
\end{lem}

\begin{proof} Obviously this is a bijection. If $(y_i)$ is a net converging to $y$ in $Y$, then for any $f\in C_c(Y)$, $\delta_{y_i}(f)=f(y_i)$ converge to $f(y)=\delta_y(f)$. Hence, $\delta_{y_i}$ converges to $\delta_y$ in $P(Y)$, which means $y\mapsto \delta_y$ is continuous.

Now let $C\subseteq Y$ be a closed subset. Assume that $(\delta_{y_\lambda})_\lambda$ is a net in $\{\delta_y: y\in C\}$ that converges to $\delta_{y'}$, $y'\in Y\setminus C$. Let $V\subseteq Y\setminus C$ be a relatively compact open neighborhood of $y'$, by Urysohn's lemma, there exists $f\in C_c(Y,[0,1])$ such that $supp(f)\subseteq Y\setminus C$ and $f|_V=1$. Hence, $\delta_{y'}(f)=1$ but $\delta_{y_\lambda}(f)=0$ for any $\lambda$, which contradicts to the fact that $\delta_{y_\lambda}$ converges to $\delta_{y'}$. In conclusion, we proved that $\{\delta_y:y\in C\}$ is closed in $\{\delta_y:y\in Y\}$. Therefore, $y\mapsto \delta_y$ is a closed map. In conclusion this is a homeomorphism.
\end{proof}

For every $\mu\in P(Y)$, $\rho_*\mu=\delta_x$ is the Dirac measure on some point $x\in X$. We write $x=\tilde\rho(\mu)$.

\begin{prop}\label{first properties of P(Y)}
    The map $\tilde\rho:P(Y)\ra X$ is continuous. Moreover,
    \begin{enumerate}
        \item if $\rho$ is surjective, then so is $\tilde\rho$;
        \item if $\rho$ is proper, then so is $\tilde\rho$, and in this case $P(Y)$ is locally compact;
        \item if $\rho$ admits a continuous section, then so is $\tilde\rho$;
        \item if $Y$ second countable, then so is $P(Y)$.
    \end{enumerate} 
\end{prop}
\begin{proof}
    For any open $U$ of $X$, take any $\mu_0\in \tilde\rho\inv(U)$ and let $x_0=\tilde\rho(\mu_0)$. By Urysohn's lemma, we can find a function $f\in C_c(Y,[0,1])$ such that $f|_{Y_{x_0}}=1$ and $supp(f)\subseteq \rho\inv(U)$. Claim: $W(\mu_0,f,1)=\{\mu\in P(Y):|\mu(f)-\mu_0(f)|<1\}$ is an open neighborhood $\mu_0$ in $\tilde\rho\inv(U)$. If $\mu\in P(Y)$ such that $|\mu(f)-\mu_0(f)|=|\mu(f)-1|<1$, then $\mu(f)>0$ and therefore $supp(\mu)\cap supp(f)\neq \emptyset$, and therefore we have $\tilde\rho(\mu)\in U$.
    
    Hence, $\tilde\rho\inv(U)$ is open. So we proved that $\tilde\rho$ is continuous.

    (1) If $\rho$ is surjective, obviously $\tilde\rho$ is also surjective.

    (2) Claim: when $Y$ is compact, $P(Y)$ is a compact. Since by Banach-Alaoglu's theorem, the unit ball of $R(Y)$ is compact under weak-*-topology, it suffices to show that $P(Y)$ is a closed subset of the unit ball. Assume that $(\mu_\lambda)_{\lambda\in \Lambda}$ is a convergent net contained in $P(Y)$ and has a limit $\mu\in R(Y)$. Since $Y$ is compact, $\chi_Y\in C(Y)$, we have $\mu(Y)=\lim_\lambda \mu_\lambda(Y)=1$, and there exists a convergent subnet of $(\tilde\rho(\mu_\lambda))_\lambda$, i.e. an upward-filtering ordered set $\Lambda'$ and a monotone cofinal map $\kappa:\Lambda'\ra \Lambda$, such that $(\tilde\rho(\mu_{\kappa(\lambda')}))_{\lambda'\in \Lambda'}$ is convergent. Assume that $x$ is the limit. We will show that $supp(\mu)\subseteq Y_x$. Otherwise, if there exists $y'\in supp(\mu)\setminus Y_x$, let $U$ be an open neighborhood of $y'$ such that $U\cap Y_x=\emptyset$. Let $V$ be a relatively compact open neighborhood of $y'$ such that $V\subseteq \overline{V}\subseteq U$. By Urysohn's lemma, there exists $f\in C_c(U,[0,1])$ such that $f|_{\overline{V}}=1$. The fact $y'\in supp(\mu)$ implies that $\mu(f)\neq 0$. But $\mu_{\kappa(\lambda')}(f)=0$ for $\lambda'$ large enough, which contradicts to the fact that $\mu_{\kappa(\lambda')}\ra \mu$. Hence, there exists $x\in X$ such that $supp(\mu)\subseteq Y_x$. So $\mu$ is also in $P(Y)$, $P(Y)$ is therefore closed in $R(Y)$.

    Now assume that $\rho$ is proper, then for any compact subset $K\subseteq \calG\units$, $Y_K$ is compact. So $\tilde\rho\inv(K)\cong P(Y_K)$ is compact. The properness of $\tilde\rho$ implies that $P(Y)$ is locally compact.

    (3) Let $s:X\ra Y$ be a continuous section of $\rho$, we define $\tilde s:X\ra P(Y), x\mapsto \delta_{s(x)}$. Then $\tilde s$ is a continuous section of $\tilde\rho$.

    (4) By Banach-Alaoglu's theorem and \cite[13.4.2]{Dieudonne1970treatisevol2}, the unit ball of $R(Y)$ is compact and metrizable, hence second countable. As a subspace, $P(Y)$ is second countable.
\end{proof}

Now we consider the case that $\calG\rightrightarrows X$ is a locally compact Hausdorff groupoid, $Y$ is a locally compact Hausdorff $\calG$-space with anchor map $\rho$. We will see that $P(Y)$ is also a well-defined $\calG$-space.

\begin{prop}\label{P(Y) has continuous action}
    Let $\calG\rightrightarrows X$ be a locally compact Hausdorff groupoid, $Y$ be a locally compact Hausdorff $\calG$-space with anchor map $\rho$. Then the map
    \[\calG\times_{s,X,\tilde\rho}P(Y)\rightarrow P(Y), (\gamma,\mu)\mapsto \gamma.\mu.\]
    \[(\gamma.\mu)(f)=\mu(f|_{Y_{r(\gamma)}}(\gamma.-)),\quad \forall f\in C_c(Y,\mathbb R)\]
    is continuous. (Here we see $f|_{Y_{r(\gamma)}}(\gamma.-)$ as an element of $C_c(Y_{s(\gamma)},\mathbb R)$ and see $\mu$ as a probability measure on $Y_{s(\gamma)}$.) And this map makes $P(Y)$ a well-defined $\calG$-space.
\end{prop}
\begin{proof}
    To prove the continuity, for a convergent net $(\gamma_\lambda)_\lambda$ with limit $\gamma$ in $\calG$ and a convergent net $(\mu_\lambda)_\lambda$ with limit $\mu$ in $P(Y)$ such that $s(\gamma_\lambda)=\tilde\rho(\mu_\lambda)$ for all $\lambda$, for any $f\in C_c(Y,\mathbb R)$, it suffices to show that $(\gamma_\lambda.\mu_\lambda)(f)\ra (\gamma.\mu)(f)$. By Tietze extension theorem, there exists $g\in C_c(Y)$ such that $g|_{Y_{s(\gamma)}}=f|_{Y_{r(\gamma)}}(\gamma.-)$. Let $\ca A=\sqcup_{x\in \calG\units}C(Y_x)$ be the associated \cst-bundle of $C_0(Y)$. We know that $\calG$ acts continuously on $\ca A$, hence $f|_{Y_{r(\gamma_\lambda)}}(\gamma_\lambda.-)$ converges to $g|_{Y_{s(\gamma)}}$ in $\ca A$. Therefore, $f|_{Y_{r(\gamma_\lambda)}}(\gamma_\lambda.-)-g|_{Y_{s(\gamma_\lambda)}}$ converges to $0\in C_0(Y_{s(\gamma)})$ in $\ca A$. This implies that $\|f|_{Y_{r(\gamma_\lambda)}}(\gamma_\lambda.-)-g|_{Y_{s(\gamma_\lambda)}}\|_\infty\ra 0$ by \cite[Lemma C.18]{williams2007crossed}. Therefore,
    \begin{align*}
        |(\gamma_\lambda.\mu_\lambda)(f)- (\gamma.\mu)(f)| & = |\mu_\lambda(f|_{Y_{r(\gamma_\lambda)}}(\gamma_\lambda.-))-\mu(f|_{Y_{r(\gamma)}}(\gamma.-))|\\
        & \leqslant |\mu_\lambda(f|_{Y_{r(\gamma_\lambda)}}(\gamma_\lambda.-)-g|_{Y_{s(\gamma_\lambda)}})|+|\mu_\lambda(g)-\mu(g)|\\
        &\leqslant \|f|_{Y_{r(\gamma_\lambda)}}(\gamma_\lambda.-)-g|_{Y_{s(\gamma_\lambda)}}\|_\infty+|\mu_\lambda(g)-\mu(g)|
    \end{align*}
    converges to zero. So we proved the continuity of the action.

    If $x=\tilde\rho(\mu)$, then for any $f\in C_c(Y,\mathbb R)$
    \[(x.\mu)(f)=\mu(f|_{Y_x})=\mu(f)\]
    since $supp(\mu)\subseteq Y_x$. And for any $(\gamma,h)\in \calG^{(2)}$, $\mu\in \tilde\rho\inv(s(h))$ and $f\in C_c(Y,\mathbb R)$,
    \[[(\gamma h).\mu](f)=\mu(f|_{Y_{r(\gamma)}}(\gamma h.-))=(h.\mu)(f|_{Y_{r(\gamma)}}(\gamma.-))=[\gamma.(h.\mu)](f).\]
    That is $(\gamma h).\mu=\gamma.(h.\mu)$. In conclusion, we have a well-defined $\calG$-action on $P(Y)$.
\end{proof}

Every $\mu\in P_{fin}(Y)$ should be in form of a convex combination of Dirac measures on a same fiber. If $Y$ is some $\calG$-space, then for any $\gamma\in \calG$, $\mu=c_0\delta_{y_0}+\cdots +c_m \delta_{y_m}$ such that $\tilde\rho(\mu)=s(\gamma)$, we have
\[\gamma.\mu=c_0\delta_{\gamma y_0}+\cdots +c_m \delta_{\gamma y_m}.\]
So $P_{fin}(Y)$ is a $\calG$-invariant subspace of $P(Y)$.

Finally, we will focus on the case that $\rho$ is a local homeomorphism. We will repeat using the following trivial lemmas.

\begin{lem}\label{trivial lemma 1}
    Let $\rho:Y\ra X$ be a local homeomorphism between locally compact Hausdorff spaces. If $U$ is an open of $Y$ such that $\rho|_U$ is a homeomorphism, $f$ is an element of $C_c(Y)$ such that $supp(f)\subseteq U$, then for $\mu\in P(Y)$, $\mu(f)\neq 0$ implies that there exists uniquely a point $y\in U$ such that $\{y\}=supp(\mu)\cap U$. 
\end{lem}
\begin{proof}
    We know that $\rho|_U$ is injective, and $supp(\mu)\subseteq \rho\inv(\tilde\rho(\mu))$, hence $supp(\mu)\cap U$ has at most one element. But if $supp(\mu)\cap U=\emptyset$ then $\mu(f)=0$.
\end{proof}

\begin{lem}\label{trivial lemma 2}
    For a net $(\mu_\lambda)_\lambda$ of positive Radon measures, if $\mu_\lambda(Y)\ra 0$, then $\mu_\lambda\ra 0$.
\end{lem}

We conclude this subsection with a lemma that in this case the convergence of nets in $P(Y)$ can be reformulated in a more intuitive way.

\begin{lem}\label{reindex}
    Let $\rho:Y\ra X$ be a local homeomorphism between locally compact Hausdorff spaces. Suppose that $(\mu_\lambda)_{\lambda\in \Lambda}$ is a convergent net in $P(Y)$ with limit $\mu=\sum_{i=0}^m c_i\delta_{y_i}$, where $y_0,\cdots, y_m$ are distinct, $c_0,\cdots, c_m\in (0,1]$. Then there exists $\lambda_0\in \Lambda$, such that for any $0\leqslant i\leqslant m$ and for all $\lambda\geqslant \lambda_0$, there exists $y_i(\lambda)\in Y$ and $c_i(\lambda)\in (0,1]$ and a positive measure $\hat\mu_\lambda\in R(Y)$, such that
    \begin{enumerate}
        \item for any $\lambda\geqslant \lambda_0$,
        \[\mu_\lambda=\sum_{i=0}^m c_i(\lambda) \delta_{y_i(\lambda)}+\hat\mu_\lambda;\]
        \item for any $\lambda\geqslant \lambda_0$, $supp(\mu_\lambda)$ is the disjoint union of $\{y_0(\lambda), \cdots, y_m(\lambda)\}$ and $supp(\hat\mu_\lambda)$;
        \item for every $0\leqslant i\leqslant m$, we have $c_i(\lambda)\ra c_i$, $y_i(\lambda)\ra y_i$ and $\hat\mu_\lambda\ra 0$.
    \end{enumerate}
    Moreover, if $\#supp(\mu_\lambda)=m+1$ for $\lambda$ large enough, then $\hat\mu_\lambda=0$ for $\lambda$ large enough.
\end{lem}
\begin{proof}
    Since $\rho$ is a local homeomorphism and $Y$ is locally compact Hausdorff, we can select opens $V_0,\cdots, V_m$ and $U_0,\cdots, U_m$ such that
    \begin{enumerate}
        \item for each $0\leqslant i\leqslant m$, $U_i$ is an open neighborhood of $y_i$ such that $\rho|_{U_i}$ is an homeomorphism onto an open of $X$, and therefore they are disjoint;
        \item for each $0\leqslant i\leqslant m$, $V_i$ is a relatively compact open neighborhood of $y_i$ such that $V_i\subseteq \overline{V_i}\subseteq U_i$.
    \end{enumerate}
    By Urysohn's lemma, for each $0\leqslant i\leqslant m$ there exists $f_i\in C_c(Y)$ such that $f_i|_{\overline{V_i}}=1$ and $supp(f_i)\subseteq U_i$. Since $\mu_\lambda(f_i)\ra \mu(f_i)=c_i>0$, we can select $\lambda_0\in \Lambda$ such that for any $\lambda\geqslant \lambda_0$ and $0\leqslant i\leqslant m$, $\mu_\lambda(f_i)>0$. By lemma \ref{trivial lemma 1}, $y_i(\lambda):=\rho|_{U_i}\inv(\tilde\rho(\mu_\lambda))$ is the unique element in $supp(\mu_\lambda)\cap U_i$. By continuity of $\tilde\rho$ and $\rho|_{U_i}\inv$, we have $y_i(\lambda)\ra y_i$. Let $c_i(\lambda):=\mu_\lambda(f_i)$ and clearly $c_i(\lambda)\ra \mu(f_i)=c_i$.

    Let $\hat\mu_\lambda=\mu_\lambda-\sum_{i=0}^m c_i(\lambda)\delta_{y_i(\lambda)}$. Then $supp(\hat\mu_\lambda)\subseteq Y_{\tilde\rho(\mu_\lambda)}$. For any $f\in C_c(\cup_{i=0}^m U_i))$, $\hat\mu_\lambda(f)=0$. Hence, $supp(\mu_\lambda)$ has no intersection with $\{y_0(\lambda),\cdots, y_m(\lambda)\}$. For any $f\in C_c(Y\setminus (\cup_{i=0}^m \overline{V_i}))_+$, $\hat\mu_\lambda(f)=\mu(f)\geqslant 0$. Therefore for any $\lambda\geqslant \lambda_0$, $\hat\mu_\lambda$ is positive. Since $\hat\mu_\lambda(Y)=1-\sum_{i=0}^mc_i(\lambda)\ra 0$, by lemma \ref{trivial lemma 2}, $\hat\mu_\lambda\ra 0$.
\end{proof}

\subsection{Groupoid simplicial complexes}

\begin{defn}\label{defn groupoid simplicial complex}
    Let $X$ be a locally compact Hausdorff space. An $X$-simplicial complex of dimension less than $n$ is a pair $(Y,\Delta)$, where $Y$ is a locally compact space with a structure map $\rho: Y\ra X$, $\Delta$ is a family of finite subsets of $Y$ of cardinality at most $n+1$ such that
    \begin{enumerate}
        \item for any $\delta\in \Delta$, there exists $x\in X$ such that $\delta\subseteq \rho\inv(x)$;
        \item $\rho: Y\ra X$ is local homeomorphism;
        \item $\delta\in \Delta$, $\delta'$ is a non-empty subset of $\delta$, then $\delta'\in \Delta$.
    \end{enumerate}
\end{defn}

\begin{defn}
    Let $\calG\rightrightarrows X$ be an étale groupoid. A (left) $\calG$-simplicial complex of dimension less than $n$ is an $X$-simplicial complex $(Y,\Delta)$ of dimension less than $n$, such that $Y$ is a left $\calG$-space, and if $\delta=\{y_1,\cdots, y_m\}\in \Delta$ is contained in some fiber $\rho\inv(x)$, then for any $\gamma\in \calG_x$, $\gamma\delta:=\{\gamma y_1, \cdots, \gamma y_m\}\in\Delta$.
\end{defn}

So an $X$-simplicial complexes can be seen as the simplicial complexes for the trivial groupoid $X\rightrightarrows X$. We will only consider finite dimensional simplicial complexes.

\begin{rem}
    For $m\in \mathbb N$, define $\Delta^m$ as $\{\delta\in \Delta: |\delta|\leqslant m+1\}$. Clearly, $(Y,\Delta^m)$ is a $\calG$-simplicial complex.
\end{rem}

\begin{defn}\label{defn H1 H2}
    The geometric realization $|\Delta|$ of $(Y,\Delta)$ is the subspace $$\{\mu\in P(Y): supp(\mu)\in \Delta\}$$ of $P(Y)$. That is, a net $(\mu_\lambda)_\lambda$ in $|\Delta|$ converges to $\mu$ if and only if $(\mu_\lambda(f))_\lambda$ converges to $\mu(f)$ for every $f$ in $C_c(Y,\mathbb R)$.
\end{defn}

A base of open neighborhoods of $\mu_0\in |\Delta|$ are opens in form of
\[W(\mu_0,f,\epsilon):=\{\mu\in |\Delta|:|\mu(f)-\mu_0(f)|<\epsilon\},\]
where $f\in C_c(Y, \mathbb R)$, $\epsilon>0$.

We need the following two hypotheses to ensure that $|\Delta|$ has good topological properties. They showed up in \cite{bessi2023} and will essential.

\begin{defn}\label{H_1 H_2}
Let $X$ be a locally compact Hausdorff space, $(Y,\Delta)$ be an $X$-simplicial complex.
    \begin{enumerate}
        \item We say that $(Y,\Delta)$ has hypothesis $(H_1)$, if for any compact subset $K$ of $Y$,
        \[C_K:=\{y\in Y: \exists y'\in K,\{y,y'\}\in \Delta\}\]
        is compact subset of $Y$.
        \item We say that $(Y,\Delta)$ has hypothesis $(H_2)$, if $(y^{(0)}_\lambda)_\lambda, \cdots, (y^{(m)}_\lambda)_\lambda$ are convergent nets, such that $\lim_\lambda y^{(i)}_\lambda=y^{(i)}\in Y$, and for any $\lambda$, $\{y^{(0)}_\lambda, \cdots, y^{(m)}_\lambda\}\in \Delta$, then $\{y^{(0)},\cdots, y^{(m)}\}\in \Delta$.
    \end{enumerate}
\end{defn}

\begin{lem}\label{H_2 implies C_K is closed}
    Let $(Y,\Delta)$ be an $X$-simplicial complex that has hypothesis $(H_2)$, then for any compact subset $K\subseteq Y$, $C_K$ is closed in $Y$.
\end{lem}
\begin{proof}
    Let $(y_\lambda)_\lambda$ be a net contained in $C_K$ that $y_\lambda\ra y\in Y$, it suffices to show that $y\in C_K$. By definition, for every $\lambda$, there exists $y'_\lambda\in K$ such that $\{y_\lambda,y'_\lambda\}\in \Delta$ for every $\lambda$. Since $K$ is compact, the net $(y_\lambda)_\lambda$ has a convergent subnet, let it be $(y'_{\lambda'})_{\lambda'}$ with limit $y'\in K$. Then since for every $\lambda'$, $\{y_{\lambda'},y'_{\lambda'}\}\in \Delta$, and by hypothesis $(H_2)$, $\{y,y'\}\in \Delta$, this implies that $y\in C_K$.
\end{proof}

When $(Y,\Delta)$ is a $\calG$-simplicial complex, $|\Delta|$ is a $\calG$-invariant subspace of $P(Y)$. Hence, $|\Delta|$ is also a well-defined $\calG$-space.

In the following proposition, we will show that if $(Y,\Delta)$ has hypothesis $(H_2)$, all centers of $m$-simplices form a $\calG$-invariant closed subset of some \'etale $\calG$-space.

\begin{prop}\label{AmBm}
Let $\rho: Y\ra X$ be a local homeomorphism between two locally compact Hausdorff spaces. Define
    \[A^m=\{\frac{1}{m+1}(\delta_{y_0}+\cdots+\delta_{y_m})\in P(Y):\forall i\neq j,y_i\neq y_j\}, \]
    equipped with weak *-topology. Define
    \[B^m:=\{(y_0,\cdots, y_m)\in Y^{\times_X^{m+1}}: \forall i\neq j, y_i\neq y_j\}/ S_{m+1},\]
    where $Y^{\times_X^{m+1}}$ is the fiber product of $(m+1)$ copies of $Y$ over $X$, the symmetry group $S_{m+1}$ acts on it by permutation. Use $[y_0,\cdots, y_m]$ to denote the image of $(y_0,\cdots, y_m)$ under quotient map by $S_{m+1}$.
    \begin{enumerate}
    \item The map
    \[\phi: A^m\ra B^m,\frac{1}{m+1}
    (\delta_{y_0}+\cdots+\delta_{y_m})\mapsto [y_0,\cdots, y_m] \]
    is a homeomorphism.
    \item The map $\tilde\rho|_{A^m}:A^m\ra X$ is a local homeomorphism.
    \item If $(Y,\Delta)$ is an $X$-simplicial complex with hypothesis $(H_2)$, 
    \[center(m,\Delta):=A^m\cap |\Delta|\]
    is closed in $A^m$.
    \end{enumerate}
\end{prop}

We need the following lemma which is an easy case of lemma \ref{reindex}.

\begin{lem}\label{reindexing}
    Suppose that $(\mu_\lambda)_{\lambda\in \Lambda}$ a convergent net in $A^m$, with limit $\mu=\frac{1}{m+1}(\delta_{y_0}+\cdots+ \delta_{y_m})$ in $A^m$, where $y_0, \cdots, y_m$ are distinct. Then there exists $\lambda_0\in \Lambda$, such that for every $\lambda\geqslant\lambda_0$, there exists $y_0(\lambda),\cdots, y_m(\lambda)$ in $Y$ such that $supp(\mu_{\lambda})=\{y_0(\lambda),\cdots, y_m(\lambda)\}$, and for every $0\leqslant i\leqslant m$, $(y_i(\lambda))_{\lambda\geqslant \lambda_0}$ is a net converging to $y_i$.
\end{lem}

\begin{proof}[Proof of Proposition \ref{AmBm}]: (1): It is easy to see that $\phi$ is a well-defined bijection. The map
\[\psi:[y_0,\cdots,y_m]\mapsto \frac{1}{m+1}(\delta_{y_0}+\cdots+\delta_{y_n})\]
is its inverse and is continuous. It suffices to show that $\phi$ is continuous. If $(\mu_\lambda)_{\lambda\in \Lambda}$ is a net in $A^m$ converging to $\mu$, by lemma \ref{reindexing}, there exists $\lambda'\in \Lambda$, such that there exists $y_0(\lambda),\cdots, y_m(\lambda)$ for every $\lambda\geqslant \lambda'$, such that $(y_i(\lambda))_{\lambda\geqslant\lambda'}$ converges to $y_i$ and $supp(\mu_\lambda)=\{y_0(\lambda),\cdots, y_m(\lambda)\}$. This implies that $(\phi(\lambda))_\lambda$ converges to $\phi(\mu)$.

(2) In the following commutative diagram, the map $B^m\ra X$ is a local homeomorphism and $\phi$ is a homeomorphism.
\[\xymatrix{
A^m \ar[r]^{\phi} \ar[d]_{\tilde\rho} & B^m \ar[dl]\\
X
}\]

(3) Assume that $(\mu_\lambda)_{\lambda\in \Lambda}$ is a net in $center(m,\Delta)$ that converges to $\mu\in A^m$. Then by lemma \ref{reindexing}, there exists $\lambda_0\in \Lambda$, such that there exists $y_0(\lambda),\cdots, y_m(\lambda)$ for every $\lambda\geqslant \lambda'$, such that $(y_i(\lambda))_{\lambda\geqslant\lambda_0}$ converges to $y_i$ and $\{y_0(\lambda),\cdots, y_m(\lambda)\}$ is exactly support of $\mu_\lambda$. Since $(Y,\Delta)$ has hypothesis $(H_2)$, $supp(\mu_{\lambda})\in \Delta$ for every $\lambda\geqslant\lambda_0$, therefore $supp(\mu)\in \Delta$, $\mu\in center(m,\Delta)$.\end{proof}

\begin{corr}\label{center is closed in some etale space}
    If $\mathcal G$ is an \'etale groupoid, $(Y,\Delta)$ is a $\calG$-simplicial complex with property $(H_2)$, then for every $m\in \mathbb N$, $center(m,\Delta)$ is a $\calG$-invariant closed subset in $A^m$, while $A^m$ is an \'etale $\calG$-space.
\end{corr}

Next we are going to find a subbase of $|\Delta|$ for some $X$-simplicial complex $(Y,\Delta)$.

\begin{defn}
    If $\mu=\sum_{i=0}^m c_i\delta_{y_i}\in |\Delta|$, where all $c_i>0$ and $\sum_{i=0}^m c_i=1$, consider a family of functions $(f_i)_{0\leqslant i\leqslant m}$ in $C_c(Y,[0,1])$ and a positive real numbers $\epsilon$ such that
\begin{enumerate}
    \item for every $i$, there exists a relatively compact open neighborhood $V_i$ of $y_i$ such that $\rho|_{V_i}$ is homeomorphism onto an open of $X$; 
    \item for every $i$, $f_i$ is supported in $V_i$ and $f_i(y_i)> 0$;
    \item for every $i$, $0<\epsilon<\min \{c_i f_i(y_i)\}$.
\end{enumerate}

For $(f_i)_{0\leqslant i\leqslant m}$, $\epsilon>0$ that satisfy the conditions above, we define
\[W(\mu,(f_i)_i,\epsilon):=\cap_{i=0}^m W(\mu,f_i,\epsilon).\]
Let $S(\mu)$ be the set of all open neighborhoods of $\mu$ in this form.
\end{defn}

\begin{prop}
    $\mathcal S:=\cup_{\mu\in |\Delta|}S(\mu)$ is a subbase of $|\Delta|$.
\end{prop}
\begin{proof}
    Opens in form of $W(\mu,f,\epsilon)$, $f\in C_c(Y,\mathbb R)$ and $\epsilon>0$ form a basis of $|\Delta|$. It suffices to show that, for fixed $\mu, f,\epsilon$, there is a finite intersection of elements in $S(\mu)$ that contained in $W(\mu,f,\epsilon)$. Now let $K=supp(f)$, let $\{U_1,\cdots, U_l\}$ be a finite cover of $K$ in $Y$ such that $\rho|_{U_i}$ are homeomorphisms onto opens of $X$. Let $(\psi_i)_i$ be a partition of unity subordinated to $(U_i)_i$, that is $\psi_i\in C_c(U_i,[0,1])$, $\sum_{i=1}^l\psi_i(y)=1$ for any $y\in K$. Hence, $f=\sum_{i=1}^m f\psi_i$,
\[\cap_{i=1}^l W(\mu, f\psi_i,\frac{\epsilon}{l}) \subseteq W(\mu, f,\epsilon). \]
Let $f\psi_i=f_{i,+}-f_{i,-}$, $f_{i,+}, f_{i,-}\in C_c(U_i)_+$, hence
\[W(\mu,f_{i,+},\frac{\epsilon}{2l})\cap W(\mu,f_{i,-},\frac{\epsilon}{2l})\subseteq W(\mu, f\psi_i,\frac{\epsilon}{l}),\]
while
\[W(\mu,\frac{f_{i,+}}{\|f_{i,+}\|_\infty},\frac{\epsilon}{2l\|f_{i,+}\|_\infty})= W(\mu,f_{i,+},\frac{\epsilon}{2l}),\]
\[W(\mu,\frac{f_{i,-}}{\|f_{i,-}\|_\infty},\frac{\epsilon}{2l\|f_{i,-}\|_\infty})=W(\mu,f_{i,-},\frac{\epsilon}{2l}).\]

Hence, it suffices to consider the case that $f$ is supported in an open $U$ such that $\rho|_U$ is homeomorphism onto an open of $X$, and $f$ is valued in $[0,1]$. In this case $U$ contains at most one element of $supp(\mu)$. Let $W=W(\mu,f,\epsilon)$.

Suppose that $\mu=\sum_{i=0}^m c_i\delta_{y_i}\in |\Delta|$, $y_i\in Y$, $c_i\in (0,1]$. Let $x=\tilde\rho(\mu)$.

Case $\mu(f)=0$: in this case $f(y_i)=0$ for all $i$, $W=\{\mu'\in |\Delta|:\mu'(f)<\epsilon\}$. For each $0\leqslant i\leqslant m$, let $V_i$ be an open neighborhood of $y_i$ in $f\inv([0,\epsilon))$ such that $\rho|_{V_i}$ is a homeomorphism onto an open of $X$. Then $\cap_{i=0}^m\rho(V_i)$ is an open neighborhood of $x$. By Urysohn's lemma, there exists $\phi\in C_c(X,[0,1])$ such that $supp(\phi)\subseteq \cap_{i=0}^m\rho(V_i)$, $\phi(x)=1$. Define $\phi_i\in C_c(V_i,[0,1])$ as $\phi_i=\phi\circ \rho|_{V_i}$. Let $\epsilon'$ be a positive real number such that $\epsilon'<\min\{c_0,\cdots, c_m,\frac{\epsilon}{m+1}\}$. So $\cap_{i=0}^m W(\mu,\phi_i,\epsilon')$ is an element of $\ca S(\mu)$. If $\mu'\in \cap_{i=0}^m W(\mu,\phi_i,\epsilon')$, for each $i$, 
\[|\mu(\phi_i)-\mu'(\phi_i)|=|\mu'(\phi_i)-c_i|<\epsilon'<c_i,\]
so $\mu'(\phi_i)\neq 0$, so $supp(\mu')\cap V_i$ is one and only one point, we denote it by $y_i'$. We have $\mu'(\{y_i'\})\in (c_i-\epsilon',c_i+\epsilon')$. The set $supp(\mu')\cap \{y\in Y:f(y)\neq 0\}$ contains at most one point. If it is empty, $\mu'(f)=0$, hence $\mu'\in W$. 

If it is not empty, there is uniquely a point $y'\in supp(\mu')$ such that $f(y')\neq 0$. Then if there is some $0\leqslant i\leqslant m$ such that $y'=y_i'$, then 
\[\mu'(f)\leqslant f(y')<\epsilon,\]
that is $\mu'\in W$. Otherwise, $y'$ is different from all $y_i'$. So we have
\[\mu'(\{y'\})\leqslant 1-\sum_{i=0}^m\mu'(\{y_i'\})\leqslant 1-\sum_{i=0}^m(c_i-\epsilon')=(m+1)\epsilon'<\epsilon.\]
We still have $\mu'\in W$. In conclusion, we have $\cap_{i=0}^m W(\mu,\phi_i,\epsilon')\subseteq W$.

Case $\mu(f)\neq 0$: in this case, without loss of generality, we can assume that $supp(\mu)\cap U$ is just one point $y_0$ and $f(y_0)\neq 0$. For $1\leqslant i\leqslant m$, let $V_i$ be an open neighborhood of $y_i$ such that $\rho|_{V_i}$ is a homeomorphism onto an open of $X$. For $1\leqslant i\leqslant m$, choose any $f_i\in C_c(Y,[0,1])$ such that $supp(f_i)\subseteq V_i$ and $f_i(y_i)=1$. Let $\epsilon'$ be a positive real number such that $\epsilon'<\min\{c_0f(y_0),c_1,\cdots, c_m,\epsilon\}$. Then $W(\mu,f,\epsilon')\cap (\cap_{i=1}^mW(\mu,f_i,\epsilon'))$ is an element of $\ca S(\mu)$ and contained in $W$.
\end{proof}

\begin{prop}\label{H_1 H_2 implies LCH}
    Let $X$ be a locally compact Hausdorff space. If $(Y,\Delta)$ is an $X$-simplicial complex with hypotheses $(H_1)$ and $(H_2)$, then every element in $\mathcal S$ is relatively compact in $|\Delta|$. Moreover, $|\Delta|$ is locally compact.
\end{prop}

\begin{proof} For any
\[\mu_0=\sum_{i=1}^m c_i\delta_{y_i}\in |\Delta|,\]
it suffices to show that if $f\in C_c(Y,[0,1])$ is supported in a relatively compact open neighborhood $V_0$ of $y_0$ such that $\rho|_{V_0}$ is a homeomorphism onto an open of $X$, $f(y_0)\neq 0$, $0<\epsilon<c_0f(y_0)$, then $W=W(\mu_0, f, \epsilon)$ is relatively compact in $|\Delta|$.

Firstly if $\mu'\in W$, $\mu'(f)\neq 0$, hence there exists uniquely a point $y'\in supp(f)\cap supp(\mu')$, this implies that $supp(\mu')\subseteq C_{supp(f)}$. Since $(Y,\Delta)$ has hypothesis $(H_1)$, we see that $C_{supp(f)}$ is compact. We just showed that $\cup_{\mu\in W}supp(\mu)\subseteq C_{supp(f)}$. Assume that $C_{supp(f)}$ can be covered by relatively compact opens $U_1,\cdots, U_l$ such that $\rho|_{U_i}$ are homeomorphisms onto opens of $X$.

For every $\mu'\in W$, there exists (maybe not uniquely) $c_i(\mu')\in [0,1]$, $s_i(\mu')\in U_i$, such that
\[\mu'=\sum_{i=1}^l c_i(\mu')\delta_{s_i(\mu')}.\]
(If $supp(\mu')\cap V_i=\emptyset$, choose randomly $s_i(\mu')\in V_i$ and let $c_i(\mu')=0$.)

We will show that for all net $(\mu_\lambda)_\lambda$ in $W$,  there is a convergent subnet. 

Clearly $supp(\mu_\lambda)\subseteq \{s_1(\mu_\lambda),\cdots, s_l(\mu_\lambda)\}$. For every $i$, $(c_i(\mu_\lambda))_\lambda$ is a net in $[0,1]$, $(s_i(\mu_\lambda))_\lambda$ is a net in $U_i$. By repeating choosing subnet ($2l$ times), there is an upward-filtering ordered set $\Lambda'$ and a monotone cofinal map $\phi:\Lambda'\ra \Lambda$, such that for every $i$, $(c_{i}(\mu_{\phi(\lambda')}))_{\lambda'\in \Lambda'}$ is a convergent subnet of $(c_{i}(\mu_\lambda))$, $(s_i(\phi(\lambda')))_{\lambda'\in \Lambda'}$ is a convergent subnet of $(s_i(\mu_\lambda))$. Hence, $(\mu_{\phi(\lambda')})_{\lambda'\in \Lambda'}$ is a subnet of $(\mu_\lambda)_\lambda$ and a convergent net in $P(Y)$.

Without loss of generality, assume that $2\leqslant p\leqslant l$, such that $\lim_{\lambda'}c_{i}(\mu_{\phi(\lambda')})=0$ for all $p\leqslant i\leqslant l$, and $\lim_{\lambda'}c_{i}(\mu_{\phi(\lambda')})\neq 0$ for all $1\leqslant i\leqslant p-1$. After replacing the net by a subnet, we can assume that $c_{i}(\mu_{\phi(\lambda')})>0$ for all $1\leqslant i\leqslant p-1$, $\lambda'\in \Lambda'$. Hence, for all $\lambda'\in \Lambda'$, 
\[\{s_1(\mu_{\phi(\lambda')}), \cdots, s_p(\mu_{\phi(\lambda')})\}\in \Delta,\]
And $(Y,\Delta)$ has hypothesis $(H_2)$, hence
\[supp(\lim_{\lambda'}\mu_{\phi(\lambda')})=\{\lim_{\lambda'}s_1(\mu_{\phi(\lambda')}), \cdots, \lim_{\lambda'}s_p(\mu_{\phi(\lambda')})\}\in \Delta.\]
Therefore, $(\mu_{\phi(\lambda')})_{\lambda'\in \Lambda'}$ is a subnet of $(\mu_\lambda)_\lambda$ and a convergent net in $|\Delta|$. 
\end{proof}

\subsection{Topological properties of groupoid simplicial complexes}

Let $\calG\rightrightarrows X$ be an \'etale groupoid and $(Y,\Delta)$ be a $\calG$-simplicial complex.

\begin{defn}
     We say that $(Y,\Delta)$ is proper, if the action of $\calG$ on $Y$ is proper. We say that $(Y,\Delta)$ is $\calG$-compact, if $|\Delta^0|$ is $\calG$-compact.
\end{defn}

\begin{lem}\label{simplicial complex-compact supported is compact}
    Let $(Y,\Delta)$ be an $X$-simplicial complex with hypothesis $(H_2)$. Let $K$ be a compact subset of $Y$. Then 
    \[D_K:=\{\mu\in |\Delta|:supp(\mu)\subseteq K\}\]
    is a compact subset of $|\Delta|$.
\end{lem}

\begin{proof} By Banach-Alaoglu's theorem, it suffices to show that this subset is closed in the unit ball of $R(Y)$. Let $(\mu_\lambda)_{\lambda\in \Lambda}$ be a net in it that converges to $\mu\in R(Y)$. Without loss of generality, we can replace it by a subnet such that the cardinality of each $supp(\mu_\lambda)$ is the same. Therefore, there exists $m\in \mathbb N$, such that every $\mu_\lambda$ can be written as
\[\mu_\lambda=\sum_{i=0}^m c_{i,\lambda}\delta_{y_{i,\lambda}},\]
where $c_{i,\lambda}\in (0,1]$, $y_{i,\lambda}\in K$. Then by compactness of $[0,1]$ and $K$, after repeating choosing subnets, there exists a subnet of $(\mu_\lambda)_{\lambda}$, that is an upward-filtering ordered set $\Lambda'$ and a monotone cofinal map $k:\Lambda'\ra \Lambda$, such that for every $0\leqslant i\leqslant m$, $(c_{i,k(\lambda')})_{\lambda'\in \Lambda'}$ and $(y_{i,k(\lambda')})_{\lambda'\in \Lambda'}$ are respectively convergent nets in $[0,1]$ and $K$. Let $y_i=\lim_{\lambda'\in \Lambda'}y_{i,k(\lambda')}$. Hence, $supp(\mu)\subseteq \{y_0,\cdots, y_m\}\subseteq K$. While the hypothesis $(H_2)$ implies that $\{y_0,\cdots, y_m\}\in \Delta$. We can conclude that $\mu\in |\Delta|$ and $supp(\mu)\subseteq K$.\end{proof}

\begin{prop}\label{simplicial complex-cocompact realisation}
    Let $(Y,\Delta)$ be a $\calG$-compact $\calG$-simplicial complex with hypotheses $(H_1)$ and $(H_2)$. Then $|\Delta|$ is a $\calG$-compact $\calG$-space.
\end{prop}

\begin{proof} By definition $|\Delta^0|$ is a $\calG$-compact, $\calG$-invariant subspace of $Y$. Then there exists a compact subset $K$ of $|\Delta^0|$ such that $|\Delta^0|=\calG K$. So if $(Y,\Delta)$ is of dimension zero, it is directly given by the definition. Assume now $(Y,\Delta)$ has dimension $m\geqslant 1$. By hypothesis $(H_1)$, $C_K$ is compact. Let $D_{C_K}=\{\mu\in |\Delta|:supp(\mu)\in C_K\}$. By lemma \ref{simplicial complex-compact supported is compact}, $D_{C_K}$ is compact. We will show that $|\Delta|=\calG D_{C_K}$. For any $\mu\in |\Delta|$, assume that $supp(\mu)=\{y_0,\cdots, y_m\}$. Then there exists $(\gamma,k)\in \calG\times_{s,\rho}K$, such that $y_0=\gamma k$. We have $supp(\gamma\inv \mu)=\{k,\gamma\inv y_1,\cdots,\gamma\inv y_m\}$ and $\gamma\inv \mu\in |\Delta|$. For any $1\leqslant i\leqslant m$, $\{k,\gamma\inv y_i\}\in \Delta$, therefore $supp(\gamma\inv \mu)\subseteq C_K$, that is $\gamma\inv \mu\in D_{C_K}$.\end{proof}

\begin{lem}\label{D_K is sufficiently many}
    Let $(Y,\Delta)$ be an $X$-simplicial complex with hypothesis $(H_1)$ and $(H_2)$, for any compact subset $K'$ of $|\Delta|$, there exists a compact subset $K\subseteq Y$ such that $K'\subseteq D_K$.
\end{lem}

\begin{proof}
    We can find finitely many $\mu_1,\cdots,\mu_n\in K'$ and an open neighborhood $W_i\in S(\mu_i)$ for each $i$ such that $K'\subseteq \cup_{i=1}^nW_i$. By definition of the subbase $\mathcal S$, for each $i$ we can select $y_i\in supp(\mu_i)$, a relatively compact open neighborhood $V_i$ of $y_i$ such that $\rho|_{V_i}$ is a homeomorphism onto an open of $X$, a function $f_i\in C_c(Y,[0,1])$ supported in $V_i$ and $\mu_i(\{y_i\})f_i(y_i)>\epsilon_i>0$, such that $W_i\subseteq W(\mu_i,f_i,\epsilon_i)$. We take $K''=\cup_{i=1}^n \overline{V_i}$ and $K=C_{K''}$. By the hypothesis $(H_1)$, $K$ is compact.

    Now for any $\mu\in K'$, without loss of generality we can assume that $\mu\in W_1\subseteq W(\mu_1,f_1,\epsilon_1)$. This implies that $|\mu(f_1)-\mu_1(f_1)|>\epsilon_1$, while $\epsilon_1<\mu_1(\{y_1\})f_1(y_1)=\mu_1(f_1)$, hence $\mu(f_1)\neq 0$. Hence, $supp(\mu)\cap V_1\neq 0$. By the definition of $C_{K''}$, we have $supp(\mu)\subseteq K$, and therefore $\mu\in D_K$.
\end{proof}

\begin{prop}\label{proper realization is proper}
    Let $(Y,\Delta)$ be a proper $\calG$-simplicial complex with hypotheses $(H_1)$ and $(H_2)$. Then the action of $\calG$ on $|\Delta|$ is proper.
\end{prop}

\begin{proof}
By the previous lemma, it suffices to show that for any compact subset $K$ of $Y$, $\{\gamma\in \calG:\gamma D_K\cap D_K\neq \emptyset\}$ is compact. If there is $\gamma\in \calG$ and $\mu\in |\Delta|$ such that $supp(\mu)\subseteq K$ and $supp(\gamma\mu)\subseteq K$, then $\gamma K\cap K\neq \emptyset$. The set $\{\gamma\in \calG:\gamma D_K\cap D_K\neq \emptyset\}$ is therefore a closed subset of $\{\gamma\in \calG:\gamma K\cap K\neq\emptyset\}$, which is compact by the properness of $\calG\ltimes Y$.\end{proof}

\begin{lem}\label{A^m is proper}
    Suppose that $(Y,\Delta)$ is a proper $\calG$-simplicial complex, then $A^m$ (as defined in proposition \ref{AmBm}) is a proper $\calG$-space, for any $0\leqslant m\leqslant \dim \Delta$.
\end{lem}
\begin{proof}
    By proposition \ref{groupoid action-proper action base change}, $Y^{\times_X^{m+1}}$ is a proper $\calG$-space. After quotient by the action of a finite group, $Y^{\times_X^{m+1}}/S_{m+1}$ is also a proper $\calG$-space. While $B^m$ is a $\calG$-invariant clopen subset of $Y^{\times_X^{m+1}}/S_{m+1}$, the action of $\calG$ on $B^m$ is also proper. By proposition \ref{AmBm}, $A^m$ is homeomorphic to $B^m$, and it is easy to check that this homeomorphism is $\calG$-equivariant. In conclusion $A^m$ is a proper $\calG$-space.
\end{proof}

\subsection{Rips complexes}

We will provide a family of examples of groupoid simplicial complex called Rips complexes.

\begin{prop}\label{Rips complex}
    Let $\calG$ be an \'etale groupoid and $K$ be a compact subset of $\calG$. See $\calG$ as a left $\calG$-space with action defined by multiplication. Define
\[\Delta_K(\calG):=\{\delta\subseteq \calG: \forall \gamma_1,\gamma_2\in \delta, r(\gamma_1)=r(\gamma_2), \gamma_1\inv \gamma_2\in K\}.\]
\begin{enumerate}
    \item $(\calG,\Delta_K(\calG))$ is a finite dimensional $\calG$-simplicial complex.
    \item The $\calG$-simplicial complex $(\calG,\Delta_K(\calG))$ has hypotheses $(H_1)$ and $(H_2)$.
    \item The $\calG$-simplicial complex $(\calG,\Delta_K(\calG))$ is proper and $\calG$-compact.
\end{enumerate}

\end{prop}
\begin{proof}
    (1) To show that $(\calG,\Delta_K(\calG))$ is a $\calG$-simplicial complex, all conditions can be checked directly except the finiteness of dimension of $\Delta_K(\calG)$. Since $K$ is a compact subset in $\calG$ and $r$ is a local homeomorphism, 
    \[M=\sup_{x\in \calG\units}\#(K\cap r\inv(x))\]
    is finite. For every $\delta\in \Delta_K(\calG)$, choose any $\gamma\in \delta$, $\gamma\inv \delta\subseteq K$, we have
    \[|\delta|=|\gamma\inv \delta|\leqslant \#(K\cap r\inv(s(\gamma)))\leqslant M.\]
    Hence, $(\calG,\Delta_K(\calG))$ is a finite dimension $\calG$-simplicial complex.

    (2) If $K'$ is a compact subset of $\calG$, then
    \[C_{K'}=\{\gamma\in \calG:\exists \gamma'\in K', r(\gamma)=r(\gamma'), \gamma\inv \gamma'\in K\}\subseteq K'K\inv\]
    is compact. So $(\calG,\Delta_K(\calG))$ has hypothesis $(H_1)$.

    Assume that $(\gamma_\lambda^{(0)})_\lambda, \cdots, (\gamma_\lambda^{(m)})_\lambda$ are converging nets such that $\lim_\lambda \gamma_\lambda^{(i)}=\gamma^{(i)}\in \calG$ for all $0\leqslant i\leqslant m$, and for all $\lambda$, $\{\gamma_\lambda^{(0)}, \cdots, \gamma_\lambda^{(m)}\}\in \Delta_K(\calG)$. This implies that for all $\lambda$, for all $0\leqslant i,j\leqslant m$, $r(\gamma_\lambda^{(i)})=r(\gamma_\lambda^{(j)})$ and $(\gamma_\lambda^{(i)})\inv \gamma_\lambda^{(j)}\in K$. Then by the continuities of $r$ and the multiplication, for all $0\leqslant i,j\leqslant m$, $r(\gamma^{(i)})=r(\gamma^{(j)})$ and $(\gamma^{(i)})\inv \gamma^{(j)}\in K$, that is $\{\gamma^{(0)}, \cdots, \gamma^{(m)}\}\in \Delta_K(\calG)$. So $(\calG,\Delta_K(\calG))$ has hypothesis $(H_2)$.

    (3) Obviously $\calG$ is a proper $\calG$-space. And $\{\gamma\}\in \Delta_K(\calG)$ if and only if $s(\gamma)=\gamma\inv \gamma\in K$, so
    $|\Delta_K^0(\calG)|$ can be identified with $\calG_{K\cap \calG\units}$, which is $\calG$-compact.
\end{proof}

\begin{defn}\label{defn Rips complex}
    Use the same notation as above, $(\calG,\Delta_K(\calG))$ is called a Rips complex of $\calG$. We denote its geometrical realization by $P_K(\calG)$. If $K_1\subseteq K_2$ are two compact subsets of $\calG$, we use $\iota_{K_1,K_2}$ to denote the canonical inclusion of $P_{K_1}(\calG)$ into $P_{K_2}(\calG)$.
\end{defn}

We will have the following propositions to ensures that Rips complexes make a sufficient family of proper cocompact $\calG$-spaces to approximate a classifying space $\E \calG$ of proper actions of $\calG$.

\begin{prop}\label{Rips complexes G-homotopy sufficient 1}
    \textnormal{\cite[Lemma 3.2]{tu2010coarse}}
    Let $\calG$ be an \'etale groupoid, $Z$ be a locally compact Hausdorff proper $\calG$-compact $\calG$-space. Then there exists a compact subset $K$ of $\calG$ and a $\calG$-equivariant continuous map $Z\ra P_K(\calG)$. 
\end{prop}
\begin{proof}
    Let $\rho:Z\ra \calG\units$ be the anchor map. By proposition \ref{cuttoff function exist}, there exists a cutoff function $c\in C_c(Y)_+$ such that for any $z\in Z$, $\sum_{\gamma\in \calG_{\rho(z)}}c(\gamma z)=1$. Let $K=\{\gamma\in \calG:\gamma supp(c)\cap supp(c)\neq \emptyset\}$, which is compact. Then for every $z\in \calG$ we define $\phi(z)=\sum_{\gamma\in\calG_{\rho(z)}}c(\gamma z)\delta_{\gamma\inv}$. 

    For every $f\in C_c(\calG,\mathbb R)$,
    \[\phi(z)(f)=\sum_{\gamma\in \calG_{\rho(z)}}c(\gamma z)f(\gamma\inv),\]
    so $z\mapsto \phi(z)(f)$ is continuous by lemma \ref{sum of fiber}. So we proved that $\phi$ is continuous.

    If $\gamma,h\in supp(\phi(z))$, then $\gamma\inv h\in K$. Hence, $\phi(Z)\subseteq P_K(\calG)$. It is easy to check that $\phi$ is $\calG$-equivariant.
\end{proof}

\begin{prop}\label{Rips complexes G-homotopy sufficient 2}
    Let $\calG$ be an \'etale groupoid, $Z$ be a locally compact Hausdorff proper $\calG$-compact $\calG$-space, $K_1$ and $K_2$ be compact subsets of $\calG$, $\phi_1:Z\ra P_{K_1}(\calG)$, $\phi_2:Z\ra P_{K_2}(\calG)$ be two $\calG$-equivariant continuous maps. Then there exists a compact subset $K$ of $\calG$, such that $K\supseteq K_1\cup K_2$, and $\iota_{K_1,K}\circ \phi_1$ is $\calG$-homotopic to $\iota_{K_2,K}\circ \phi_2$.
    \[\xymatrix{
Z \ar[r]^{\phi_1}  \ar[d]^{\phi_2} & P_{K_1}(\calG) \ar[d]^{\iota_{K_1,K}}\\
P_{K_2}(\calG) \ar[r]^{\iota_{K_2,K}} & P_K(\calG)
}
\]
\end{prop}
\begin{proof}
By proposition \ref{G-compact iff quotient compact} there exists a compact subset $L$ of $Z$ such that $Z=\calG L$. Then $\phi_1(L)$ is a compact subset of $P_{K_1}(\calG)$, $\phi_2(L)$ is a compact subset of $P_{K_2}(\calG)$. By lemma \ref{D_K is sufficiently many}, there exists compact subsets $K_1', K_2'$ of $\calG$, such that $\phi_1(L)\subseteq\{\mu\in P(\calG):supp(\mu)\subseteq K_1'\}$ and $\phi_2(L)\subseteq\{\mu\in P(\calG):supp(\mu)\subseteq K_2'\}$. Now take $K_3=K_1'\cup K_2'$,  $K=K_1\cup K_2\cup K_3\inv K_3$. Define the map
\[\psi:[0,1]\times Z\ra P(\calG), (t,z)\mapsto \psi_t(z)=t\phi_1(z)+(1-t)\phi_2(z).\]
Clearly $\psi$ is $\calG$-equivariant. For every $z\in Z$, there exists $\gamma\in \calG$ and $z'\in L$ such that $z=\gamma z'$ and hence $\psi_t(z)=\gamma \psi_t(z')$. We see that $supp(\psi_t(z'))\subseteq supp(\phi_1(z'))\cup supp(\phi_2(z'))\subseteq K_3$, this implies that for any $h_1,h_2\in supp(\psi_t(z))$, $h_1\inv h_2\in K_3\inv K_3\subseteq K$. That is the image of $\psi_t$ is contained in $P_K(\calG)$. Therefore, $\psi$ is a linear homotopy between $\iota_{K_1,K}\circ \phi_1$ and $\iota_{K_2,K}\circ \phi_2$.
\end{proof}

\begin{rem}
    This can be seen as an analog of \cite[Lemma 7.2.13]{willett2020higher}.
\end{rem}

\subsection{Barycentric subdivision}

\begin{defn}\label{defn typed}
   A $\calG$-simplicial complex $(Y,\Delta)$ is typed, if there exists a finite discrete space $T$ and a continuous map $\tau: Y\ra T$, such that
    \begin{enumerate}
        \item the map $\tau$ is $\calG$-invariant, that is for any $(\gamma,y)\in \calG\times_{s,\rho}Y$, $\tau(\gamma y)=\tau(y)$;
        \item for any $\delta\in \Delta$, $\tau|_{\delta}$ is injective.
    \end{enumerate}
\end{defn}

\begin{prop}\label{typed sim com}
    Let $(Y,\Delta)$ be a typed $\calG$-simplicial complex of dimension less than $n$. Then for any $1\leqslant m\leqslant n$, there is a $\calG$-equivariant homeomorphism
    \[|\Delta^m|\setminus |\Delta^{m-1}|\cong center(m,\Delta)\times \sigma_m,\]
    where $\sigma_m:=\{(t_0,\cdots, t_m)\in \mathbb R^{m+1}:\forall i, t_i>0;\sum_{i=0}^m t_i=1\}$, $\calG$ acts on $center(m,\Delta)\times \sigma_m$ by \[\gamma.(\mu,(t_0,\cdots, t_m))=(\gamma\mu,(t_0,\cdots, t_m))\] when $s(\gamma)=\tilde\rho(\mu)$.
\end{prop}
\begin{proof}
    Without loss of generality, assume that $T=\{0,\cdots, N\}$. Let $T'$ be the finite set of all finite subset of cardinality $m+1$ of $T$, equipped with discrete topology. By lemma \ref{reindexing} and the continuity of $\tau$,
    \[\tilde\tau:|\Delta^m|\setminus |\Delta^{m-1}|\ra T', \mu\mapsto \{\tau(y):y\in supp(\mu)\}\]
    is continuous. Hence, for every $t\in T'$, $\tilde\tau\inv(t)$ is a clopen in $|\Delta^m|\setminus |\Delta^{m-1}|$. And we have $|\Delta^m|\setminus |\Delta^{m-1}|=\sqcup_{t\in T'}\tilde\tau\inv(t)$.
    
    Every $\mu\in |\Delta^m|\setminus |\Delta^{m-1}|$ can be uniquely expressed as $\sum_{k=0}^m c_k \delta_{y_k}$ such that $c_k\in (0,1]$, $\sum_{k=0}^mc_k=1$ and $\tau(y_0)<\tau(y_1)<\cdots <\tau(y_m)$. Every $c_k=\mu(\{y_k\})$. We define the map
    \[\phi:|\Delta^m|\setminus |\Delta^{m-1}|\cong center(m,\Delta)\times \sigma_m,\]
    \[\phi(\mu)=(\frac{1}{m+1}\sum_{k=0}^m\delta_{y_k},(c_0,\cdots, c_m)).\]
    We will show that for any $t=\{d_0<\cdots<d_m\}\in T'$, $\phi|_{\tilde\tau\inv(t)}$ is continuous. For any converging net $(\sum_{k=0}^mc_k(\lambda)\delta_{y_i(\lambda)})_\lambda$ in $\tilde\tau\inv(t)$ with limit $\sum_{k=0}^mc_k\delta_{y_k}$, by lemma \ref{reindex}, without loss of generality we can assume that for every $k$, $(c_k(\lambda))_\lambda$ is a net converging to $c_k$, $(y_k(\lambda))_\lambda$ is a net contained in $\tau\inv(k)$ converging to $y_k$. So $\frac{1}{m+1}\sum_{k=0}^m\delta_{y_k(\lambda)}$ converges to $\frac{1}{m+1}\sum_{k=0}^m\delta_{y_k}$ and $(c_0(\lambda),\cdots, c_m(\lambda))$ converges to $(c_0,\cdots,c_m)$. So $\phi|_{\tilde\tau\inv(t)}$ is continuous for all $t\in T$. So $\phi$ is continuous.

    Conversely, for any $t=\{d_0<\cdots< d_m\}\in T'$, define
    \[\psi_t: (center(m,\Delta)\cap \tilde\tau\inv(t))\times \sigma_m\ra |\Delta^m|\setminus |\Delta^{m-1}|,\]
    \[(\mu,(c_0,\cdots, c_m))\mapsto \sum_{k=0}^m c_k\delta_{\tau|_{supp(\mu)}\inv(d_k)}.\]
    The map $\psi_t$ is also continuous. We can define $\psi:center(m,\Delta)\times \sigma_m\ra |\Delta^m|\setminus |\Delta^{m-1}|$ by making its restriction on $(center(m,\Delta)\cap \tilde\tau\inv(t))\times \sigma_m$ to be $\psi_t$. It is easy to check that $\phi,\psi$ are $\calG$-equivariant and inverse to each other.
\end{proof}

\begin{prop}
    Let $\calG$ be an \'etale groupoid and $(Y,\Delta)$ be a $\calG$-simplicial complex of dimension less that $n$. We define $Y'=\sqcup_{i=1}^n A^i\subseteq P_{fin}(Y)$, define $\rho':Y\ra X$ to be $\tilde\rho|_{Y'}$. We define $\Delta'$ to be the family of finite subsets of $Y'\cap|\Delta|$ consisting of elements in form of $\{\mu_1,\mu_2,\cdots,\mu_k\}$ such that $supp(\mu_1)\subset supp(\mu_2)\subset\cdots\subset supp(\mu_k)$.

    Then $Y'$ is a well-defined \'etale $\calG$-space, $(Y',\Delta')$ is a well-define $\calG$-simplicial complex and $(Y',\Delta')$ is typed.
\end{prop}
\begin{proof}
    Every $A^i$ is $\calG$-invariant subspace of $P_{fin}(Y)$, hence $Y'$ is a well-defined $\calG$-space.
    And by proposition \ref{AmBm}, $\rho'|_{A^i}$ are local homeomorphisms, so $\rho'$ is a local homeomorphism and $Y'$ is an \'etale $\calG$-space. It is easy to check that $(Y',\Delta')$ satisfies all other conditions to be a $\calG$-simplicial complex.

    We define $\tau:Y'\ra \{1,\cdots, n+1\}$ as $\tau(\mu)=\#supp(\mu)$. For $0\leqslant i\leqslant n$, if a net $(\mu_\lambda)_\lambda$ is contained in $A^i$ that converges to $\mu\in P_{fin}(Y)$, then by lemma \ref{reindex}, $\tau(\mu)=i$. Therefore, every $A^i$ is closed in $Y'$. This implies that every $A^i$ is clopen in $Y'$. Hence, $\tau$ is continuous. It is easy to see that $\tau$ is $\calG$-invariant and for any $\delta\in \Delta'$, $\tau|_{\delta}$ is injective. So we proved that $(Y',\Delta')$ is typed.
\end{proof}

\begin{defn}\label{defn barycentric subdivision}
    Use the same notation in the previous proposition, we call $(Y',\Delta')$ the barycentric subdivision of $(Y,\Delta)$.
\end{defn}

\begin{prop}\label{bary center subdivision invariance}
    The geometric realization $|\Delta'|$ is $\calG$-equivariantly homeomorphic to $|\Delta|$.
\end{prop}
\begin{proof}
    We define
    \[\phi:|\Delta'|\ra |\Delta|,\quad \sum_{i=0}^m c_i\delta_{\mu_i}\mapsto \sum_{i=0}^m c_i\mu_i.\]
    
    Apparently $\phi$ is bijective and $\calG$-equivariant. For a converging net $(\mu'_\lambda)_\lambda$ in $|\Delta'|$  with limit $\sum_{k=1}^m c_k\delta_{\mu_k}\in |\Delta'|$, by lemma \ref{reindex}, we can assume that there are converging nets $c_k(\lambda)\ra c_k$ and $\mu_k(\lambda)\ra \mu_k$ such that $(\mu'_\lambda-\sum_{k=0}^m c_k(\lambda)\delta_{\mu_k(\lambda)})_\lambda$ is a net of positive measures on $Y'$ that converges to 0 and $(\mu'_\lambda-\sum_{k=0}^m c_k(\lambda)\delta_{\mu_k(\lambda)})(Y')\ra 0$. Then $(\phi(\mu'_\lambda)-\sum_{k=0}^m c_k(\lambda)\mu_k(\lambda))_\lambda$ is a net of positive measures on $Y$ and
    \[(\phi(\mu'_\lambda)-\sum_{k=0}^m c_k(\lambda)\mu_k(\lambda))(Y) = (\mu'_\lambda-\sum_{k=0}^m c_k(\lambda)\delta_{\mu_k(\lambda)})(Y')\ra 0. \]
    Hence, $\phi(\mu'_\lambda)\ra\sum_{k=0}^m c_k\mu_k$. So $\phi$ is continuous.

    For convenience of writing, for a finite subset $\{y_0,\cdots, y_m\}$ of $Y$ such that $\rho(y_1)=\cdots=\rho(y_j)$, we will write
    \[\lfloor y_0,\cdots, y_m\rfloor=\frac{1}{m+1}\sum_{i=0}^m\delta_{y_i}\in A^m\subseteq Y'.\]
    The inverse map $\phi\inv:|\Delta|\ra |\Delta'|$ is described as, if $\mu=\sum_{i=0}^m c_i\delta_{y_i}\in |\Delta|$ and $c_0\leqslant c_1\leqslant \cdots\leqslant c_m$, we have
    \begin{align*}
        \phi\inv(\mu)  = & (c_m-c_{m-1})\delta_{\delta_{y_m}}+2(c_{m-1}-c_{m-2})\delta_{ \lfloor y_{m-1},y_{m}\rfloor}+\cdots\\
        & +m(c_1-c_0)\delta_{\lfloor y_1,\cdots, y_m\rfloor}+(m+1)c_0 \delta_{\lfloor y_0,\cdots, y_m\rfloor}.
    \end{align*}

    We need to show that $\phi\inv$ is also a continuous map. Assume that $(\mu_\lambda)_\lambda$ is a net in $|\Delta|$ converging to $\mu\in |\Delta|=\sum_{i=0}^m c_i\delta_{y_i}$ such that $c_0\leqslant\cdots\leqslant c_m$, it suffices to prove that $\phi\inv(\mu_\lambda)\ra \phi\inv(\mu)$. By lemma \ref{reindex}, for $\lambda$ large enough, there exists $y_0(\lambda),\cdots, y_m(\lambda)\in supp(\mu_\lambda)$ and $c_i(\lambda)=\mu_\lambda(\{y_i(\lambda)\})$ for all $i$, such that for each $i$, $y_i(\lambda)\ra y_i$, $c_i(\lambda)\ra c_i$.
    
    Let $S_j(\lambda)=\lfloor y_{m+1-j}(\lambda),\cdots, y_m(\lambda)\rfloor\in Y'$ for all $1\leqslant j\leqslant m+1$. Define
    \[\mu_\lambda'=\sum_{j=1}^m j(c_{m+1-j}(\lambda)-c_{m-j}(\lambda))\delta_{S_j(\lambda)}+(m+1)c_0(\lambda)\delta_{S_{m+1}(\lambda)}.\]
    Clearly $\mu_\lambda'\ra \phi\inv(\mu)$.

    Unfortunately $\mu_\lambda'$ does not equal to $\phi\inv(\mu_\lambda)$ in general: $c_0(\lambda), \cdots , c_m(\lambda)$ may not be in an increasing order, and $supp(\mu_\lambda)$ may contain elements that are different from $y_0(\lambda),\cdots, y_m(\lambda)$. But we will see that $\phi\inv(\mu_\lambda)-\mu'_\lambda\ra 0$.

    Suppose that $n_1<\cdots<n_d$ are the integers such that
    \[c_0=\cdots=c_{n_1}<c_{n_1+1}=\cdots=c_{n_2}<\cdots <c_{n_d+1}=\cdots=c_m.\]
    And for convenience of writing let $n_0=-1$.
    
    For $0\leqslant l\leqslant d$, since $(c_{n_l+1}(\lambda))_\lambda,\cdots, (c_{n_{l+1}}(\lambda))_\lambda$
    are all nets that converge to same limit $c_{n_l}$, we define \[b_l(\lambda)=\max_{n_l+1\leqslant p<q\leqslant n_{l+1}}|c_p(\lambda)-c_q(\lambda)|,\quad b(\lambda)=\max_{0\leqslant l\leqslant d}b_l(\lambda).\]
    Then $b(\lambda)$ converges to 0. For a fixed $\lambda$, we can write $\mu_\lambda$ as
    \[\mu_\lambda=\sum_{i=0}^m c_i(\lambda)\delta_{y_i(\lambda)}+c_{-1}(\lambda)\delta_{y_{-1}(\lambda)}+\cdots +c_{-r}(\lambda)\delta_{y_{-r}(\lambda)}\]
    such that $c_{-1}(\lambda)\geqslant c_{-2}(\lambda)\geqslant\cdots \geqslant c_{-r}(\lambda)$, $\sum_{i=-r}^m c_i(\lambda)=1$ (here $r$ depends on $\lambda$). For $\lambda$ large enough, we have
    \[\min_{n_{l+1}+1\leqslant p\leqslant n_{l+2}}c_p(\lambda)>\max_{n_{l}+1\leqslant q\leqslant n_{l+1}}c_q(\lambda), \quad \forall 1\leqslant l
    \leqslant d,\]
    and
    \[\min_{0\leqslant p\leqslant n_1}c_p(\lambda)>1-\sum_{i=0}^mc_i(\lambda).\]
    (Since the two sides converge to two different limits.)

    These inequalities imply that, for a large enough and fixed $\lambda$, we can choose a permutation $\sigma$, such that for every $0\leqslant l\leqslant d$, $\sigma(n_l+1),\cdots, \sigma(n_{l+1})$ is a permutation of $n_l+1, \cdots, n_{l+1}$, and
    \[c_{-r}(\lambda)\leqslant \cdots \leqslant c_{-1}(\lambda)<c_{\sigma(0)}(\lambda)\leqslant \cdots\leqslant c_{\sigma(n_1)}(\lambda)<c_{\sigma(n_1+1)}(\lambda)\leqslant \cdots \leqslant c_{\sigma(m)},\]
    We write $\sigma(-p)=-p$ for $1\leqslant p
    \leqslant r(\lambda)$. For $1\leqslant j\leqslant m+r+1$, we define \[T_j(\lambda)=\lfloor y_{\sigma(m+1-j)}(\lambda),\cdots, y_{\sigma(m)}(\lambda)\rfloor\in Y'.\] Then especially $T_{m-n_l}(\lambda)=S_{m-n_l}(\lambda)$ for $0\leqslant l\leqslant d$.
    
    With these notations, we have
    \[\phi\inv(\mu_\lambda) =   \sum_{j=1}^{m+r}j(c_{\sigma(m+1-j)}(\lambda)-c_{\sigma(m-j)}(\lambda))\delta_{T_j(\lambda)}+ (m+1+r)c_{-r}(\lambda)\delta_{T_{m+r+1}(\lambda)}.
    \]

    Now we need to compare $\phi\inv(\mu_\lambda)$ with $\mu'_\lambda$. Let $I=\{1,2,\cdots, m+1\}\setminus \{m-n_l:0\leqslant l\leqslant d\}$. Then
    \[\phi\inv(\mu_\lambda)=A(\lambda)+B(\lambda)+C(\lambda),\]
    where
    \[A(\lambda)=\sum_{l=0}^d(m-n_l)(c_{\sigma(n_l+1)}(\lambda)-c_{\sigma(n_l)}(\lambda))\delta_{T_{m-n_l}(\lambda)},\]
    \[B(\lambda)=\sum_{j\in I} j(c_{\sigma(m+1-j)}(\lambda)-c_{\sigma(m-j)}(\lambda))\delta_{T_j(\lambda)}, \]
    \[C(\lambda)=\sum_{j=m+2}^{m+r}j(c_{m+1-j}(\lambda)-c_{m-j}(\lambda))\delta_{T_j(\lambda)}+(m+1+r)c_{-r}(\lambda)\delta_{T_{m+r+1}(\lambda)}.\]
    And
    \[\mu'_\lambda=A'(\lambda) + B'(\lambda),\]
    where
    \[A'(\lambda)=\sum_{l=0}^d(m-n_l)(c_{n_l+1}(\lambda)-c_{n_l}(\lambda))\delta_{S_{m-n_l}(\lambda)}+(m+1)c_{-1}(\lambda)\delta_{S_{m+1}(\lambda)},\]
    \[B'(\lambda)=\sum_{j\in I} j(c_{m+1-j}(\lambda)-c_{m-j}(\lambda))\delta_{S_j(\lambda)}.\]
    (Here $\sigma$ and $r$ depend on $\lambda$).

    Claim: $A(\lambda)-A'(\lambda)$, $B(\lambda)$, $B'(\lambda)$ and $C(\lambda)$ are nets of finitely supported measures that converge to 0. We know that $T_{m-n_l}(\lambda)=S_{m-n_l}(\lambda)$ for all $0\leqslant l\leqslant d$, and 
    \[(m-n_l)|c_{n_l+1}(\lambda)-c_{\sigma(n_l+1)}(\lambda)-c_{n_l}(\lambda)+c_{\sigma(n_l)}(\lambda)|\leqslant 2(m+1) b(\lambda)\ra 0.\]
    And $(m+1)c_{-1}(\lambda)$ is also a net that converges to 0. Hence, $A(\lambda)-A'(\lambda)$ is a net that converges to 0. All the coefficients in $B(\lambda)$ and $B'(\lambda)$ are controlled by $(m+1)b(\lambda)$, hence we have $B(\lambda)\ra 0$ and $B'(\lambda)\ra 0$.

    Finally,
    \begin{align*}
        C(\lambda)(Y')  & = \sum_{j=m+2}^{m+r}j(c_{m+1-j}(\lambda)-c_{m-j}(\lambda)) + (m+1+r)c_{-1}(\lambda)\\
        & = \sum_{k=-r}^{-1}c_k(\lambda)+(m+1)c_{-1}(\lambda)\\
        & = (1-\sum_{k=0}^m c_k(\lambda))+(m+1)c_{-1}(\lambda).
    \end{align*}
    Obviously $1-\sum_{k=0}^m c_k(\lambda)\ra 0$, $(m+1)c_{-1}(\lambda)\ra 0$. Hence, $C(\lambda)(Y')\ra 0$. And all $C(\lambda)$ are positive, therefore $C(\lambda)\ra 0$. We proved our claim.
    
    Hence, $\mu_\lambda'-\phi\inv(\mu_\lambda)\ra 0$, therefore $\phi\inv(\mu_\lambda)\ra\lim_\lambda\mu'_\lambda=\phi\inv(\mu)$, we proved that $\phi\inv$ is also continuous.
\end{proof}

\begin{prop}\label{promence property of barycenter subdivision}
Let $\calG$ be an \'etale groupoid and $(Y,\Delta)$ be a $\calG$-simplicial complex.
    \begin{enumerate}
        \item If $(Y,\Delta)$ has hypothesis $(H_2)$, then so is $(Y',\Delta')$.
        \item If $(Y,\Delta)$ has hypotheses $(H_1)$ and $(H_2)$, then $(Y',\Delta')$ has hypothesis $(H_1)$.
        \item If $(Y,\Delta)$ is proper, then so is $(Y',\Delta')$.
        \item If $(Y,\Delta)$ has hypotheses $(H_1)$ and $(H_2)$ and is $\calG$-compact, then $(Y',\Delta')$ is $\calG$-compact.
    \end{enumerate}
\end{prop}
\begin{proof}
    (1) Assume that $(\mu_\lambda^{(0)})_\lambda,\cdots,(\mu_\lambda^{(m)})_\lambda$ are convergent nets and with limits $\lim_\lambda \mu_\lambda^{(i)}=\mu^{(i)}\in Y'$, and for any $\lambda$, $\{\mu_\lambda^{(0)},\cdots,\mu_\lambda^{(m)}\}\in \Delta'$. Without loss of generality, after replacing by subnets, we assume that for each $i$, $\#supp(\mu_\lambda^{(i)})$ is a constant $n_i$. Since $(Y,\Delta)$ has hypothesis $(H_2)$, by proposition \ref{AmBm}, $center(n_i,\Delta)$ is closed in $A^{n_i}$ and therefore closed in $Y'$, we have $\mu^{(i)}\in center(n_i,\Delta)=|\Delta|\cap A^{n_i}$.

    Without loss of generality, we can assume that $n_0<\cdots<n_m$, so for every $\lambda$, $supp(\mu_\lambda^{(0)})\subset\cdots\subset supp(\mu_\lambda^{(m)})$. Using lemma \ref{reindexing} and the hypothesis $(H_2)$ of $(Y,\Delta)$, we can prove that $supp(\mu^{(0)})\subset\cdots\subset supp(\mu^{(m)})$. Hence, $\{\mu^{(0)},\cdots, \mu^{(m)}\}\in \Delta'$.

    (2) Let $K'$ be a compact subset of $Y'$. Use the previous result, $(Y',\Delta')$ has also hypothesis $(H_2)$. And $|\Delta|\cap Y'$ is closed in $Y'$ by the hypothesis $(H_2)$, this implies that $K'\cap |\Delta|$ is compact. Using lemma \ref{D_K is sufficiently many}, there exists a compact subset $K$ of $Y$ such that $K'\cap |\Delta|\subseteq D_K=\{\mu\in \Delta:supp(\mu)\subseteq K\}$. Now
    \begin{align*}
        & \{\mu\in Y':\exists\mu'\in K', \{\mu,\mu'\}\in \Delta'\} \\
        & = \{\mu\in Y'\cap |\Delta|:\exists\mu'\in K'\cap |\Delta|, \{\mu,\mu'\}\in \Delta'\}\\
        & \subseteq \{\mu\in Y'\cap |\Delta|:\exists\mu'\in D_K\cap Y', \{\mu,\mu'\}\in \Delta'\}\\
        & \subseteq \{\mu\in Y'\cap |\Delta|:\exists \mu'\in Y'\cap D_K, supp(\mu)\subseteq supp(\mu') \text{ or }supp(\mu')\subseteq supp(\mu)\}\\
        & \subseteq \{\mu\in Y'\cap |\Delta|: supp(\mu)\subseteq K\cup C_{K}\},
    \end{align*}
    where $C_{K}$ is a compact subset of $Y$ by hypothesis $(H_1)$ of $(Y,\Delta)$, then by lemma \ref{simplicial complex-compact supported is compact}, $\{\mu\in |\Delta|: supp(\mu)\subseteq K\cup C_{K}\}$ is compact. And by lemma \ref{H_2 implies C_K is closed}, $\{\mu\in Y':\exists \mu'\in K',\{\mu,\mu'\}\in \Delta'\}$ is a closed subset, hence also compact. So we proved that $(Y',\Delta')$ has hypothesis $(H_1)$.

    (3) Using lemma \ref{A^m is proper}, every $A^i$ is a proper $\calG$-space, hence $Y'$ is a proper $\calG$-space.

    (4) By the definition ${\Delta'}^{(0)}=\{\{\mu\}:\mu\in Y'\cap |\Delta|\}$. Then by lemma \ref{prob mes dim 0}, $|{\Delta'}\units|=Y'\cap |\Delta|$, which is a $\calG$-invariant closed subset of $|\Delta|$, while we see that $|\Delta|$ is $\calG$-compact by proposition \ref{simplicial complex-cocompact realisation}.
\end{proof}

\section{Induction-restriction adjunction}
In this section our main goal is to prove the following compression isomorphism. Let $\calG$ be an \'etale groupoid with a proper open subgroupoid $\calH\subseteq \calG$ (which is therefore relatively clopen by corollary \ref{proper subgroupoid is rel clopen}), $(A,\alpha)$ be an $\calH$-\cst-algebra, there is an $\calH$-equivariant embedding $i_A:A\ra \ind_{\calG^{\calH\units}_{\calH\units}}A$ given by
\[i_A(a)(\gamma)=\begin{cases}
    \alpha_{\gamma\inv}(a(r(\gamma))), & \gamma\in \calH\\
    0_{s(\gamma)}, & \gamma\not\in \calH.
\end{cases}\]
By lemma 3.21 of \cite{bonicke2020going}, let $\Omega$ be $\calG_{\calH\units}$, the restriction $(\ind_\Omega A)|_{\calH}$ of the $\calG$-\cst-algebra $\ind_\Omega A$ to $\calH$ can be identified with $\ind_{\calG^{\calH\units}_{\calH\units}}A$. The compression homomorphism $comp^\calG_\calH$ is defined as the composite
\[\kk^\calG_*(\ind_\Omega A,B)\xrightarrow{res^{\calG}_\calH}\kk^\calH_*(\ind_{\calG^{\calH\units}_{\calH\units}}A, B|_\calH)\xrightarrow{i_A^*}\kk^\calH_*(A,B|_\calH).\]

\begin{thm}\label{compression isomorphism}
    Let $\calG$ be a second countable \'etale groupoid with a proper open subgroupoid $\calH\subseteq \calG$. Let $\Omega=\calG_{\calH\units}$. If $A$ is a separable $\calH$-\cst-algebra and $B$ is a separable $\calG$-\cst-algebra, then
    \[comp^{\calG}_{\calH}: \kk^{\calG}_*(\ind_\Omega A,B)\ra \kk^{\calH}_*(A,B|_{\calH})\]
    is an isomorphism.
\end{thm}

It can be seen as a generalization of \cite[theorem 6.2]{bonicke2020going}. In fact, even though the setting is slightly different, we can prove it using the same method word verbatim as in \cite{bonicke2020going}, since we have pointed it out that proper open subgroupoids are relatively clopen(corollary \ref{proper subgroupoid is rel clopen}). However here we would like to provide a bicategorical proof. We will provide an observation that this induction-restriction adjunction comes from an adjunction in the bicategory of second countable \'etale groupoids. \cite{johnson20212} is our main reference about bicategories. Remark that we accept hom categories that may not be locally small.

\subsection{Bicategory \texorpdfstring{$\frgr$}{Gr} of \'etale groupoids}

In comparison with the definition of a bicategory \cite[Definition 2.1.3]{johnson20212}, consider the following data:
\begin{enumerate}
    \item Objects: the class of all \textbf{second countable} \'etale groupoids $B_0$. These are the 0-cells.
    \item Hom categories: for each pair of elements $\calG_1,\calG_2$ in $B_0$, there is an associated category $\mathfrak{Gr}(\calG_1,\calG_2)$, such that
    \begin{enumerate}
        \item its objects are \textbf{second countable} correspondences $\Omega:\calG_2\leftarrow \calG_1$. These are the 1-cells;
        \item for $\Omega_,\Omega'\in \mathfrak{Gr}(\calG_1,\calG_2)$, a morphism from $\Omega$ to $\Omega'$ is a $\calG_2,\calG_1$-equivariant continuous map $f:\Omega\ra \Omega'$. These are the 2-cells;
        \item for $\Omega,\Omega',\Omega''\in \mathfrak{Gr}(\calG_1,\calG_2)$, the vertical composition is composition of maps
        \[\mathfrak{Gr}(\calG_1,\calG_2)(\Omega'',\Omega')\times \mathfrak{Gr}(\calG_1,\calG_2)(\Omega',\Omega)\ra \mathfrak{Gr}(\calG_1,\calG_2)(\Omega'',\Omega), \]
        \[(g,f)\mapsto g\circ f;\]
        \item identity 2-cells are the identity maps $id_\Omega:\Omega\ra \Omega$.
    \end{enumerate}
    \item Identity 1-cells: for $\calG\in B_0$, the identity 1-cell in $\mathfrak{Gr}(\calG,\calG)$ is $\calG$ itself seen as a correspondence. Let $\{*\}$ be the category of one element, whose only morphism is the identity morphism. Define $1_\calG$ as the functor $1_\calG:\{*\}\ra \frgr(\calG,\calG)$ that sends $*$ to $\calG$.
    \item Horizontal compositions: for $\calG_1,\calG_2,\calG_3\in B_0$, the functor
    \[c_{\calG_1,\calG_2,\calG_3}:\frgr(\calG_2,\calG_1)\times \frgr(\calG_3,\calG_2)\ra \frgr(\calG_3,\calG_1)\]
    that sends $(\Omega,\Omega')$ to $\Omega\circ \Omega'=\Omega\times_{\calG_2}\Omega'$, and the horizontal composition of a pair of 2-cells $(f,f')$ is the 2-cell $[f,f']$. We omit the index and write $c$ when there is no ambiguity.
    \item Associators: for $\calG_1,\calG_2,\calG_3,\calG_4\in B_0$, the natural isomorphism $a_{\calG_1,\calG_2,\calG_3,\calG_4}$ from the functor \[c(c\times id_{\frgr(\calG_4,\calG_3)}):\frgr(\calG_2,\calG_1)\times \frgr(\calG_3,\calG_2)\times \frgr(\calG_4,\calG_3)\ra \frgr(\calG_4,\calG_1)\]
    \[(\Omega,\Omega',\Omega'')\mapsto (\Omega\circ \Omega')\circ \Omega''\]
    to the functor \[c(id_{\frgr(\calG_2,\calG_1)}\times c):\frgr(\calG_2,\calG_1)\times \frgr(\calG_3,\calG_2)\times \frgr(\calG_4,\calG_3)\ra \frgr(\calG_4,\calG_1)\]
    \[(\Omega,\Omega',\Omega'')\mapsto \Omega\circ (\Omega'\circ \Omega''),\]
    is defined as, at every $(\Omega,\Omega',\Omega'')$, the component 2-cell is the following $\calG_1,\calG_4$-equivariant homeomorphism as proved in \cite[Lemma 6.4]{antunes2021bicategory}
    \[a_{\Omega,\Omega',\Omega''}:(\Omega\circ \Omega')\circ \Omega''\ra \Omega\circ (\Omega'\circ \Omega''),\]
    \[[[\omega,\omega'],\omega'']\mapsto [\omega,[\omega',\omega'']].\]
    \item Unitors: for every $\calG_1,\calG_2\in B_0$, the left unitor is defined as the natural isomorphism $l_{\calG_2,\calG_1}$ from the functor
    \[c(1_{\calG_1}\times id_{\frgr(\calG_2,\calG_1)}):\frgr(\calG_2,\calG_1)\ra \frgr(\calG_2,\calG_1), \Omega\mapsto \calG_1\circ \Omega\]
    to the functor $id_{\frgr(\calG_2,\calG_1)}$, that at every $\Omega\in \frgr(\calG_2,\calG_1)$, the component 2-cell is the following $\calG_1,\calG_2$-equivariant homeomorphism by \cite[Lemma 6.3]{antunes2021bicategory},
    \[l_\Omega:\calG_1\circ \Omega\ra \Omega, [\gamma,\omega]\mapsto \gamma\omega.\]
    Similarly, the right unitor is the natural isomorphism 
    \[r_{\calG_2,\calG_1}:c(id_{\frgr(\calG_2,\calG_1)}\times i_{\calG_2})\ra id_{\frgr(\calG_2,\calG_1)}\] that at $\Omega\in \frgr(\calG_2,\calG_1)$, the component 2-cell is the following $\calG_1,\calG_2$-equivariant homeomorphism
    \[r_\Omega:\Omega\circ \calG_2\ra \Omega,[\omega,\gamma']\mapsto\omega\gamma'.\]
\end{enumerate}

We will often omit the subscripts of $c,a,l,r$.

\begin{prop}
    The data above defines a bicategory $\frgr$.
\end{prop}
\begin{proof} All second countable locally compact Hausdorff groupoids, correspondences (not necessarily second countable and Hausdorff), bi-equivariant continuous maps, and similarly defined identity 1-cells, horizontal composition functor, associators and unitors form a bicategory by proposition 6.5 and remark 6.2 of \cite{antunes2021bicategory}. Now we take a restriction on choice of 1-cells and 2-cells. Horizontal composition of two second countable locally compact Hausdorff correspondences is still second countable locally compact Hausdorff.\end{proof}

\begin{rem}
    For two \'etale groupoids $\calG,\calH\in B_0$, $\calG$ is Morita equivalent to $\calH$ if and only if there exists  correspondences $\Omega:\calG\leftarrow \calH$ and $\Lambda:\calH\leftarrow \calG$ such that $\Omega\circ \Lambda$ is $\calG,\calG$-equivariantly homeomorphic to $\calG$, $\Lambda\circ \Omega$ is $\calH,\calH$-equivariantly homeomorphic to $\calH$. That is $\Omega,\Lambda$ are 1-cells that are inverse to each other.
\end{rem}

From now on, if $\Omega,\Omega'$ are correspondences $\calG\leftarrow \calH$ and $f: \Omega\ra \Omega'$ be a $\calG,\calH$-equivariant continuous map, we define $\bar{f}:\Omega/\calH\ra \Omega'/\calH$ as the map $\omega\calH\mapsto f(\omega)\calH$.

We will use the freeness of the right actions on correspondences to have the following identifications of fibers.

\begin{lem}\label{basic property of 2-cells}
Let $\calG,\calH$ be \'etale groupoids, $\Omega,\Omega':\calG\leftarrow \calH$ be two correspondences, $f: \Omega\ra \Omega'$ be a $\calG,\calH$-equivariant continuous map. Then $f, \bar f$ are local homeomorphisms. For any $\omega'\in \Omega'$,
        \[f\inv(\omega')\ra \overline{f}\inv(\omega'\calH),\omega\mapsto \omega\calH\]
        is a bijection.
\end{lem}
\begin{proof}
    The map $f$ is $\calH$-equivariant, then $\sigma_{\Omega'}\circ f=\sigma_\Omega$, where $\sigma_{\Omega},\sigma_{\Omega'}$ are local homeomorphisms, therefore so is $f$. Let $q_\Omega:\Omega\ra \Omega/\calH$, $q_{\Omega'}:\Omega'\ra \Omega'/\calH$ be the quotient maps. By proposition \ref{orbit space of free proper action of etale grpd}, $q_{\Omega}, q_{\Omega'}$ are local homeomorphisms. Since $\bar f\circ q_\Omega=q_{\Omega'}\circ f$, the map $\bar f$ is also a local homeomorphism.

    If $\omega\in \Omega$ such that $\overline{f}(\omega\calH)=\omega'\calH$, then $f(\omega)\calH=\omega'\calH$, so there exists $\gamma\in \calH$ such that $f(\omega)=\omega'\gamma$, therefore $\omega\gamma\inv\in f\inv(\omega')$ and $\omega\gamma\inv\calH=\omega\calH\in \overline{f}\inv(\omega'\calH)$. That is the map is surjective. Now if $\omega_1,\omega_2\in f\inv(\omega')$ such that $\omega_1\calH=\omega_2\calH$, there exists $\gamma\in \calH$ such that $\omega_1=\omega_2\gamma$. Then $\omega'=f(\omega_1)=f(\omega_2\gamma)=f(\omega_2)\gamma=\omega'\gamma$. By the freeness of the action of $\calH$ on $\Omega'$, $\gamma\in \calH\units$, which implies that $\omega_1=\omega_2$. Hence, the map is injective.
\end{proof}

\begin{lem}\label{canonical bijection between fibers}
    Let $\calG_1,\calG_2,\calG_3$ be \'etale groupoids,   $\Omega,\Omega':\calG_3\leftarrow \calG_2$ and $\Lambda,\Lambda':\calG_2\leftarrow \calG_1$ be \'etale groupoid correspondences, $f:\Omega\ra \Omega'$ be a $\calG_3,\calG_2$-equivariant continuous map, $g:\Lambda\ra \Lambda'$ be a $\calG_2,\calG_1$-equivariant continuous map. 

    \begin{enumerate}
        \item For any $[\omega',\lambda]$ in $\Omega'\circ\Lambda$,
        \[f\inv(\omega')\ra [f,id]\inv([\omega',\lambda]), \omega\mapsto [\omega,\lambda]\]
        is a bijection.
        \item For any $[\omega,\lambda']\in \Omega\circ \Lambda'$,
        \[g\inv(\lambda')\ra [id,g]\inv([\omega,\lambda']),\lambda\mapsto [\omega,\lambda]\]
        is a bijection.
    \end{enumerate}
\end{lem}
\begin{proof}
    It is easy to see that all these maps are well-defined.
    
    (1) For any $[\omega_*,\lambda_*]$ in $[f,id]\inv([\omega',\lambda])$, $[\omega',\lambda]=[f(\omega_*),\lambda_*]$. This implies that there exists $\gamma\in \calG_2$, $f(\omega_*)=\omega'\gamma$ and $\lambda_*=\gamma\inv\lambda$. Thus, $[\omega_*,\lambda_*]=[\omega_*\gamma\inv,\lambda]$. That is $[\omega_*,\lambda_*]$ is the image of $\omega_*\gamma\inv\in f\inv(\omega')$, so the map is surjective. If there exists $\omega_1,\omega_2\in f\inv(\omega')$ such that $[\omega_1,\lambda]=[\omega_2,\lambda]$, there exists $\gamma\in \calG_2$ such that $\omega_1=\omega_2\gamma$, $\lambda=\gamma\inv\lambda$, so we have $\omega'=f(\omega_1)=f(\omega_2)\gamma=\omega'\gamma$. By the freeness of the action of $\calG_2$ on $\Omega'$, $\gamma\in \calG_2\units$ and therefore $\omega_1=\omega_2$. So the map is injective.

    (2) If $[\omega_*,\lambda_*]\in [id,g]\inv([\omega,\lambda'])$, then $[\omega,\lambda']=[\omega_*,g(\lambda_*)]$. So there exists $\gamma\in \calG_2$ such that $\omega=\omega_*\gamma$, $\lambda'=\gamma\inv g(\lambda_*)=g(\gamma\inv\lambda_*)$. That is $[\omega_*,\lambda_*]$ is the image of $\gamma\inv\lambda_*\in g\inv(\lambda')$, so the map is surjective. If there exists $\lambda_1,\lambda_2\in g\inv(\lambda')$ such that $[\omega,\lambda_1]=[\omega,\lambda_2]$, there exists $\gamma\in \calG_2$ such that $\omega=\omega\gamma$ and $\lambda_1=\gamma\inv\lambda_2$. The freeness of the action of $\calG_2$ on $\Omega$ implies that $\gamma\in\calG\units_2$ and therefore $\lambda_1=\lambda_2$. So the map is injective.
\end{proof}

\subsection{Bicategory \texorpdfstring{$\mathfrak{KK}$}{KK}}

For an \'etale groupoid $\calG$, we use $\kk^\calG$ to denote the category whose objects are (ungraded) separable $\calG$-\cst-algebras, where the hom set $\kk^\calG(A,B)$ is the Kasparov group $\kk_0^\calG(A,B)$ as defined in \cite{le1999theorie}, composition is the Kasparov product as defined in \cite[Theorem 6.3]{le1999theorie}, the identity morphism in $\kk^\calG(A,A)$ is the element $1_A=[(A,id_A,0)]$. The Kasparov product has associativity (see \cite[Theorem 6.4]{le1999theorie}) and unity, hence $\kk^\calG$ is a well-defined category. We will write $A\in \kk^\calG$ to denote that $A$ is a separable $\calG$-\cst-algebra.

Follow the definition, for two \'etale groupoids $\calG_1$ and $\calG_2$, a functor $F:\kk^{\calG_1}\ra \kk^{\calG_2}$ is therefore the data of
\begin{enumerate}
    \item the assignment that maps $A\in \kk^{\calG_1}$ to some $F(A)\in \kk^{\calG_2}$;
    \item the class of maps
    \[\kk^{\calG_1}(A,B)\ra \kk^{\calG_2}(F(A),F(B)), x\mapsto F(x),\]
\end{enumerate}
which satisfy the condition that $F$ preserves Kasparov product, that is, for any three separable $\calG_1$-\cst-algebras $A,B,C$ and for any $x\in \kk^{\calG_1}(A,B)$, $y\in \kk^{\calG_1}(B,C)$, we have $F(x\otimes y)=F(x)\otimes F(y)$.

When every map $\kk^{\calG_1}(A,B)\ra \kk^{\calG_2}(F(A),F(B))$ is a homomorphism of abelian groups, we can say that $F$ is additive.

If $F,G:\kk^{\calG_1}\ra \kk^{\calG_2}$ are two functors, a natural transformation $\alpha:F\ra G$ consists of a class of elements of $(\alpha_A)_{A\in \kk^{\calG_1}}$, in which for any $A\in \kk^{\calG_1}$, $\alpha_A$ is an element of $\kk^{\calG_2}(F(A),G(A))$, such that for any $A,B\in \kk^{\calG_1}$ and $x\in \kk^{\calG_1}(A,B)$, we have
\[\alpha_A\otimes G(x)=F(x)\otimes\alpha_B.\]
Equivalently, we can say that the following diagram commutes, in which arrows are elements of KK-groups.
\[\xymatrix{
F(A) \ar[r]^{F(x)} \ar[d]^{\alpha_A} & F(B) \ar[d]^{\alpha_B} \\
G(A) \ar[r]^{G(x)} & G(B)
}\]

Consider the following data:
\begin{enumerate}
    \item Objects: the class $B_0'$ of all categories $\kk^\calG$, where $\calG\in B_0$. These are the 0-cells.
    \item Hom categories: for two \'etale groupoids $\calG_1,\calG_2$ in $B_0'$, define the hom category as the functor categories $\mathrm{Fun}(\kk^{\calG_1},\kk^{\calG_2})$. That is
    \begin{enumerate}
        \item the 1-cells are functors from $\kk^{\calG_1}$ to $ \kk^{\calG_2}$;
        \item the 2-cells are natural transformations;
        \item the vertical composition of 2-cells is the composition of natural transformations;
        \item identity 2-cells are identity functors.
    \end{enumerate}
    \item Identity 1-cells: for $\calG\in B_0$, the identity 1-cell in $\mathrm{Fun}(\kk^\calG,\kk^\calG)$ is the identity functor $id_{\kk^\calG}$. We use $1'_\calG$ to denote the functor $\{*\}\ra \mathrm{Fun}(\kk^\calG,\kk^\calG)$ that sends $*$ to $id_{\calG}$.
    \item The horizontal composition of 1-cells is composition of functors, and the horizontal composition of 2-cells is the horizontal composition (or Godement product) of natural transformations (see \cite[Definition 1.1.8]{johnson20212}). That is, for functors $F,F'\in \mathrm{Fun}(\kk^{\calG_1},\kk^{\calG_2})$, $G,G'\in \mathrm{Fun}(\kk^{\calG_2},\kk^{\calG_3})$ and natural transformations $\alpha:F\ra F'$, $\beta:G\ra G'$, the horizontal composition of $\alpha$ and $\beta$ is the natural transformation $\alpha\star \beta:GF\ra G'F'$, consisting of $((\alpha\star\beta)_A)_{A\in \kk^{\calG_1}}$ that
    \[(\alpha\star\beta)_A=G(\alpha_A)\otimes \beta_{F'(A)}=\beta_{F(A)}\otimes G'(\alpha_A)\in \kk^{\calG_3}(GF(A),G'F'(A)).\]

    \item The composition of functors is associative. We can just let the associators be the identity natural transformations.
    \item The composition of functors is unital with respect to identity functors. We can just let the left and right unitors be the identity natural transformations.
\end{enumerate}

We see that the data above forms not only a bicategory but also a strict 2-category (i.e a bicategory where associators and unitors are identity natural transformations). The proof is exactly same as the case of bicategory of all small categories in example 2.3.14 of \cite{johnson20212}.
\begin{prop}
    The data above defines a strict 2-category $\mathfrak{KK}$.
\end{prop}

\subsection{Pseudofunctor from \texorpdfstring{$\frgr$}{frgr} to \texorpdfstring{$\mathfrak{KK}$}{KK}}

Motivated by the construction of induction functor defined by Miller in \cite{miller2024functors}, in this subsection we will prove that there is a pseudofunctor (see \cite[Definition 4.1.2]{johnson20212}) from $\frgr$ to $\mathfrak{KK}$, in which we send an \'etale groupoid $\calG$ to the Kasparov category $\kk^\calG$, and send a correspondence $\Omega:\calG\leftarrow \calH$ to a functor $\ind_\Omega:\kk^\calH\ra \kk^\calG$. To achieve that, we will
\begin{enumerate}
    \item recall the results of Miller in \cite{miller2024functors}, which covers all information concerns 0-cells and 1-cells;
    \item define and describe the construction on 2-cells, that is for a 2-cell $f\in \frgr(\calH,\calG)(\Omega_1,\Omega_2)$, and $A\in \kk^\calH$, we will construct $\ind_{f,A}\in \kk^\calG(\ind_{\Omega_1}A, \ind_{\Omega_2}A)$, such that $\ind_f=(\ind_{f,A})_{A\in \kk^\calH}$ is a natural transformation from $\ind_{\Omega_1}$ to $\ind_{\Omega_2}$ (proposition \ref{ind_f is natural transformation}), and 
    \[\frgr(\calH,\calG)\ra \mathrm{Fun}(\kk^\calH,\kk^\calG), \Omega\mapsto \ind_\Omega, f\mapsto \ind_f\]
    is a functor (proposition \ref{ind is functor});
    \item prove that there are well-defined lax functoriality constraint (proposition \ref{Lax functoriality constraint}) and lax unity constraint (proposition \ref{Lax unity constraint});
    \item prove the lax associativity (proposition \ref{lax associativity}), lax left unity and lax right unity (proposition \ref{lax unity}).
\end{enumerate}

\subsubsection{Results of Miller on 1-cells}

Let $\Omega:\calG\leftarrow \calH$ be a second countable locally compact Hausdorff correspondence. There is a cutoff function $c$ for $\Omega\rtimes \calH$. Suppose that $(E,V)$ is an $\calH$-Hilbert module over some $\calH$-\cst-algebra, recall that an adjointable operator in $\calL(\ind_\Omega E)$ is identified with an $\calH$-equivariant section in $\Gamma_b(\Omega,\sigma^*\calL(\ca E))$. For $T$ in $\calL(E)$ we define $\ind_{\Omega,c_\Omega}T\in \calL(\ind_\Omega E)$ such that
\[(\ind_{\Omega,c}T)(\omega)=\sum_{h\in \calH^{\sigma_\Omega(\omega)}}c_\Omega(\omega h)V_hT_{s(h)}V_h\inv\in \calL(E_{\sigma_\Omega(\omega)}). \]

Here we summarize some results about functoriality of induction functor of correspondence, proved by Miller in \cite{miller2024functors}, section 6 and 7.

\begin{prop}
    Let $\Omega:\calG\leftarrow \calH$ be a second countable \'etale groupoid correspondence, $A,B$ be in $\kk^{\calH}$ (that is, $A,B$ are separable $\calH$-\cst-algebras).
    \begin{enumerate}
        \item Let $c:\Omega\ra \mathbb R$ be a cutoff function for $\Omega\rtimes\calH$, $(E,\pi,T)$ is a Kasparov cycle in $\mathbb E^{\calH}(A,B)$, then $(\ind_\Omega E,\ind_\Omega \pi, \ind_{\Omega,c}T)$ is a Kasparov cycle in $\mathbb E^{\calG}(\ind_\Omega A,\ind_\Omega B)$.
        \item The map
        \[\kk^\calH(A,B)\ra\kk^{\calG}(\ind_\Omega A,\ind_\Omega B),\]
        \[[E,\pi,T]\mapsto [\ind_\Omega E,\ind_\Omega \pi, \ind_{\Omega,c}T]\]
        is a well-defined homomorphism, independent of choice of cutoff functions, respects Kasparov product and identity classes.
        \item The assignment $A \mapsto \ind_\Omega A$, $[E,\pi,T]\mapsto [\ind_\Omega E,\ind_\Omega \pi, \ind_{\Omega,c}T]$ gives a well-defined additive functor $\ind_\Omega\in \mathrm{Fun(}\kk^\calH,\kk^\calG)$.
        \item For $\calG,\calH,\calK\in B_0$, $\Omega\in \frgr(\calH,\calG)$, $\Lambda\in \frgr(\calK,\calH)$, $D\in \kk^\calK$, the following map
        \[\varphi_{\Omega,\Lambda, D}:\ind_\Omega\ind_\Lambda D\ra \ind_{\Omega\circ\Lambda}D,\xi\mapsto \varphi_{\Omega,\Lambda, D}(\xi),\]
        \[\varphi_{\Omega,\Lambda, D}(\xi)([\omega,\lambda])=\xi(\omega)(\lambda),\quad \forall [\omega,\lambda]\in \Omega\circ \Lambda\]
        is a $\calG$-equivariant isomorphism. So the induced element in KK-group 
        \[[\varphi_{\Omega,\Lambda,D}]\in \kk^\calG(\ind_\Omega\ind_\Lambda D,\ind_{\Omega\circ\Lambda}D)\] is invertible, i.e. an isomorphism in $\kk^\calG$.
        \item For $\calG,\calH,\calK\in B_0$, $\Omega\in \frgr(\calH,\calG)$, $\Lambda\in \frgr(\calK,\calH)$, there is a natural isomorphism $\varphi_{\Omega,\Lambda}:\ind_\Omega\ind_\Lambda\ra \ind_{\Omega\circ\Lambda}$, consisting of isomorphisms $[\varphi_{\Omega,\Lambda,D}]$ for each $D\in \kk^\calK$. That is, for any $D_1,D_2\in \kk^{\calK}$ and $x\in \kk^{\calK}(D_1,D_2)$, the following diagram (whose arrows are elements of KK-groups) is commutative.
        \[\xymatrix{
        \ind_\Omega\ind_\Lambda D_1 \ar[r]^{\ind_\Omega \ind_\Lambda (x) } \ar[d]_{\varphi_{\Omega,\Lambda,D_1}} & \ind_\Omega\ind_\Lambda D_2 \ar[d]^{\varphi_{\Omega,\Lambda,D_2}} \\
        \ind_{\Omega\circ \Lambda}D_1 \ar[r]_{\ind_{\Omega\circ \Lambda}(x)} & \ind_{\Omega\circ \Lambda} D_2
        }
        \]
    \end{enumerate}
\end{prop}

\subsubsection{Wrong way functoriality of equivariant local homeomorphisms}

We will give the remaining part of information of a pseudofunctor from $\frgr$ to $\mathfrak{KK}$. Firstly  $\Omega\mapsto \ind_\Omega$ need to be a covariant functor on $\Omega$. Motivated by \cite{ConnesSkandalis}, we give the following definition of the wrong way functoriality of an equivariant local homeomorphism (in comparison with the proposition 4.3 and 4.5 of \cite{ConnesSkandalis}).

Let $\calG$ be a second countable \'etale groupoid, $\rho:Z\ra Y$ be a $\calG$-equivariant local homeomorphism between two locally compact Hausdorff $\calG$-spaces. Let $\rho_Y,\rho_Z$ be the anchor map of $Y, Z$ respectively. There is a pre-Hilbert $C_0(Y)$-module structure on $C_c(Z)$, where right action of $C_0(Y)$ is defined as
\[(f\cdot g)(z)=f(z)g(\rho(z)), \quad\forall f\in C_c(Z),g\in C_0(Y),z\in Z,\]
the $C_0(Y)$-valued inner product is defined as
\[\langle f_1,f_2\rangle(y)=\sum_{z\in \rho\inv(y)}\overline{f_1(z)}f_2(z), \quad\forall f_1,f_2\in C_c(Z),y\in Y.\]
Let $L^2(Z,\rho)$ be its completion. We have action of $C_0(Z)$ on $L^2(Z,\rho)$ by multiplication
\[\pi_\rho:C_0(Z)\ra \calL_{C_0(Y)}(L^2(Z,\rho)),\pi_\rho(f_1)(f_2)=f_1f_2, \forall f_1\in C_0(Z), f_2\in C_c(Z).\]

\begin{lem}\label{local homeo acts like compact operator}
    Use the same notations as above, $\pi_\rho(C_0(Z))\subseteq \calK(L^2(Z,\rho))$.
\end{lem}
\begin{proof} Assume that $f\in C_c(Z,[0,1])$ is supported in some open set $U\subseteq Z$ such that $\rho|_U$ is a homeomorphism onto an open of $Y$, then for any $f_2\in C_c(Z)$,
\[\langle f^{\frac{1}{2}},f_2\rangle(y)=\sum_{z\in \rho\inv(y)}f^{\frac{1}{2}}(z)f_2(z)=\begin{cases}
    f^{\frac{1}{2}}(\rho|_{U}\inv(y))f_2(\rho|_{U}\inv(y)), & y\in \rho(U);\\
    0, & y\not\in \rho(U).
\end{cases}\]
When $z\not\in U$, $f^{\frac{1}{2}}(z)=0$ hence $[f^{\frac{1}{2}}\cdot\langle f^{\frac{1}{2}},f_2\rangle](z)=0=f(z)f_2(z)=\pi_\rho(f)(f_2)(z)$. When $z\in U$, we have
\[[f^{\frac{1}{2}}\cdot\langle f^{\frac{1}{2}},g\rangle](z)=f^{\frac{1}{2}}(z)f^{\frac{1}{2}}(\rho|_U\inv(\rho(z)))f_2(\rho|_U\inv(\rho(z)))=f(z)f_2(z)=\pi_\rho(f)(f_2)(z).\]
So $\pi_\rho(f)=f^{\frac{1}{2}}\cdot\langle f^{\frac{1}{2}},-\rangle$ is a rank-one operator. Every element of $C_c(Z)$ can be written as a finite complex linear combination of functions like this, hence $\pi_\rho(C_c(Z))$ consists of finite rank operators.\end{proof}

Now we need to show that $\pi_\rho$ is $\calG$-equivariant. For any $x\in \calG\units$, we write $Y_x:=\rho_Y\inv(x)$, $Z_x=\rho_Z\inv(x)$. The fiber of $L^2(Z,\rho)$ at $x$ can be identified with $L^2(Y_x, \rho|_{Y_x})$. For every $\gamma\in \calG$, let $V_\gamma\in \calL(L^2(Y_{s(\gamma)}, \rho|_{Y_{s(\gamma)}}), L^2(Y_{r(\gamma)}, \rho|_{Y_{r(\gamma)}}))$ be the unitary isomorphism defined by left translation
$f\mapsto f(\gamma\inv-)$.

The following lemma is an analog of \cite[Proposition 4]{Paterson2009TheST}.

\begin{lem}
    Use the same notation as above, there exists a unitary isomorphism $V:s^*L^2(Z,\rho)\ra r^*L^2(Z,\rho)$, such that for any $\gamma\in \calG$, $V_\gamma$ is the fiber of $V$ at $\gamma$. Moreover, $L^2(Z,\rho)$ is a $\calG$-Hilbert $C_0(Y)$-module, $\pi_\rho$ is a $\calG$-equivariant non-degenerate representation.
\end{lem}
\begin{proof} Let $\ca E:=\sqcup_{x\in \calG\units}L^2(Z_x, \rho|_{Z_x})$ be the associated Hilbert bundle of $L^2(Z,\rho)$, by lemma \ref{continuity of groupoid action of Hilbert module}, it suffices to show that for every $f\in C_c(Z)$ and a converging net $(\gamma_\lambda)_\lambda$ with limit $\gamma$ in $\calG$, $V_{\gamma_\lambda}(f|_{Z_{s(\gamma_\lambda)}})$ converges to $V_\gamma(f|_{Z_{s(\gamma)}})$ in $\ca E$. By Tietze's extension theorem, there exists $g\in C_c(Z)$ such that $g|_{Z_{r(\gamma)}}=V_\gamma(f|_{Z_{s(\gamma)}})=f|_{Z_{s(\gamma)}}(\gamma\inv-)$.

Claim: $\|g|_{Z_{r(\gamma_\lambda)}}-V_{\gamma_\lambda}(f|_{Z_{s(\gamma_\lambda)}})\|_{L^2(Z_{r(\gamma_\lambda)},\rho|_{Z_{r(\gamma_\lambda)}})}\ra 0$.

Proof of the claim: let $D$ be a compact neighborhood of $\gamma$, since $\gamma_\lambda$ will eventually be in $D$, we can assume that $D$ be a compact subset of $\calG$ containing $\{\gamma_\lambda\}_\lambda$ after replacing by a subnet. Let $C=Dsupp(f)\cup supp(g)\subseteq Z$, which is compact, hence $M=\sup_{y\in Y}\#(\rho\inv(y)\cap C)$ is finite. This implies that
\begin{align*}
    \|g|_{Z_{r(\gamma_\lambda)}}-V_{\gamma_\lambda}(f|_{Z_{s(\gamma_\lambda)}})\|^2_{L^2(Z_{r(\gamma_\lambda)},\rho|_{Z_{r(\gamma_\lambda)}})} & =\sup_{y\in Y_{r(\gamma_\lambda)}}\sum_{z\in\rho\inv(y)}|g(z)-f(\gamma_\lambda\inv z)|^2\\
    & \leqslant M\cdot \sup_{z\in Z_{r(\gamma_\lambda)}}|g(z)-f(\gamma_\lambda\inv z)|^2.
\end{align*}

Let $\ca A=\sqcup_{x\in \calG\units}C_0(Z_x)$ be the associated \cst-bundle of $C_0(Z)$, the continuity of the action on $C_0(Z)$ implies that $f|_{Z_{s(\gamma_\lambda)}}(\gamma_\lambda\inv-)$ converges to $f|_{Z_{s(\gamma)}}(\gamma\inv-)=g|_{Z_{r(\gamma)}}$ in $\ca A$. So $g|_{Z_{r(\gamma_\lambda)}}-f|_{Z_{s(\gamma_\lambda)}}(\gamma_\lambda\inv-)$ converges to 0 in $\ca A_{r(\gamma)}=C_0(Z_{r(\gamma)})$. By \cite[Lemma C.18]{williams2007crossed}, we have $\|g|_{Z_{r(\gamma_\lambda)}}-f|_{Z_{s(\gamma_\lambda)}}(\gamma_\lambda\inv-)\|_\infty\ra 0$. Then using the inequality above, we proved our claim.

Using proposition \ref{convergence in banach bundle}, we proved that $V_{\gamma_\lambda}(f|_{Z_{s(\gamma_\lambda)}})$ converges to $V_\gamma(f|_{Z_{s(\gamma)}})$ in $\ca E$. Therefore, the action of $\calG$ on $\ca E$ is continuous, $L^2(Z,\rho)$ is a $\calG$-Hilbert $C_0(Y)$-module.

Clearly $C_c(Z)\subseteq \pi_\rho(C_0(Z))(L^2(Z,\rho))$, so $\pi_\rho$ is non-degenerate.

Now for any $f_1\in C_0(Z_{s(\gamma)})$, $f_2\in C_c(Z_{s(\gamma)})\subseteq L^2(Z_{s(\gamma)},\rho|_{Z_{s(\gamma)}})$,
\[(\pi_{\rho,r(\gamma)}(f_1(\gamma\inv-))\circ V_\gamma)(f_2)=f_1(\gamma\inv-)f_2(\gamma\inv-)=(\pi_{\rho,s(\gamma)}(f_1)(f_2))(\gamma\inv-).\]
So $\pi_\rho$ is $\calG$-equivariant.\end{proof}

Recall that, for a locally compact Hausdorff groupoid $\calG$ and a locally compact Hausdorff $\calG$-space $Z$, there is naturally a forgetful functor $\kk^{\calG\ltimes Z}\ra \kk^\calG$. In this section, we make a convention that, if $x$ is an element of $\kk^{\calG\ltimes Z}(A,B)$, we still use $x$ to denote its image in $\kk^\calG(A,B)$.

\begin{defn}
    Use the same notation as above, we define 
    \[\rho!:=[L^2(Z,\rho),\pi_\rho,0]\in \kk^{\calG\ltimes Y}(C_0(Z),C_0(Y)).\]
    By abuse of language, we also see $\rho!$ as an element of $\kk^\calG(C_0(Z),C_0(Y))$.
\end{defn}

\begin{prop}\label{wrong way functoriality}
Let $\calG$ be an \'etale groupoid, let $f:Y\ra X$ and $g:Z\ra Y$ be $\calG$-equivariant local homeomorphisms between $\sigma$-compact locally compact Hausdorff $\calG$-spaces.
    \begin{enumerate}
        \item $g!\otimes f!=(f\circ g)!$ in $\kk^\calG(C_0(Z),C_0(X))$.
        \item When $f:Y\ra X$ is an open inclusion, $f!$ is induced by the  extension by zero *-homomorphism $\iota: C_0(Y)\hookrightarrow C_0(X)$.
        \item When $f$ is a homeomorphism, $f!$ is induced by the *-isomorphism 
        \[(f\inv)^*:C_0(Y)\ra C_0(X), \phi\mapsto \phi\circ f\inv.\]
    \end{enumerate}
\end{prop}

\begin{proof}
(1) For $h\in C_c(Z)$, $h'\in C_c(Y)$, we denote its image in the internal tensor product $L^2(Z,g)\otimes_{\pi_f}L^2(Y,f)$ by $h\otimes h'$. For the convenience of writing, let $E=L^2(Z,g)$, $F=L^2(Y,f)$.
Let $V=span\{h\otimes h':h\in C_c(Z),h'\in C_c(Y)\}$, it is a sub $C_0(X)$-module of $E\otimes_{\pi_f}F$. And we have a well-defined linear map
\[\Phi:V\ra L^2(Z,f\circ g),\Phi(h\otimes h')(z)=h(z)h'(g(z)).\]

For any $h_1,h_2\in C_c(Z)$ and $h_1',h_2'\in C_c(Y)$, 
\begin{align*}
    \langle h_1\otimes h_1',h_2\otimes h_2'\rangle_{E\otimes_{\pi_f}F}(x) & =\langle h_1',\pi_f(\langle h_1,h_2\rangle_E)h_2'\rangle_F(x)\\
&  = \sum_{y\in f\inv(x)}\overline{h_1'(y)}(\sum_{z\in g\inv(y)}\overline{h_1(z)}h_2(z))h_2'(y)\\
& = \sum_{z\in g\inv(f\inv(x))}\overline{h_1(z)h_1'(g(z))}h_2(z)h_2'(g(z))\\
& = \langle \Phi(h_1\otimes h_1'),\Phi(h_2\otimes h_2')\rangle_{L^2(Z,f\circ g)}
\end{align*}

Now for any $h\in C_c(Z)\subseteq L^2(Z,f\circ g)$, by Urysohn's lemma, there exists $\phi\in C_c(Y)$ such that $\phi|_{g(supp(h))}=1$. Hence, for any $z\in Z$, $h(z)=h(z)\phi(g(z))$. Therefore, $h=\Phi(h\otimes \phi)$. So we proved that $C_c(Z)\subseteq \Phi(V)$, $\Phi(V)$ is dense in $L^2(Z,f\circ g)$.

By lemma \ref{to be a unitary}, $\Phi$ can be extended to unitary isomorphism. By abuse of language, we denote it as $\Phi:E\otimes_{\pi_f}F\ra L^2(Z,f\circ g)$.

Moreover, for any $\psi\in C_0(Z)$, $h\in C_c(Z)$ and $h'\in C_c(Y)$, $z\in Z$,
\begin{align*}
    [\Phi \circ (\pi_g\otimes id)(\psi)](h\otimes h')(z) & = \Phi((\psi h)\otimes h')(z)\\
    & = \psi(z) h(z) h'(g(z))\\
    & = \pi_{f\circ g}(\psi)(\Phi(h\otimes h'))(z).
\end{align*}
That is $\Phi\circ (\pi_g\otimes id(-))=\pi_{f\circ g}(-)\circ \Phi$. It is easy to check that $\Phi$ is $C_0(X)$-linear. It is easy to check that $\Phi$ is $\calG$-equivariant if we apply lemma \ref{to be equivariant adjointable op} to the pre-$\calG$-Hilbert $C_0(X)$-module $V$. So $(E\otimes_{\pi_f}F,\pi_g\otimes id)$ and $(L^2(Z,f\circ g),\pi_{f\circ g})$ are $\calG$-equivariantly unitarily equivalent. Hence,
\[g!\otimes f!=[E\otimes_{\pi_f}F,\pi_g\otimes id,0]=[L^2(Z,f\circ g),\pi_{f\circ g},0]=(f\circ g)!.\]

(2) Without loss of generality, we can see $Y$ as an open of $X$, $C_0(Y)$ is an ideal of $C_0(X)$, $L^2(Y,f)$ is $C_0(Y)$ itself. Let $\iota: C_0(Y)\ra C_0(X)$ be the inclusion of extension by zero. By corollary \ref{kk-equiv for inclusion},
\[f!=[C_0(Y),id,0]=[\iota]\in \kk^{\calG\ltimes X}(C_0(Y),C_0(X)).\]

(3) is a special case of (2).\end{proof}

Now assume that $\calG,\calH\in B_0$, $\Omega_1,\Omega_2\in \frgr(\calH,\calG)$, $f\in \frgr(\calH,\calG)(\Omega_1,\Omega_2)$, i.e. $f$ is a $\calG,\calH$-equivariant continuous map $\Omega_1\ra\Omega_2$.
Our aim is to construct the corresponding 2-cell in $\mathfrak{KK}$ for the 2-cell $f$. Remark that if $\Omega:\calG\leftarrow\calH$ is a correspondence and $A$ is an $\calH$-\cst-algebra, $\ind_\Omega A$ can be seen as a $\calG\ltimes (\Omega/\calH)$-algebra after identified with $(\sigma_\Omega^*A)^{\Omega\rtimes\calH}$ as mentioned in the remark \ref{remark on ind_c}.

\begin{lem}\label{identify pullback of induction}
    Use the same notation as above, let $A$ be an $\calH$-\cst-algebra, then the \cst-algebra $\bar f^*(\ind_{\Omega_2}A)\cong C_0(\Omega_1/\calH)\otimes _{\Omega_2/\calH}\ind_{\Omega_2}A$ is isomorphic to $\ind_{\Omega_1}A$ as $\calG\ltimes (\Omega_1/\calH)$-\cst-algebra, through the *-isomorphism
     \[\Upsilon_{f,A}:C_0(\Omega_1/\calH)\otimes _{\Omega_2/\calH}\ind_{\Omega_2}A\ra \ind_{\Omega_1}A,\]
     which is defined by
    \[\Upsilon_{f,A}(\phi\otimes \eta)(\omega_1)=\phi(\omega_1\calH)\eta(f(\omega_1)).\]
\end{lem}

\begin{proof} 
It is easy to check the $C_0(\Omega_1/\calH)$-linearity of $\Upsilon_{f,A}$. For every $\omega_1\calH\in \Omega_1/\calH$, the fiber of $\ind_{\Omega_1}A$ at $\omega_1\calH$ can be canonically identified with $A_{\sigma_{\Omega_1}(\omega_1)}$ (see \cite[Proposition 3.14]{bonicke2020going}), while the fiber of $\bar f^*(\ind_{\Omega_2}A)$ at $\omega_1\calH$ can be identified with the fiber of $\ind_{\Omega_2}A$ at $\bar f(\omega_1\calH)=f(\omega_1)\calH$, which is also $A_{\sigma_{\Omega_1}(\omega_1)}$ since $\sigma_{\Omega_1}(\omega_1)=\sigma_{\Omega_2}(f(\omega_1))$, and the fiber of $\Upsilon_{f,A}$ at $\omega_1\calH$ is identified with the identity map. So $\Upsilon_{f,A}$ is an isomorphism.

After pushout through $\bar\rho_{\Omega_1}:\Omega_1/\calH\ra \calG\units$, we see $\Upsilon_{f,A}$ also as a $C_0(\calG\units)$-linear map. Then for any $\gamma\in \calG$, $\phi\in C_0(\Omega^{s(\gamma)}/\calH)$ and $\eta\in \ind_{\Omega_2}A$,
\begin{align*}
    (\gamma.\Upsilon_{f,A, s(\gamma)}(\phi\otimes \eta))(\omega) & =\Upsilon_{f,A,s(\gamma)}(\phi\otimes \eta)(\gamma\inv\omega)\\
    & = \phi(\gamma\inv\omega\calH)\eta(f(\gamma\inv\omega))\\
    & = (\gamma.\phi)(\omega\calH)\cdot(\gamma.\eta)(f(\omega))\\
    & = \Upsilon_{f,A,r(\gamma)}(\gamma.(\phi\otimes \eta)).
\end{align*}
So $\Upsilon_{f,A}$ is also $\calG$-equivariant. Hence, $\Upsilon_{f,A}$ is $\calG\ltimes(\Omega_1/\calH)$-equivariant.
\end{proof}

\subsubsection{Construction on 2-cells}

\begin{defn}\label{def_wrong way functoriality}
    For a 2-cell $f\in \frgr(\calH,\calG)(\Omega_1,\Omega_2)$ defined as above (i.e. $f:\Omega_1\ra \Omega_2$ is a bi-equivariant continuous map between two correspondences), for any $A\in \kk^\calH$, we define
    \[\ind_{f,A}=[\Upsilon_{f,A}\inv]\otimes\tau_{\ind_{\Omega_2}A}^{\calG\ltimes(\Omega_2/\calH)}(\bar f!)\in \kk^{\calG\ltimes(\Omega_2/\calH)}(\ind_{\Omega_1}A,\ind_{\Omega_2}A).\]
    By abuse of language, we will also denote its image in $\kk^\calG(\ind_{\Omega_1}A,\ind_{\Omega_2}A)$ as $\ind_{f,A}$.

\[\xymatrix{
\ind_{\Omega_1}A \ar[r]^{\hspace{-1.5cm}\Upsilon_{f,A}\inv} \ar[dr]_{\ind_{f,A}} & C_0(\Omega_1/\calH)\otimes_{\Omega_2/\calH}\ind_{\Omega_2}A \ar[d]^{\tau(\bar f!)} \\
& \ind_{\Omega_2}A
}\]
\end{defn}

We will have a more concrete construction. Before that, we need the following lemmas.

\begin{lem}\label{trivial unitary isomorphism}
    Let $\rho:Z\ra Y$ be a local homeomorphism between two locally compact Hausdorff spaces. Assume that $A$ is a $C_0(Y)$-algebra. On $\Gamma_c(Z,\rho^*\ca A)$, we can define a pre-Hilbert $A$-module structure as
    \[(\xi\cdot a)(z)=\xi(z)a(\rho(z)),\forall \xi\in \Gamma_c(Z,\rho^*\ca A),a\in A,z\in Z,\]
    \[\langle \xi_1,\xi_2\rangle(y)=\sum_{z\in \rho\inv(y)}\xi_1(z)^*\xi_2(z), \forall \xi_1,\xi_2\in \Gamma_c(Z,\rho^*\ca A).\]
    We define its completion as $L^2(Z,\rho,A)$. Then $L^2(Z,\rho,A)$ is unitarily isomorphic to $L^2(Z,\rho)\otimes_Y A$.
\end{lem}
\begin{proof}
    For any $\phi\in C_c(Z)$, $a\in A$, we define $\Phi(\phi\otimes a)\in \Gamma_c(Z,\rho^*\ca A)$ as $z\mapsto \phi(z)a(\rho(z))$.

    Now for any $\phi_1,\phi_2\in C_c(Z)$ and $a_1,a_2\in A$,
    \begin{align*}
        \langle \Phi(\phi_1\otimes a_1),\Phi(\phi_2\otimes a_2)\rangle_{L^2(Z,\rho,A)}(y) & = \sum_{z\in \rho\inv(y)}(\Phi(\phi_1\otimes a_1)(z))^*\Phi(\phi_2\otimes a_2)(z)\\
        & = \sum_{z\in \rho\inv(y)}\phi_1(z)^*a_1(y)^*\phi_2(z)a_2(y)\\
        & = \langle \phi_1,\phi_2\rangle_{L^2(Z,\rho)}a_1(y)^*a_2(y)\\
        & = \langle \phi_1\otimes a_1,\phi_2\otimes a_2\rangle_{L^2(Z,\rho)\otimes_Y A}.
    \end{align*}
    Hence, $\Phi$ is a well-defined isometric map.

    Claim: $\Gamma_c(Z,\rho^*\ca A)$ is contained in the image of $\Phi$. For any $\xi\in \Gamma_c(Z,\rho^*\ca A)$, there are finitely many opens $V_1,\cdots, V_n$ of $Z$, such that for any $1\leqslant i\leqslant n$, $\rho|_{V_i}$ is a homeomorphism onto an open of $Y$. Let $(\phi_i)_{i=1}^n$ be a family of functions in $C_c(Z,[0,1])$ such that $supp(\phi_i)\subseteq V_i$ and $\sum_{i=1}^n \phi_i(z)=1$ for any $z\in supp(\xi)$. So $\phi_i^{\frac{1}{2}}\cdot\xi$ is a continuous section in $\Gamma_c(V_i,\rho^*\ca A)$.
    
    For convenience of writing, we make the following convention: if $U$ is an open of $Z$ such that $\rho|_U$ is a homeomorphism onto an open of $Y$, and if $\eta\in \Gamma_c(Z,\rho^*\ca A)$ or $f\in C_c(Z)$ such that $supp(\eta)\subseteq U$, $supp(f)\subseteq U$, let $\eta\circ \rho|_U\inv\in \Gamma_c(\rho(U),\ca A)$ be seen as an element of $\Gamma_c(Y,\ca A)$ through extension by zero, let $f\circ \rho|_U\inv\in C_c(\rho(U))$ be seen as an element of $C_c(Y)$ through extension by zero.
    
    Let $\Psi(\xi)=\sum_{i=1}^n \phi_i^{\frac{1}{2}}\otimes ((\phi_i^{\frac{1}{2}}\cdot \xi)\circ\rho|_{V_i}\inv)\in span\{\phi\otimes a:\phi\in C_c(Z),a\in A\}$.

    For any $z\in Z$,
    \[\Phi(\Psi(\xi))(z)=\sum_{i=1}^n \phi_i(z)^\frac{1}{2}(\phi_i^{\frac{1}{2}}\circ\rho|_{V_i}\inv)(\rho(z))(\xi\circ\rho|_{V_i}\inv)(\rho(z))=\sum_{i=1}^n \phi_i(z)\xi(z)=\xi(z),\]
    therefore $\xi=\Phi(\Psi(\xi))$ is in the image of $\Phi$. Now we can use lemma \ref{to be a unitary}, $\Phi$ extends to a unitary between $L^2(Z,\rho)\otimes_Y A$ and $L^2(Z,\rho,A)$. 
\end{proof}

\begin{lem}\label{well defined pre hilbmod}
    Use the same notation as the definition \ref{def_wrong way functoriality}, let $\ind_{\Omega_2}A$ acts on the right of $\ind_{\Omega_1,c}A$ by
    \[(\xi\cdot \eta) (\omega_1)=\xi(\omega_1)\eta(f(\omega_1)), \quad \forall \xi\in \ind_{\Omega_1,c}A,\eta\in \ind_{\Omega_2}A,\omega_1\in \Omega_1,\]
    and for any $\xi_1,\xi_2\in \ind_{\Omega_1,c}A$, we define
    \[\langle \xi_1,\xi_2\rangle(\omega')=\sum_{\omega\in f\inv(\omega')}\xi_1(\omega)^*\xi_2(\omega),\quad \forall \xi_1,\xi_2\in \ind_{\Omega_1,c}A,\omega'\in \Omega_2.\]
    Then
    \begin{enumerate}
        \item for any $\xi_1,\xi_2\in \ind_{\Omega_1,c}A$, $\omega\mapsto \langle \xi_1,\xi_2\rangle(\omega)$ is a well-defined element of $\ind_{\Omega_2,c}A$;
        \item the above data gives a pre-Hilbert $\ind_{\Omega_2}A$-module structure on $\ind_{\Omega_1,c}A$.
    \end{enumerate}
\end{lem}
\begin{proof}
    (1) For any $\xi_1,\xi_2\in \ind_{\Omega_1,c}A$, $supp(\xi_1)$ is a compact subset of $\Omega_1/\calH$. Use lemma \ref{basic property of 2-cells}, for any $\omega'\in \Omega_2$, $f\inv(\omega')\cap q\inv(supp(\xi_1))$ is in bijection with $\bar f\inv(\omega'\calH)\cap supp(\xi_1)$, which is a finite set. Hence, the sum is always a finite sum. Apply lemma \ref{sum of fiber} to the $C_0(\Omega_2/\calH)$-algebra $\ind_{\Omega_2}A$, the local homeomorphism $\bar f:\Omega_1/\calH\ra \Omega_2/\calH$ and notice that $\bar f^*\ind_{\Omega_2}A\cong \ind_{\Omega_1}A$, we can prove that $\langle \xi_1,\xi_2\rangle$ is a well-defined element of $\ind_{\Omega_2,c}A$. (2) is now easy to check.
\end{proof}

\begin{defn}
    Use the same notation as the definition \ref{def_wrong way functoriality}, we define the Hilbert $\ind_{\Omega_2}A$-module $E_{f,A}$ as the completion of the pre-Hilbert $\ind_{\Omega_2}A$-module $\ind_{\Omega_1,c}A$ as mention in lemma \ref{well defined pre hilbmod}.
\end{defn}

\begin{prop}
    Use the same notation as above. For any $\gamma\in \calG$, let $V_\gamma:(E_{f,A})_{s(\gamma)}\ra(E_{f,A})_{r(\gamma)}$ be the adjointable operator extended from
\[\ind_{\Omega^{s(\gamma)},c}A\ra \ind_{\Omega^{r(\gamma)},c}A, \xi\mapsto \xi(\gamma\inv-).\]
For any $\xi\in \ind_{\Omega_1}A$, we can define an adjointable operator $\pi_{f,A}(\xi)\in \calL(E_{f,A})$ as
\[\pi_{f,A}(\xi)(\eta)=\xi\eta,\quad \forall \eta\in \ind_{\Omega_1,c}A.\]
Then we have
\begin{enumerate}
    \item the bimodule $(E_{f,A},\pi_{f,A})$ is unitarily equivalent to \[(L^2(\Omega_1/\calH,\bar f)\otimes_{\Omega_2/\calH}\ind_{\Omega_2}A,(\pi_{\bar f}\otimes id)\circ \Upsilon_{f,A}\inv).\]
    \item $(V_\gamma)_{\gamma\in \calG}$ coincides with the action of $\calG$ on $L^2(\Omega_1/\calH,\bar f)\otimes_{\Omega_2/\calH}\ind_{\Omega_2}A$ through this unitary equivalence. So $(V_\gamma)_\gamma$ make $E_{f,A}$ a $\calG$-Hilbert $\ind_{\Omega_2}A$-module, and $\pi_{f,A}$ is $\calG$-equivariant.
    \item $\pi_{f,A}(\ind_{\Omega_1}A)\subseteq \ca K(E_{f,A})$.
    \item $\ind_{f,A}=[E_{f,A},\pi_{f,A},0]\in \kk^\calG(\ind_{\Omega_1}A,\ind_{\Omega_2}A)$.
\end{enumerate}
\end{prop}
\begin{proof}

    (1) The Hilbert $\ind_{\Omega_2}A$-module $E_{f,A}$ can be canonically identified with $L^2(\Omega_1/\calH,\overline{f}, \ind_{\Omega_2}A)$ as defined in lemma \ref{trivial unitary isomorphism}, where $\ind_{\Omega_2}A$ is seen as a $C_0(\Omega_2/\calH)$-algebra. By lemma \ref{trivial unitary isomorphism}, the map
    \[\Phi:span\{\phi\otimes \eta:\phi\in C_c(\Omega_1/\calH),\eta\in \ind_{\Omega_2}A\}\ra \ind_{\Omega_1,c}A,\]
    \[\Phi(\phi\otimes \eta)(\omega_1)=\phi(\omega_1\calH)\eta(f(\omega_1))\]
    extend to a unitary isomorphism $\Phi\in \calL(L^2(\Omega_1/\calH,\overline{f})\otimes_{\Omega_2/\calH}\ind_{\Omega_2}A, E_{f,A})$.

    Now for any $\phi,\psi\in C_c(\Omega_1/\calH)$, $\eta,\xi\in \ind_{\Omega_2}A$, for any $\omega_1\in \Omega_1$,
    \begin{align*}
        [\pi_{f,A}(\Upsilon_{f,A}(\phi\otimes \eta))\circ \Phi](\psi\otimes\xi)(\omega_1) & = \phi(\omega_1\calH)\psi(\omega_1\calH)\xi(f(\omega_1))\eta(f(\omega_1))\\
        & = \Phi((\phi\psi)\otimes (\eta\xi))(\omega_1)\\
        & = [\Phi\circ (\pi_{\bar f}\otimes id)(\phi\otimes \eta)](\psi\otimes \xi)(\omega_1).
    \end{align*}
    That is, $\pi_{f,A}(\Upsilon_{f,A}(-))\circ \Phi=\Phi\circ (\pi_{\bar f}\otimes id)$, therefore $\Phi^*\circ \pi_{f,A}(-)\circ\Phi=(\pi_{\bar f}\otimes id)\circ \Upsilon_{f,A}\inv(-)$.

    (2) If $\phi\in C_c(\Omega_1/\calH)$, $\eta\in \ind_{\Omega_2}A$, we have
    \begin{align*}
    V_\gamma(\Phi(\phi\otimes\eta)|_{\Omega_1^{s(\gamma)}}) & = \Phi(\phi\otimes\eta)|_{\Omega_1^{s(\gamma)}}(\gamma\inv-)\\
        & = \phi|_{\Omega_1^{s(\gamma)}/\calH}(\gamma\inv-)\eta|_{\Omega_2^{s(\gamma)}}(\gamma\inv-)\\
        & = \Phi_{r(\gamma)}(\gamma.(\phi\otimes \eta)_{s(\gamma)})
    \end{align*}
    Hence, $(\Phi^*_{r(\gamma)}\circ V_\gamma\circ \Phi_{s(\gamma)})_{\gamma\in \calG}$ coincides with the action of $\calG$ on $L^2(\Omega_1/\calH,\bar f)\otimes_{\Omega_2/\calH}\ind_{\Omega_2}A$. It is easy to check that $\pi_{f,A}$ is $\calG$-equivariant if we apply lemma \ref{to be equivariant representation} to the pre-$\calG$-Hilbert $\ind_{\Omega_2}A$-module $\ind_{\Omega_1,c}A$.

    (3) Similar to lemma \ref{local homeo acts like compact operator}, if $\xi\in \ind_{\Omega_1}A$ such that $supp(\xi)$ is contained in an open $V$ of $\Omega_1/\calH$ on which $\overline{f}|_V$ is a homeomorphism onto an open of $\Omega_2/\calH$, then $\pi_{f,A}(\xi)$ is a rank 1 operator. (4) follows directly from (2) and (3).
\end{proof}

\begin{rem}\label{special case of ind 2-cells}
    When $f:\Omega_1\ra \Omega_2$ is an open inclusion, \[\iota_A:\ind_{\Omega_1}A\ra \ind_{\Omega_2}A,\]
\[\iota_A(\xi)(\omega_2) =\begin{cases}
    \xi(f\inv(\omega_2)) & \omega_2\in f(\Omega_1),\\
    0 & else
\end{cases}\]
is a well-defined $\calG$-equivariant inclusion, and $\ind_{f,A}=[\iota_A]$. Specifically, when $f$ is a homeomorphism, $\ind_{f,A}$ is an element induced by the $\calG$-equivariant *-isomorphism
\[(f\inv)^*:\ind_{\Omega_1}A\ra \ind_{\Omega_2}A, \xi\mapsto \xi\circ f\inv.\]
\end{rem}

Now we check that for a 2-cell $f$ in $\frgr$, $A\mapsto \ind_{f,A}$ a well-defined natural transformation.

\begin{prop}\label{ind_f is natural transformation}
    For $\calG,\calH\in B_0$, $\Omega_1,\Omega_2\in \frgr(\calH,\calG)$, $f\in \frgr(\calH,\calG)(\Omega_1,\Omega_2)$, i.e. $f$ is a $\calG,\calH$-equivariant continuous map $\Omega_1\ra\Omega_2$, there is a natural transformation $\ind_f:\ind_{\Omega_1}\ra \ind_{\Omega_2}$ consisting of morphisms \[\ind_{f,A}\in \kk^{\calG}(\ind_{\Omega_1}A,\ind_{\Omega_2}A)\]
    for each $A\in \kk^{\calH}$.
\end{prop}
\begin{proof} It suffices to prove that, for any $A,B\in \kk^\calH$, $x\in \kk^\calH(A,B)$, 
\[\ind_{\Omega_1}(x)\otimes \ind_{f,B}=\ind_{f,A}\otimes \ind_{\Omega_2}(x).\]
Use the decomposition property (theorem A 2.2 of \cite{lafforgue2007k}), it suffices to prove the case that $x=[\varphi]$, where $\varphi:A\ra B$ is an $\calH$-equivariant *-homomorphism.

By proposition \ref{functoriality of tau}, after identifying $id_{C_0(\Omega_2/\calH)}\otimes \ind_{\Omega_2}(\varphi)$ with $\ind_{\Omega_2}(\varphi)$, we have
\begin{equation}\label{eqn1}
    \tau_{\ind_{\Omega_2}A}^{\calG\ltimes(\Omega_2/\calH)}(\bar f!)\otimes [\ind_{\Omega_2}(\varphi)]=[id_{C_0(\Omega_1/\calH)}\otimes \ind_{\Omega_1}(\varphi)]\otimes \tau_{\ind_{\Omega_2}B}^{\calG\ltimes(\Omega_2/\calH)}(\bar f!)
\end{equation}
in $\kk^{\calG\ltimes (\Omega_2/\calH)}$, and therefore in $\kk^\calG$.

It is easy to check that the following diagram of $\calG$-equivariant *-homomorphism commutes.
\[\xymatrix{
C_0(\Omega_1/\calH)\otimes_{\Omega_2/\calH}\ind_{\Omega_2}A \ar[r]^{id\otimes \ind_{\Omega_1}(\varphi)} \ar[d]^{\Upsilon_{f,A}} & C_0(\Omega_1/\calH)\otimes_{\Omega_2/\calH}\ind_{\Omega_2}B \ar[d]^{\Upsilon_{f,B}}\\
\ind_{\Omega_1}A \ar[r]^{\ind_{\Omega_1}(\varphi)} & \ind_{\Omega_1}B
}\]

So by formula (\ref{eqn1}),
\begin{align*}
    \ind_{f,A}\otimes [\ind_{\Omega_2}(\varphi)] & = [\Upsilon_{f,A}\inv]\otimes \tau_{\ind_{\Omega_2}A}^{\calG\ltimes(\Omega_2/\calH)}(\bar f!)\otimes [\ind_{\Omega_2}(\varphi)]\\
    & = [\Upsilon_{f,A}\inv]\otimes[id_{C_0(\Omega_1/\calH)}\otimes \ind_{\Omega_1}(\varphi)]\otimes \tau_{\ind_{\Omega_2}B}^{\calG\ltimes(\Omega_2/\calH)}(\bar f!)\\
    & = [\ind_{\Omega_1}(\varphi)]\otimes [\Upsilon_{f,B}\inv]\otimes \tau_{\ind_{\Omega_2}B}^{\calG\ltimes(\Omega_2/\calH)}(\bar f!)\\
    & = [\ind_{\Omega_1}(\varphi)]\otimes \ind_{f,B}.
\end{align*}
\end{proof}

Then we give the local functor.

\begin{prop}\label{ind is functor}
    For $\calG,\calH\in B_0$, the assignment on objects
    \[\frgr(\calH,\calG)\ra \mathrm{Fun(\kk^\calH,\kk^\calG)},\Omega\ra \ind_\Omega,\]
    and the assignment on morphisms
    \[\frgr(\calH,\calG)(\Omega_1,\Omega_2)\ra \mathrm{Nat}(\ind_{\Omega_1},\ind_{\Omega_2}),f\mapsto\ind_f\]
    define a functor
    \[\ind:\frgr(\calH,\calG)\ra \mathrm{Fun(\kk^\calH,\kk^\calG)}.\]
\end{prop}
\begin{proof} Suppose that $\Omega_1,\Omega_2,\Omega_3$ are second countable locally compact Hausdorff correspondences $\calG\leftarrow \calH$, $f:\Omega_1\ra \Omega_2$ and $g:\Omega_2\ra\Omega_3$ are $\calG,\calH$-equivariant continuous maps. For the functoriality, it suffices to prove that, for any $A\in \kk^{\calH}$,
\[\ind_{g\circ f,A}=\ind_{f,A}\otimes \ind_{g,A},\ind_{id_{\Omega_1},A}=1_{\ind_{\Omega_1}A}.\]

It is easy to check that the following diagram of $\calG$-equivariant *-isomorphisms commute,
\[\xymatrix{
C_0(\Omega_1/\calH)\otimes_{\Omega_2/\calH} (C_0(\Omega_2/\calH)\otimes_{\Omega_3/\calH} \ind_{\Omega_3}A) \ar[r]^{\hspace{1cm}h} \ar[d]^{id\otimes \Upsilon_{g,A}} & C_0(\Omega_1/\calH)\otimes_{\Omega_3/\calH} \ind_{\Omega_3}A \ar[d]^{\Upsilon_{g\circ f,A}} \\
C_0(\Omega_1/\calH)\otimes_{\Omega_2/\calH} \ind_{\Omega_2}A \ar[r]^{\Upsilon_{f,A}} & \ind_{\Omega_1}A 
}\]
where $h$ is the canonical isomorphism.

Apply proposition \ref{functoriality of tau} to the $\calG\ltimes(\Omega_2/\calH)$ *-isomorphism 
\[\Upsilon_{g,A}: C_0(\Omega_2/\calH)\otimes_{\Omega_3/\calH}\ind_{\Omega_3}A \ra \ind_{\Omega_2}A
\]
and $\bar f!\in \kk^{\calG\ltimes (\Omega_2/\calH)}(C_0(\Omega_1/\calH),C_0(\Omega_2/\calH))$, we have
\[\tau^{\calG\ltimes(\Omega_2/\calH)}_{C_0(\Omega_2/\calH)\otimes_{\Omega_3/\calH}\ind_{\Omega_3}A}(\bar f!)\otimes [\Upsilon_{g,A}]=[id_{C_0(\Omega_1/\calH)}\otimes \Upsilon_{g,A}]\otimes \tau^{\calG\ltimes(\Omega_2/\calH)}_{\ind_{\Omega_2}A}(\bar f!)\]
in $\kk^{\calG\ltimes(\Omega_2/\calH)}$ and therefore in $\kk^\calG$.

We can identify $L^2(\Omega_1/\calH,\bar f)\otimes_{\Omega_3/\calH}\ind_{\Omega_3}A$ with $L^2(\Omega_1/\calH,\bar f)\otimes_{\Omega_2/\calH}(C_0(\Omega_2/\calH)\otimes_{\Omega_3/\calH}\ind_{\Omega_3}A)$ and $(\pi_{\bar f}\otimes id_{\ind_{\Omega_3}A})\circ h=\pi_{\bar f}\otimes id_{C_0(\Omega_2/\calH)\otimes_{\Omega_3/\calH}\ind_{\Omega_3}A}$, hence we have 
\[[h]\otimes \tau^{\calG\ltimes (\Omega_3/\calH)}_{\ind_{\Omega_3}A}(\bar f!)=\tau^{\calG\ltimes (\Omega_2/\calH)}_{C_0(\Omega_2/\calH)\otimes_{\Omega_3/\calH}\ind_{\Omega_3}A}(\bar f!)\] 
in $\kk^{\calG\ltimes(\Omega_3/\calH)}$ and therefore in $\kk^\calG$.

By proposition \ref{wrong way functoriality}, $(\overline{g\circ f})!=\bar f\otimes \bar g !$. Since $\tau_{\ind_{\Omega_3}A}^{\calG\ltimes (\Omega_3/\calH)}$ preserves Kasparov product, we have
\begin{align*}
    \ind_{g\circ f,A} & = [\Upsilon_{g\circ f,A}\inv]\otimes \tau^{\calG\ltimes (\Omega_3/\calH)}_{\ind_{\Omega_3}A}((\overline{g\circ f})!)\\
    & = [\Upsilon_{f,A}\inv]\otimes [id\otimes \Upsilon_{g,A}]\inv \otimes [h]\otimes \tau^{\calG\ltimes (\Omega_3/\calH)}_{\ind_{\Omega_3}A}(\bar f!)\otimes \tau^{\calG\ltimes (\Omega_3/\calH)}_{\ind_{\Omega_3}A}(\bar g!)\\
    & =[\Upsilon_{f,A}\inv]\otimes [id\otimes \Upsilon_{g,A}]\inv \otimes \tau^{\calG\ltimes (\Omega_2/\calH)}_{C_0(\Omega_2/\calH)\otimes_{\Omega_3/\calH}\ind_{\Omega_3}A}(\bar f!)\otimes \tau^{\calG\ltimes (\Omega_3/\calH)}_{\ind_{\Omega_3}A}(\bar g!)\\
    & = [\Upsilon_{f,A}\inv]\otimes \tau^{\calG\ltimes (\Omega_2/\calH)}_{\ind_{\Omega_2}A}(\bar f!)\otimes [\Upsilon_{g,A}\inv]\otimes \tau^{\calG\ltimes (\Omega_3/\calH)}_{\ind_{\Omega_3}A}(\bar g!)\\
    & = \ind_{f,A}\otimes \ind_{g,A}.
\end{align*}

And by definition $(id_{\Omega_1})!=1_{C_0(\Omega_1/\calH)}$, therefore $\ind_{id_{\Omega_1},A}=1_{\ind_{\Omega_1}A}$.\end{proof}

\subsubsection{Laxity constraint}

As the most difficult part, we need to check the lax functoriality constraint in our construction of a pseudofunctor, that is the compatibility of our functors $\ind$ with horizontal compositions. The following several lemmas will be used for the proof of proposition \ref{Lax functoriality constraint}.

\begin{lem}\label{phi maps ind_c into ind E}
    Let $\Omega:\calG_3\leftarrow \calG_2$ and $\Lambda,\Lambda':\calG_2\leftarrow\calG_1$ be \'etale groupoid correspondences, $g:\Lambda\ra \Lambda'$ be a $\calG_2,\calG_1$-equivariant continuous map, $A$ be a $\calG_1$-\cst-algebra, then for any $\xi\in \ind_{\Omega\circ \Lambda,c}A$,
    \begin{enumerate}
        \item for any $\omega\in \Omega$, $\varphi_{\Omega,\Lambda,A}\inv(\xi)(\omega)=\xi([\omega,-])\in \ind_{\Lambda^{\sigma_\Omega(\omega)},c}A$;
        \item let $\ca E=\sqcup_{x\in \calG_2\units}(E_{g,A})_x$ be the associated Hilbert bundle of $E_{g,A}$, then 
        \[\Omega\ra \sigma_{\Omega}^* \ca E, \omega\mapsto \varphi_{\Omega,\Lambda,A}\inv(\xi)(\omega)\]
        is an element of $\ind_\Omega E_{g,A}$. That is, $\omega\mapsto \varphi_{\Omega,\Lambda,A}\inv(\xi)(\omega)$ is an element of $\Gamma_b(\Omega,\sigma_\Omega^*\ca E)$, such that it is $\calG_3$-equivariant, and the map $\Omega/\calG_2\ra \mathbb R, \omega\calG_2\mapsto \|\varphi\inv_{\Omega,\Lambda,A}(\xi)(\omega)\|_{\ind_{\Lambda^{\sigma_\Omega(\omega)}}A}$ vanishes at infinity.
    \end{enumerate}
\end{lem}
\begin{proof}
    Firstly we will show that, for any $\omega\in \Omega$, $\varphi_{\Omega,\Lambda,A}\inv(\xi)(\omega)=\xi([\omega,-])$ is an element of $\ind_{\Lambda^{\sigma_\Omega(\omega)},c}A$. The map $pr:(\Omega\circ \Lambda)/\calG_1\ra \Omega/\calG_2, [\omega,\lambda]\calG_1\mapsto \omega\calG_2$ is a continuous map.
    For any $\omega\in \Omega$, we define $j_\omega:\Lambda^{\sigma_\Omega(\omega)}/\calG_1\ra (\Omega\circ \Lambda)/\calG_1, \lambda\calG_1\mapsto [\omega,\lambda]\calG_1$. The map $j_\omega$ is continuous and injective, and $j_\omega(\Lambda^{\sigma_\Omega(\omega)}/\calG_1)=pr\inv(\omega\calG_2)$ is closed in $(\Omega\circ \Lambda)/\calG_1$. So $j_\omega$ is closed inclusion. For any $\xi\in \ind_{\Omega\circ\Lambda,c}A$,
    \[supp(\xi)=\overline{\{[\omega,\lambda]\calG_1\in (\Omega\circ \Lambda)/\calG_1:\xi([\omega,\lambda])\neq 0\in A_{\sigma_\Lambda(\lambda)}\}}\]
    is a compact subset of $(\Omega\circ \Lambda)/\calG_1$. For every $\omega$, 
    \[supp(\varphi\inv_{\Omega,\Lambda,A}(\xi)(\omega))=\overline{\{\lambda\calG_1\in \Lambda/\calG_1:\xi([\omega,\lambda])\neq 0\}}=\overline{j_\omega\inv(\{[\omega,\lambda]\calG_1\in (\Omega\circ\Lambda)/\calG_1:\xi([\omega,\lambda])\neq 0\})}.\]
    The support $supp(\varphi\inv_{\Omega,\Lambda,A}(\xi)(\omega))$ is therefore compact subset of $j_\omega\inv(supp(\xi))$. Hence, for any $\omega\in \Omega$, $\varphi_{\Omega,\Lambda,A}\inv(\xi)\in\ind_{\Lambda^{\sigma_\Omega(\omega)},c}A$.

    Let $\ca A=\sqcup_{x\in \calG_2\units}\ind_{\Lambda^x}A$ be the associated \cst-bundle over $\calG_2\units$ of the $C_0(\calG_2\units)$-algebra $\ind_\Lambda A$, $\ca E=\sqcup_{x\in \calG_2\units}(E_{g,A})_x$ be the associated Hilbert bundle over $\calG\units_2$ of the Hilbert $\ind_{\Lambda'}A$-module $E_{g,A}$. And let $\ca B=\sqcup_{x\in \calG_2\units}\ind_{\Lambda,c}A$. The set $\ca B$ is simultaneously a subset of $\ca A$ and a subset of $\ca E$. We have proved that for any $\omega\in \Omega$, $\varphi_{\Omega,\Lambda,A}\inv(\xi)(\omega)\in \ca B$. The key idea is comparing the two topologies on $\ca B$.
    
    To show that $\omega\mapsto \varphi_{\Omega,\Lambda,A}\inv(\xi)(\omega)$ is an element of $\Gamma_b(\Omega,\sigma^*_\Omega\ca E)$, it suffices to prove that, if $(\omega_i)_i$ is a convergent net with limit $\omega$ in $\Omega$, then $\varphi_{\Omega,\Lambda,A}\inv(\xi)(\omega_i)\ra \varphi_{\Omega,\Lambda,A}\inv(\xi)(\omega)$ in the topology of $\ca E$. Since $\varphi_{\Omega,\Lambda,A}\inv(\xi)$ is an element of $\ind_\Omega\ind_\Lambda A$, we have that $\varphi_{\Omega,\Lambda,A}\inv(\xi)(\omega_i)\ra \varphi_{\Omega,\Lambda,A}\inv(\xi)(\omega)$ in the topology of $\ca A$. Let $\eta\in \ind_\Lambda A\cong\Gamma_0(\calG_2\units,\ca A)$ such that $\eta|_{\Lambda^{\sigma_\Omega(\omega)}}=\varphi_{\Omega,\Lambda,A}\inv(\xi)(\omega)$. Recall that $\ind_\Lambda A$ is also a $C_0(\Lambda/\calG_1)$-algebra. Let $K=supp(\varphi_{\Omega,\Lambda,A}\inv(\xi)(\omega))$, it is a compact subset of $\Lambda/\calG_1$. By Tietze's extension theorem, there exists $\phi\in C_c(\Lambda/\calG_1)$ such that $\phi|_K=1$. We have therefore $\eta|_{\Lambda^{\sigma_\Omega(\omega)}}=(\phi\cdot \eta)|_{\Lambda^{\sigma_\Omega(\omega)}}$. After replacing $\eta$ by $\phi\cdot \eta$, we can assume that $\eta\in \ind_{\Lambda,c}A$. The convergences $\varphi_{\Omega,\Lambda,A}\inv(\xi)(\omega_i)\ra \varphi_{\Omega,\Lambda,A}\inv(\xi)(\omega)$ in $\ca A$ and $\eta|_{\Lambda^{\sigma_\Omega(\omega_i)}}\ra \eta|_{\Lambda^{\sigma_\Omega(\omega)}}=\varphi_{\Omega,\Lambda,A}\inv(\xi)(\omega)$ in $\ca A$ implies that $\varphi_{\Omega,\Lambda,A}\inv(\xi)(\omega_i)-\eta|_{\Lambda^{\sigma_\Omega(\omega_i)}}\ra 0_{\sigma_\Omega(\omega)}$ in $\ca A$. Then by \cite[Lemma C.18]{williams2007crossed}, we have $\|\varphi_{\Omega,\Lambda,A}\inv(\xi)(\omega_i)-\eta|_{\Lambda^{\sigma_\Omega(\omega_i)}}\|_{\ind_{\Lambda^{\sigma_\Omega(\omega_i)}}A}\ra 0$.

    \begin{align*}
        & \|\varphi_{\Omega,\Lambda,A}\inv(\xi)(\omega_i)-\eta|_{\Lambda^{\sigma_\Omega(\omega_i)}}\|^2_{(E_{g,A})_{\sigma_\Omega(\omega_i)}} \\
        & = \sup_{\lambda'\in {\Lambda'}^{\sigma_\Omega(\omega_i)}}\|\langle \varphi_{\Omega,\Lambda,A}\inv(\xi)(\omega_i)-\eta|_{\Lambda^{\sigma_\Omega(\omega_i)}},\varphi_{\Omega,\Lambda,A}\inv(\xi)(\omega_i)-\eta|_{\Lambda^{\sigma_\Omega(\omega_i)}}\rangle(\lambda')\|_{A_{\sigma_{\Lambda'}(\lambda')}}\\
        & = \sup_{\lambda'\in {\Lambda'}^{\sigma_\Omega(\omega_i)}}\|\sum_{\lambda\in g\inv(\lambda')} |\varphi_{\Omega,\Lambda,A}\inv(\xi)(\omega_i)(\lambda)-\eta(\lambda)|^2\|_{A_{\sigma_{\Lambda'}(\lambda')}}\\
        & = \sup_{\lambda'\in {\Lambda'}^{\sigma_\Omega(\omega_i)}}\|\sum_{\lambda\in g\inv(\lambda')} |\xi([\omega_i,\lambda])-\eta(\lambda)|^2\|_{A_{\sigma_{\Lambda'}(\lambda')}}\\
        & \leqslant N\cdot \sup_{\lambda'\in {\Lambda'}^{\sigma_\Omega(\omega_i)}}\|\xi([\omega_i,\lambda])-\eta(\lambda)\|^2_{A_{\sigma_{\Lambda'}(\lambda')}}\\
        & = N\cdot \|\varphi_{\Omega,\Lambda,A}\inv(\xi)(\omega_i)-\eta|_{\Lambda^{\sigma_\Omega(\omega_i)}}\|^2_{\ind_{\Lambda^{\sigma_\Omega(\omega_i)}}A},
    \end{align*}
    where
    \[N=\sup_{\omega_* \in \Omega}\sup_{\lambda'\in \Lambda^{\sigma_\Omega(\omega_*)}}\#\{\lambda\in g\inv(\lambda'):\xi([\omega_*,\lambda])-\eta(\lambda)\neq 0\}.\]
    By lemma \ref{canonical bijection between fibers}, for $[\omega_*,\lambda']\in \Omega\circ \Lambda'$,
    \[g\inv(\lambda),\overline{g}\inv(\lambda'\calG_1),[id,g]\inv([\omega_*,\lambda']),\overline{[id,g]}\inv([\omega_*,\lambda']\calG_1)\]
     are canonically pairwisely in bijection. Since $supp(\xi)$ is a compact subset of $(\Omega\circ\Lambda)/\calG_1$, $\overline{[id,g]}:(\Omega\circ\Lambda)/\calG_1\ra (\Omega\circ\Lambda')/\calG_1$ is a local homeomorphism, by lemma \ref{uniformly finite fiber},
    \begin{align*}
        N_1 & =\sup_{\omega_*\in \Omega}\sup_{\lambda'\in {\Lambda'}^{\sigma_{\Omega}(\omega_*)}}\#\{\lambda\in g\inv(\lambda'):\xi([\omega_*,\lambda])\neq 0\}\\
        & \leqslant \sup_{[\omega_*,\lambda']\calG_1\in (\Omega\circ \Lambda')/\calG_1}\#\{[\omega_*,\lambda]\calG_1\in \overline{[id,g]}\inv([\omega_*,\lambda']): [\omega_*,\lambda]\calG_1\in supp(\xi)\}
    \end{align*}
    is finite.
    \begin{align*}
        N_2 & = \sup_{\omega_*\in \Omega}\sup_{\lambda'\in {\Lambda'}^{\sigma_{\Omega}(\omega_*)}}\#\{\lambda\in g\inv(\lambda'):\eta(\lambda)\neq 0\}\\
        & \leqslant \sup_{\lambda'\calG_1\in \Lambda'/\calG_1} \#(\overline{g}\inv(\lambda')\cap supp(\eta))
    \end{align*}
    is finite. So $N\leqslant N_1+N_2$ is finite. This implies that $\varphi_{\Omega,\Lambda,A}\inv(\xi)(\omega_i)-\eta|_{\Lambda^{\sigma_\Omega(\omega_i)}}\ra 0_{\sigma_\Omega(\omega)}$ in $\ca E$. While $\eta\in \ind_{\Lambda,c}A\subseteq E_{g,A}\cong \Gamma_0(\calG_2\units,\ca E)$. So $\eta|_{\Lambda^{\sigma_\Omega(\omega_i)}}\ra\eta|_{\Lambda^{\sigma_\Omega(\omega)}}$ in $\ca E$. We have therefore $\varphi_{\Omega,\Lambda,A}\inv(\xi)(\omega_i)\ra \eta|_{\Lambda^{\sigma_\Omega(\omega)}}=\varphi_{\Omega,\Lambda,A}\inv(\xi)(\omega)$ in $\ca E$. So we proved that $\varphi_{\Omega,\Lambda,A}\inv(\xi)\in \Gamma_b(\Omega,\sigma_\Omega^*\ca E)$.

    If $\gamma\in \calG_2^{\sigma_\Omega(\omega)}$,
    \[\varphi\inv_{\Omega,\Lambda,A}(\xi)(\omega\gamma) = \xi([\omega\gamma,-])= \xi([\omega,\gamma.-])=\gamma\inv.(\varphi\inv_{\Omega,\Lambda,A}(\xi)(\omega)).\]
    
    Finally, $supp(\varphi\inv_{\Omega,\Lambda,A}(\xi))=\overline{\{\omega\calG_2\in \Omega/\calG_2:\exists \lambda\in\Lambda^{\sigma_\Omega(\omega)},\|\xi([\omega,\lambda])\|\neq 0\}}$ is $pr(supp(\xi))$, so the map $\omega\calG_2\mapsto \|\varphi\inv_{\Omega,\Lambda,A}(\xi)(\omega)\|_{\ind_{\Lambda^{\sigma_\Omega(\omega)}}A}$ is compactly supported.

    In conclusion $\varphi\inv_{\Omega,\Lambda,A}(\xi)$ is a well-defined element in $\ind_{\Omega}E_{g,A}$.
\end{proof}

\begin{lem}\label{phi and diamond}
    Let $\calG_1,\calG_2,\calG_3$ be \'etale groupoids, $\Omega:\calG_3\leftarrow \calG_2$, $\Lambda:\calG_2\leftarrow \calG_1$ be \'etale groupoid correspondences, then for any $h_1\in C_c(\Omega)$, $h_2\in C_c(\Lambda)$, $a\in A$,
    \[\varphi_{\Omega,\Lambda,A}(h_1\diamond (h_2\diamond a))=(h_1\diamond h_2)\diamond a,\]
    where $h_1\diamond h_2\in C_c(\Omega\circ \Lambda)$ is defined as
    \[h_1\diamond h_2([\omega,\lambda])=\sum_{\gamma\in \calG_2^{\sigma_{\Omega}(\omega)}}h_1(\omega\gamma)h_2(\gamma\inv\lambda).\]
\end{lem}
\begin{proof}
    For any $h_1\in C_c(\Omega)$, $h_2\in C_c(\Lambda)$, $a\in A$, we have, for any $\omega\in \Omega$ and for any $\lambda\in \Lambda^{\sigma_{\Omega}(\omega)}$,
\begin{align*}
    [h_1\diamond (h_2\diamond a)](\omega)(\lambda) & =\big(\sum_{\gamma\in \calG_2^{\sigma_{\Omega}(\omega)}}h_1(\omega \gamma)(h_2\diamond a)|_{\Lambda^{s(\gamma)}}(\gamma\inv-)\big)(\lambda)\\
    & = \sum_{\gamma\in \calG_2^{\sigma_{\Omega}(\omega)}}h_1(\omega \gamma)(h_2\diamond a)(\gamma\inv\lambda)\\
    & = \sum_{\gamma\in \calG_2^{\sigma_{\Omega}(\omega)}}h_1(\omega \gamma)\sum_{\gamma'\in \calG_1^{\sigma_\Lambda(\lambda)}}h_2(\gamma\inv \lambda \gamma')\alpha_{\gamma'}(a(s(\gamma')))\\
    & = \sum_{\gamma'\in \calG_1^{\sigma_\Lambda(\lambda)}}\big(\sum_{\gamma\in \calG_2^{\sigma_{\Omega}(\omega)}}h_1(\omega\gamma)h_2(\gamma\inv\lambda\gamma')\big)\alpha_{\gamma'}(a(s(\gamma')))\\
    & = \sum_{\gamma'\in \calG_1^{\sigma_\Lambda(\lambda)} } (h_1\diamond h_2)([\omega,\lambda\gamma'])\alpha_{\gamma'}(a(s(\gamma')))\\
    & = (h_1\diamond h_2)\diamond a([\omega,\lambda]).
\end{align*}
Therefore, $(h_1\diamond h_2)\diamond a=\varphi_{\Omega',\Lambda,A}(h_1\diamond (h_2\diamond a))$.
\end{proof}

\begin{corr}\label{phi map CCA into ind_c}
    Use the same notation as above, we define
    \[C_c(\Omega)\diamond (C_c(\Lambda)\diamond A)=span\{h_1\diamond (h_2\diamond a):h_1\in C_c(\Omega),h_2\in C_c(\Lambda),a\in A\},\]
    then we have
    \[\varphi_{\Omega,\Lambda,A}(C_c(\Omega)\diamond (C_c(\Lambda)\diamond A))\subseteq \ind_{\Omega\circ \Lambda,c}A.\]
\end{corr}

\begin{lem}\label{very technical nonsense 2}
    Let $\Omega,\Omega':\calG\leftarrow \calH$ be two correspondences, $A$ be an $\calH$-\cst-algebra, $f:\Omega\ra \Omega'$ be a $\calG,\calH$-equivariant continuous map. Let $A'$ be a dense subset of $A$. Then $C_c(\Omega)\diamond A':=span\{\phi\diamond a:\phi\in C_c(\Omega),a\in A'\}$ is dense in $E_{f,A}$.
\end{lem}
\begin{proof}
    By corollary \ref{ind transfer pre-hilb mod}, $C_c(\Omega)\diamond A':=span\{\phi\diamond a:\phi\in C_c(\Omega),a\in A'\}\subseteq \ind_{\Omega,c}A$ is dense in $\ind_\Omega A$. For $\xi\in \ind_{\Omega,c}A$, assume that $(\xi_i)_i$ is a sequence in $C_c(\Omega)\diamond A'$ that $\xi_i\xrightarrow{\|\cdot\|_{\ind_\Omega A}}\xi$. Let $K= supp(\xi)$. There is a relatively compact open neighborhood $U$ of $K$ in $\Omega/\calH$. By Urysohn's lemma, there exists $h\in C_c(U,[0,1])$ such that $h|_K=1$. Recall that $\ind_{\Omega}A$ can be seen as a $C_0(\Omega/\calH)$-algebra, $\xi_i\xrightarrow{\|\cdot\|_{\ind_\Omega A}}\xi$ implies that $h\cdot \xi_i\xrightarrow{\|\cdot\|_{\ind_\Omega A}}h\cdot\xi=\xi$.

    For any $f\in C_c(\Omega)$ and $a\in A'$,
    \[(h\cdot(f\diamond a))(\omega)=h(\omega\calH)\sum_{h\in \calH^{\sigma(\omega)}}f(\omega h)\alpha_h(a(s(h)))=((h\cdot f)\diamond a)(\omega),\]
    Hence, every $h\cdot \xi_i$ is still in $C_c(\Omega)\diamond A'$. So after replacing $\xi_i$ by $h\cdot \xi_i$, we can assume that for all $i$, $supp(\xi_i)\subseteq \overline{U}$, which is compact. Let $N=\sup_{\omega'\in \Omega'}\#(\overline{U}\cap \bar f\inv(\omega'\calH))<\infty$.
    \begin{align*}
        \|\xi-\xi_i\|^2_{E_{f,A}} & = \sup_{\omega'\in \Omega'}\|\sum_{\omega\in f\inv(\omega')}|(\xi-\xi_i)(\omega)|^2\|_{A_{\sigma_{\Omega'}(\omega')}}\\
        & \leqslant N\cdot \|\xi-\xi_i\|^2_{\ind_\Omega A}.
    \end{align*}
    Hence, $\xi_i\xrightarrow{\|\cdot\|_{E_{f,A}}}\xi$.
\end{proof}

\begin{prop}[Lax functoriality constraint]\label{Lax functoriality constraint}
    For any $\calG_1,\calG_2,\calG_3\in B_0$, there is a natural isomorphism $\varphi_{\calG_1,\calG_2,\calG_3}$ (or write $\varphi$ in short)
    
    \[\begin{tikzcd}
\frgr(\calG_2, \calG_3) \times \frgr(\calG_1, \calG_2) \arrow[r, "c"] \arrow[d, "\ind \times \ind"'] & \frgr(\calG_1, \calG_3) \arrow[d, "\ind"] \\
\mathrm{Fun}(\kk^{\calG_2},\kk^{\calG_3}) \times \mathrm{Fun}(\kk^{\calG_1},\kk^{\calG_2}) \arrow[r] \arrow[ru, Rightarrow, "\varphi", shorten <=10pt, shorten >=10pt] & \mathrm{Fun}(\kk^{\calG_1},\kk^{\calG_3})
\end{tikzcd}\]
consisting of isomorphism $\varphi_{\Omega,\Lambda}$ for each $\Omega\in \frgr(\calG_2, \calG_3)$, $\Lambda\in \frgr(\calG_1, \calG_2)$, where the bottom arrow in the diagram is composition of functors. (Recall that isomorphisms in $\mathrm{Fun}(\kk^{\calG_1},\kk^{\calG_3})$ are natural isomorphisms.)
\end{prop}

We will divide the proof into two parts.

\begin{lem}\label{lax fun constraint part 1}
    Use the same notation as above, let $\Omega,\Omega'\in \frgr(\calG_2,\calG_3)$ be two correspondences, $\Lambda\in \frgr(\calG_1,\calG_2)$ be a correspondence, $f:\Omega\ra\Omega'$ be a $\calG_3,\calG_2$-equivariant continuous map, then for any separable $\calG_1$-\cst-algebra $A$,
    \begin{equation*}
    [\varphi_{\Omega,\Lambda,A}]\otimes \ind_{[f,id],A}=\ind_{f,\ind_\Lambda A}\otimes[\varphi_{\Omega',\Lambda,A}] \in \kk^{\calG_3}(\ind_\Omega\ind_\Lambda A,\ind_{\Omega'\circ \Lambda} A).
\end{equation*}
That is, the following diagram, whose arrows are elements of KK-groups, is commutative.
\[
\xymatrix{
\ind_\Omega \ind_\Lambda A \ar[r]^{\varphi_{\Omega,\Lambda,A}} \ar[d]_{\ind_{f,\ind_\Lambda A}} & \ind_{\Omega\circ \Lambda}A \ar[d]^{\ind_{[f,id],A}} \\
\ind_{\Omega'}\ind_\Lambda A \ar[r]^{\varphi_{\Omega',\Lambda,A}} & \ind_{\Omega'\circ \Lambda}A
}
\]
(Here $[f,id]:\Omega\circ \Lambda\ra \Omega'\circ \Lambda$ is the map $[\omega,\lambda]\mapsto[f(\omega),\lambda]$.)
\end{lem}
\begin{proof}
    It suffices to show that, as $\ind_\Omega\ind_\Lambda A, \ind_{\Omega'\circ \Lambda}A$-bimodules, we have $\calG$-equivariant unitarily equivalence
\[(E_{[f,id],A},\pi_{[f,id],A}\circ \varphi_{\Omega,\Lambda,A})\cong (E_{f,\ind_\Lambda A}\otimes_{\varphi_{\Omega',\Lambda,A}}\ind_{\Omega'\circ \Lambda}A,\pi_{f,\ind_\Lambda A}\otimes id).\]
Let $F$ be the underlying Banach space of $E_{f,\ind_\Lambda A}$ but equipped with $\calG$-Hilbert $\ind_{\Omega'\circ \Lambda}A$-module structure defined by
\[\xi\cdot\eta=\xi\cdot\varphi_{\Omega',\Lambda,A}\inv(\eta), \forall \xi\in F,\eta\in \ind_{\Omega'\circ \Lambda}A,\]
\[\langle \xi_1,\xi_2\rangle_F=\varphi_{\Omega',\Lambda,A}(\langle \xi_1,\xi_2\rangle_{E_{f,\ind_\Lambda A}}),\forall \xi_1,\xi_2\in F.\]
Let $\pi:\ind_\Omega\ind_\Lambda A\ra \calL_{\ind_{\Omega'\circ \Lambda}A}(F)$ be defined as
\[\pi(\zeta)(\xi)=\pi_{f,\ind_\Lambda A}(\zeta)(\xi).\]
Clearly $(E_{f,\ind_\Lambda A}\otimes_{\varphi_{\Omega',\Lambda,A}}\ind_{\Omega'\circ \Lambda}A,\pi_{f,\ind_\Lambda A}\otimes id)$ is $\calG$-equivariantly unitarily equivalent to $(F,\pi)$.

By proposition \ref{diamond construction}, $C_c(\Lambda)\diamond A$ is dense in $\ind_\Lambda A$, then by lemma \ref{very technical nonsense 2}, $C_c(\Omega)\diamond (C_c(\Lambda)\diamond A)$ is dense in $F=E_{f,\ind_\Lambda A}$.

By corollary \ref{phi map CCA into ind_c}, $\varphi_{\Omega,\Lambda,A}$ sends $C_c(\Omega)\diamond (C_c(\Lambda)\diamond A)$ into $\ind_{\Omega\circ \Lambda,c}A$. If $\xi\in \ind_{\Omega\circ \Lambda,c}A$, then $supp(\varphi_{\Omega,\Lambda,A}\inv(\xi))$ is a compact subset of $\Omega/\calG_2$. So $\varphi_{\Omega,\Lambda,A}\inv$ maps $\ind_{\Omega\circ \Lambda,c}A$ into $\ind_{\Omega,c}\ind_\Lambda A\subseteq F$. Let $T:\ind_{\Omega\circ \Lambda,c}A\ra F$ be the map $\xi \mapsto \varphi\inv_{\Omega,\Lambda,A}(\xi)$. See $\ind_{\Omega\circ \Lambda,c}A$ as a dense subspace of $E_{[f,id],A}$.

For any $\xi_1,\xi_2\in \ind_{\Omega\circ \Lambda,c}A$, $[\omega',\lambda]\in \Omega'\circ \Lambda$, using lemma \ref{canonical bijection between fibers},
\begin{align*}
    \langle T\xi_1, T\xi_2\rangle_{F}([\omega',\lambda]) & = \langle T\xi_1, T\xi_2\rangle_{E_{f,\ind_\Lambda A}}(\omega')(\lambda)\\
    & = \sum_{\omega\in f\inv(\omega')}[(T\xi_1)(\omega)(\lambda)]^*[(T\xi_2)(\omega)(\lambda)]\\
    & = \sum_{\omega\in f\inv(\omega')}\xi_1([\omega,\lambda])^*\xi_2([\omega,\lambda])\\
    & = \sum_{[\omega,\lambda]\in f\inv([\omega',\lambda])}\xi_1([\omega,\lambda])^*\xi_2([\omega,\lambda])\\
    & = \langle \xi_1,\xi_2\rangle_{E_{[f,id],A}}([\omega',\lambda]).
\end{align*}
So $\langle T\xi_1, T\xi_2\rangle_{F}=\langle \xi_1,\xi_2\rangle_{E_{[f,id],A}}$.

Finally, by corollary \ref{phi map CCA into ind_c}, $T(\ind_{\Omega\circ \Lambda,c}A)$ contains $C_c(\Omega)\diamond (C_c(\Lambda)\diamond A)$. We have proved that $C_c(\Omega)\diamond (C_c(\Lambda)\diamond A)$
is dense in $F$. So by lemma \ref{to be a unitary}, $T$ extends to a unitary element of $\calL(E_{[f,id],A},F)$. By abuse of language, we still denote this unitary by $T$.

The *-homomorphism $\varphi\inv_{\Omega,\Lambda,A}$ is $\calG_3$-equivariant, so it is easy to check that $T$ is also $\calG_3$-equivariant by applying lemma \ref{to be equivariant adjointable op} to the pre-$\calG_3$-Hilbert $\ind_{\Omega'\circ \Lambda}A$-module $\ind_{\Omega\circ\Lambda,c}A$.

For any $\phi\in \ind_{\Omega\circ\Lambda,c} A\subseteq \ind_{\Omega\circ \Lambda}A$, $\varphi_{\Omega,\Lambda,A}\inv(\phi)\in \ind_{\Omega,c}\ind_\Lambda A$, and for any $\xi\in \ind_{\Omega,c}\ind_\Lambda A\subseteq E_{f,\ind_\Lambda A}=F$, for any $[\omega,\lambda]\in \Omega\circ \Lambda$,
\begin{align*}
    [(\pi_{[f,id],A}(\phi)\circ T)\xi]([\omega,\lambda])
    & = \phi([\omega,\lambda])\varphi_{\Omega,\Lambda,A}(\xi)([\omega,\lambda])\\
    & = [\varphi_{\Omega,\Lambda,A}\inv(\phi)\xi](\omega)(\lambda)\\
    & = [\pi_{f,\ind_\Lambda A}(\varphi_{\Omega,\Lambda,A}\inv(\phi))\xi](\omega)(\lambda)\\
    & = [\pi(\varphi_{\Omega,\Lambda,A}\inv(\phi))\xi](\omega)(\lambda)\\
    & = [(T\circ \pi(\varphi\inv_{\Omega,\Lambda,A}(\phi))\xi]([\omega,\lambda])
\end{align*}
That is, for any $\psi\in \varphi\inv_{\Omega,\Lambda,A}(\ind_{\Omega\circ \Lambda,c}A)$, which is dense in $\ind_\Omega\ind_\Lambda A$, we have
\[\pi_{[f,id],A}(\varphi_{\Omega,\Lambda,A}(\psi))\circ T=T\circ \pi(\psi).\]
In conclusion, $(E_{[f,id],A},\pi_{[f,id],A}\circ \varphi_{\Omega,\Lambda,A})$ is $\calG$-equivariantly unitarily equivalent to $(F,\pi)$. We finish the proof.
\end{proof}

\begin{lem}\label{lax fun constraint part 2}
    Use the same notation as above, let $\Omega'\in \frgr(\calG_2,\calG_3)$ be a correspondence, $\Lambda,\Lambda'\in \frgr(\calG_1,\calG_2)$ be two correspondences, $g:\Lambda\ra\Lambda'$ be a $\calG_2,\calG_1$-equivariant continuous map, then
    \begin{equation*}[\varphi_{\Omega',\Lambda,A}]\otimes\ind_{[id,g],A}=\ind_{\Omega'}(\ind_{g,A})\otimes [\varphi_{\Omega',\Lambda',A}].
\end{equation*}
That is, the following diagram, whose arrows are elements of KK-groups, is commutative.
\[
\xymatrix{
\ind_{\Omega'}\ind_\Lambda A \ar[r]^{\varphi_{\Omega',\Lambda,A}} \ar[d]_{\ind_{\Omega'}\ind_{g,A}}  & \ind_{\Omega'\circ \Lambda} A \ar[d]^{\ind_{[g,id],A}}\\
\ind_{\Omega'}\ind_{\Lambda'}A \ar[r]_{\varphi_{\Omega',\Lambda',A}} & \ind_{\Omega'\circ \Lambda'} A
}
\]
(Here $[id,g]:\Omega'\circ \Lambda\ra \Omega'\circ \Lambda'$ is the map $[\omega',\lambda]\mapsto [\omega,g(\lambda)]$.)
\end{lem}
\begin{proof}
    It suffices to prove that $(E_{[id,g],A},\pi_{[id,g],A}\circ \varphi_{\Omega',\Lambda,A})$ is $\calG_3$-equivariantly unitarily equivalent to \[((\ind_{\Omega'}E_{g,A})\otimes_{\varphi_{\Omega',\Lambda',A}}\ind_{\Omega'\circ \Lambda'}A,(\ind_{\Omega'}\pi_{g,A})\otimes id),\]
which is $\calG_3$-equivariantly unitarily equivalent to the bimodule $(F',\pi')$ defined as follows.
Let $F'$ be the underlying Banach space of $\ind_{\Omega'}E_{g,A}$ but equipped with Hilbert $\ind_{\Omega'\circ \Lambda'}A$-module structure as
\[\xi\cdot\eta=\xi\cdot \varphi_{\Omega',\Lambda',A}\inv(\eta), \forall \xi\in \ind_{\Omega'}E_{g,A}, \eta\in \ind_{\Omega'\circ \Lambda'}A,\]
\[\langle \xi_1,\xi_2\rangle_{F'}=\varphi_{\Omega',\Lambda',A}(\langle \xi_1,\xi_2\rangle_{\ind_{\Omega'}E_{g,A}}).\]
Let $\pi':\ind_{\Omega'}\ind_\Lambda A\ra\calL_{\ind_{\Omega'\circ \Lambda'}A}(F')$ be defined as
\[\pi'(\zeta)(\xi)=[(\ind_{\Omega'}(\pi_{g,A}))(\zeta)]\xi.\]

Now consider the map $T':\ind_{\Omega'\circ\Lambda,c}A\ra F'=\ind_{\Omega'}E_{g,A},\xi\mapsto \varphi\inv_{\Omega',\Lambda,A}(\xi)$. Lemma \ref{phi maps ind_c into ind E} implies that this is a well-defined map.

For $\xi_1,\xi_2\in \ind_{\Omega'\circ \Lambda,c}A\subseteq E_{[id,g],A}$, $[\omega',\lambda']\in \Omega'\circ \Lambda'$, using lemma \ref{canonical bijection between fibers},
\begin{align*}
    \langle T'\xi_1,T'\xi_2\rangle_{F'}([\omega',\lambda']) & = \varphi_{\Omega',\Lambda',A}(\langle T'\xi_1,T'\xi_2\rangle_{\ind_{\Omega'}E_{g,A}})([\omega',\lambda'])\\
    & = \langle T'\xi_1,T'\xi_2\rangle_{\ind_{\Omega'}E_{g,A}}(\omega')(\lambda')\\
    & = \langle (T'\xi_1)(\omega'),(T'\xi_2)(\omega')\rangle_{E_{g,A,\sigma_{\Omega'}(\omega')}}(\lambda')\\
    & = \sum_{\lambda\in g\inv(\lambda')}\xi_1([\omega',\lambda])^*\xi_2([\omega',\lambda])\\
    & = \sum_{[\omega',\lambda]\in [id,g]\inv([\omega',\lambda'])}\xi_1([\omega',\lambda])^*\xi_2([\omega',\lambda])\\
    & = \langle \xi_1,\xi_2\rangle_{E_{[id,g],A}}([\omega',\lambda']).
\end{align*}

By lemma \ref{very technical nonsense 2}, $C_c(\Lambda)\diamond A$ is dense in $E_{g,A}$, then by corollary \ref{ind transfer pre-hilb mod}, $C_c(\Omega')\diamond (C_c(\Lambda)\diamond A)$ is dense in $\ind_{\Omega'}E_{g,A}=F'$. While corollary \ref{phi map CCA into ind_c} implies that $C_c(\Omega')\diamond (C_c(\Lambda)\diamond A)$ is contained in $T'(\ind_{\Omega'\circ \Lambda,c}A)$. So $T'$ has dense image. Then by lemma \ref{to be a unitary}, $T'$ extends to a unitary element of $\calL_{\ind_{\Omega'\circ\Lambda'}A}(E_{[id,g],A},F')$. By abuse of language, we still denote this operator as $T'$.

It is easy to check that $T'$ is $\calG_3$-equivariant since  $\varphi_{\Omega',\Lambda,A}\inv$ is $\calG_3$-equivariant, by applying lemma \ref{to be equivariant adjointable op} to the pre-$\calG_3$-Hilbert $\ind_{\Omega'\circ \Lambda'}A$-module $\ind_{\Omega'\circ \Lambda,c}A$.

Now for any $\phi\in \varphi\inv_{\Omega',\Lambda,A}(\ind_{\Omega'\circ\Lambda,c}A)\subseteq \ind_{\Omega'}\ind_\Lambda A$, $\xi\in C_c(\Omega')\diamond (C_c(\Lambda)\diamond A)\subseteq \ind_{\Omega'}E_{g,A}=F'$, $[\omega',\lambda]\in \Omega'\circ \Lambda$,
\begin{align*}
    & [(\pi_{[id,g],A}(\varphi_{\Omega',\Lambda,A}(\phi))\circ T')\xi]([\omega',\lambda])\\
    & = \varphi_{\Omega',\Lambda,A}(\phi)([\omega',\lambda])(T'\xi)([\omega',\lambda])\\
    & = \phi(\omega')(\lambda)\xi(\omega')(\lambda)\\
    & = [T'(\pi'(\phi)\xi)]([\omega',\lambda']).
\end{align*}
We conclude that $\pi_{[id,g],A}( \varphi_{\Omega',\Lambda,A}(-))\circ T'=T'\circ \pi'(-)$. Hence, $(E_{[id,g],A},\pi_{[id,g],A}\circ \varphi_{\Omega',\Lambda,A})$ is $\calG_3$-equivariantly unitarily equivalent to $(F',\pi')$.
\end{proof}

\begin{proof}[Proof of proposition \ref{Lax functoriality constraint}] It suffices to show that, for any $\Omega,\Omega'\in \frgr(\calG_2, \calG_3)$ and any $\calG_3,\calG_2$-equivariant continuous map $f:\Omega\ra \Omega'$, for any $\Lambda,\Lambda'\in \frgr(\calG_1, \calG_2)$ and $\calG_2,\calG_1$-equivariant continuous map $g:\Lambda\ra \Lambda'$, use $[f,g]$ to denote the map $\Omega\circ\Lambda\ra \Omega'\circ \Lambda', [\omega,\lambda]\mapsto [f(\omega),g(\lambda)]$, the following diagram of natural transformations commutes.
\[
    \xymatrix{
    \ind_{\Omega}\ind_{\Lambda} \ar[r]^{\varphi_{\Omega,\Lambda}} \ar[d]_{\ind_f\star\ind_g} & \ind_{\Omega\circ \Lambda} \ar[d]^{\ind_{[f,g]}} \\
    \ind_{\Omega'}\ind_{\Lambda'} \ar[r]_{\varphi_{\Omega',\Lambda'}} & \ind_{\Omega'\circ \Lambda'}
    }
\]
where $\ind_f\star\ind_g$ is the horizontal composition of the two natural transformations $\ind_f$ and $\ind_g$.

In other word, we need to check that for any $A\in \kk^{\calG_1}$,
\[[\varphi_{\Omega,\Lambda,A}]\otimes \ind_{[f,g],A}=\ind_{f,\ind_{\Lambda}A}\otimes\ind_{\Omega'}(\ind_{g,A})\otimes [\varphi_{\Omega',\Lambda',A}] \]
is an equality in $\kk^{\calG_3}(\ind_{\Omega}\ind_{\Lambda}A,\ind_{\Omega'\circ\Lambda'}A)$.

Since $[f,g]=[id_{\Omega'},g]\circ [f,id_\Lambda]:\Omega\times_{\calG_2}\Lambda\ra \Omega'\times_{\calG_2}\Lambda\ra\Omega'\times_{\calG_2}\Lambda'$, by proposition \ref{ind is functor}, we have
\[\ind_{[f,g],A}=\ind_{[f,id],A}\otimes \ind_{[id,g],A}.\]

Then use lemma \ref{lax fun constraint part 1} and lemma \ref{lax fun constraint part 2}, we have
\begin{align*}
    [\varphi_{\Omega,\Lambda,A}]\otimes \ind_{[f,g],A} & = [\varphi_{\Omega,\Lambda,A}]\otimes \ind_{[f,id],A}\otimes \ind_{[id,g],A} \\
    & = \ind_{f,\ind_\Lambda A}\otimes[\varphi_{\Omega',\Lambda,A}]\otimes \ind_{[id,g],A} \\
    & = \ind_{f,\ind_\Lambda A}\otimes \ind_{\Omega'}(\ind_{g,A})\otimes [\varphi_{\Omega',\Lambda',A}].
\end{align*}
\end{proof}

For any $\calG\in B_0$ and $A\in \kk^\calG$, we can define such a $\calG$-equivariant *-isomorphism
\[\psi_{\calG,A}:\ind_{\calG}A\ra A,\]
\[\psi_{\calG,A}(\xi)(x)=\xi(x),\quad \forall x\in \calG\units.\]
\begin{prop}[Lax unity constraint]\label{Lax unity constraint}
    For any $\calG\in B_0$, there is a natural isomorphism $\psi_{\calG}:\ind_\calG\ra id_{\kk^\calG}$ consisting of isomorphisms $[\psi_{\calG,A}]\in \kk^\calG(\ind_\calG A,A)$.
\end{prop}
\begin{proof} This is exactly \cite[proposition 3.27]{miller2022k}.\end{proof}

\subsubsection{Lax associativity, lax left and right unity}

\begin{prop}[Lax associativity]\label{lax associativity}
    For any 0-cells $\calG_1,\calG_2,\calG_3,\calG_4\in B_0$, 1-cells $\Omega\in \frgr(\calG_1,\calG_2)$, $\Lambda\in \frgr(\calG_2,\calG_3)$, $\Sigma\in \frgr(\calG_3,\calG_4)$, the following diagram of natural transformations commutes.
\[\xymatrix{
\ind_\Sigma \ind_\Lambda \ind_\Omega \ar[d]_{\varphi_{\Sigma,\Lambda}\star id} \ar[dr]^{id\star \varphi_{\Lambda,\Omega}} \\
\ind_{\Sigma\circ \Lambda} \ind_\Omega \ar[d]_{\varphi_{\Sigma\circ \Lambda,\Omega}} & \ind_\Sigma \ind_{\Lambda\circ \Omega} \ar[d]^{\varphi_{\Sigma,\Lambda\circ \Omega}}\\
\ind_{(\Sigma\circ \Lambda)\circ \Omega}\ar[r]_{\ind_{a_{\Sigma,\Lambda,\Omega}}}  & \ind_{\Sigma\circ (\Lambda\circ \Omega)}
}\]
\end{prop}
\begin{proof}
    Since $a_{\Sigma,\Lambda,\Omega}: (\Sigma\circ \Lambda)\circ \Omega\ra \Sigma\circ (\Lambda\circ \Omega), [[s,\lambda],\omega]\mapsto [s,[\lambda,\omega]] $ is a $\calG_4,\calG_1$-equivariant homeomorphism, for any $A\in \kk^{\calG_1}$, $\ind_{a_{\Sigma,\Lambda,\Omega},A}\in \kk^{\calG_4}(\ind_{(\Sigma\circ \Lambda)\circ \Omega}A,\ind_{\Sigma\circ (\Lambda\circ \Omega)}A)$ is induced by the $\calG_4$-equivariant *-isomorphism
\[(a_{\Sigma,\Lambda,\Omega}\inv)^*:\ind_{(\Sigma\circ \Lambda)\circ \Omega}A\ra \ind_{\Sigma\circ (\Lambda\circ \Omega)}A,\]
\[(a_{\Sigma,\Lambda,\Omega}\inv)^*(\xi)([\sigma,[\lambda,\omega]])=\xi([[\sigma,\lambda],\omega]).\]

It suffices to prove the following equality of $\calG_4$-equivariant *-isomorphisms: for any $A\in \kk^{\calG_1}$,
\[(a_{\Sigma,\Lambda,\Omega}\inv)^*\circ \varphi_{\Sigma\circ \Lambda,\Omega,A}\circ \varphi_{\Sigma,\Lambda,\ind_\Omega A}=\varphi_{\Sigma, \Lambda\circ \Omega,A}\circ\ind_{\Sigma}(\varphi_{\Lambda,\Omega,A}).\]

For any $\xi\in \ind_{\Sigma}\ind_\Lambda\ind_\Omega A$, $s\in \Sigma$, $\lambda\in \Lambda$ and $\omega\in \Omega$ such that $\sigma_\Sigma(s)=\rho_\Lambda(\lambda)$, $\sigma_\Lambda(\lambda)=\rho_\Omega(\omega)$, we have,
\[\varphi_{\Sigma,\Lambda,\ind_\Omega A}(\xi)([s,\lambda])=\xi(s)(\lambda)\in \ind_{\Omega^{\sigma(\lambda)}}A,\]
\[(\varphi_{\Sigma\circ \Lambda,\Omega,A}\circ \varphi_{\Sigma,\Lambda,\ind_\Omega A})(\xi)([[s,\lambda],\omega])=\varphi_{\Sigma,\Lambda,\ind_\Omega A}(\xi)([s,\lambda])(\omega)=\xi(s)(\lambda)(\omega).\]

On another side,
\[\ind_{\Sigma}(\varphi_{\Lambda,\Omega,A})(\xi)(s)\in \ind_{(\Lambda\circ \Omega)^{\sigma(s)}} A,\]
\[\ind_{\Sigma}(\varphi_{\Lambda,\Omega,A})(\xi)(s)([\lambda,\omega])=\xi(s)(\lambda)(\omega),\]
\begin{align*}
    (\varphi_{\Sigma, \Lambda\circ \Omega,A}\circ\ind_{\Sigma}(\varphi_{\Lambda,\Omega,A}))(\xi)([s,[\lambda,\omega]]) & =\ind_{\Sigma}(\varphi_{\Lambda,\Omega,A})(\xi)(s)([\lambda,\omega])\\
    & =\xi(s)(\lambda)(\omega),
\end{align*}
and $a_{\Sigma,\Lambda,\Omega}([[s,\lambda],\omega])=[s,[\lambda,\omega]]$, we proved our claim.
\end{proof}

\begin{prop}[Lax left and right unity]\label{lax unity}
    For any 0-cells $\calG_1,\calG_2\in B_0$, 1-cell $\Omega\in \frgr(\calG_1,\calG_2)$, the following diagrams of natural transformations commute.
\[\xymatrix{
\ind_{\calG_2} \ind_{\Omega} \ar[r]^{\psi_{\calG_2}\star id} \ar[d]_{\varphi_{\calG_2,\Omega}} & \ind_\Omega\\
\ind_{\calG_2\circ \Omega} \ar[ur]_{\ind_{l_\Omega}}
},\quad
\xymatrix{
\ind_\Omega \ind_{\calG_1} \ar[r]^{id\star \psi_{\calG_1}} \ar[d]_{\varphi_{\Omega,\calG_1}} & \ind_\Omega \\
\ind_{\Omega\circ \calG_1} \ar[ur]_{\ind_{r_\Omega}}
}
\]
\end{prop}
\begin{proof}
We know that $l_\Omega:\calG_2\circ \Omega,[\gamma,\omega]\mapsto \gamma \omega$ is a $\calG_2,\calG_1$-equivariant homeomorphism, so for any $A\in \kk^{\calG_1}$, $\ind_{l_\Omega,A}$ is an element of $\kk^{\calG_1}(\ind_{\calG_2\circ\Omega}A,\ind_\Omega A)$ induced by the $\calG_2$-equivariant *-isomorphism
\[(l_\Omega\inv)^*:\ind_{\calG_2\circ\Omega}A\ra\ind_\Omega A, \]
\[(l_\Omega\inv)^*(\xi)(\omega)=\xi(l_{\Omega}\inv(\omega))=\xi([\rho_\Omega(\omega),\omega])\]

For any $A\in \kk^{\calG_1}$, $\xi\in \ind_{\calG_2}\ind_\Omega A$, $\omega\in \Omega$,
\[((l_\Omega\inv)^*\circ\varphi_{\calG_2,\Omega,A})(\xi)(\omega)=\varphi_{\calG_2,\Omega,A}(\xi)([\rho_\Omega,\omega])=\xi(\rho_\Omega(\omega))(\omega),\]
\[\psi_{\calG_2,\ind_\Omega A}(\xi)(\omega)=\xi(\rho_\Omega(\omega))(\omega).\]
Therefore, we have an equality of $\calG_2$-equivariant *-isomorphisms
\[\psi_{\calG_2,\ind_\Omega A}= (l_\Omega\inv)^*\circ \varphi_{\calG_2,\Omega,A},\]
which deduces that the diagram of natural transformations commute. So we proved the lax left unity. We can prove the lax right unity in the same way.
\end{proof}

Finally, we conclude our construction of the pseudofunctor.

\begin{thm}\label{pseudofunctor}
    Assign every $\calG\in B_0$ to $\kk^\calG\in B_0'$. These data $(\ind, \varphi,\psi)$ defined above constitute a pseudofunctor
    \[(\ind, \varphi,\psi):\frgr\ra \mathfrak{KK}.\]
\end{thm}

\begin{proof}
On objects, we assign every $\calG\in B_0$ to $\kk^\calG\in B_0'$.
On hom categories, for each pair $\calG,\calH\in B_0$, we have a local functor $\ind:\frgr(\calH,\calG)\ra \mathrm{Fun}(\kk^\calH,\kk^\calG)$ as described in proposition \ref{ind is functor}.
As proved in proposition \ref{Lax functoriality constraint}, the natural isomorphism $\varphi$ satisfies the conditions to be a lax functoriality constraint. As proved in proposition \ref{Lax unity constraint}, the natural isomorphism $\psi$ satisfies the conditions to be a lax unity constraint.
Then as proved in proposition \ref{lax associativity} and proposition \ref{lax unity}, the above data satisfy the lax associativity, lax left and right unity.

So these data constitute a pseudofunctor (c.f. \cite[Definition 4.1.2]{johnson20212}).
\end{proof}

\subsection{Internal adjunctions}

Internal adjunctions (see definition 6.1.1 of \cite{johnson20212}) are generalizations of adjoint functors for general bicategories. As an example, an internal adjunction in $\mathfrak{KK}$ is a pair of adjoint functors $F\dashv G$ between Kasparov categories. That is, a quadruple $(F,G,\eta,\epsilon)$ consisting of
    \begin{enumerate}
        \item two second countable \'etale groupoids $\calG,\calH\in B_0$;
        \item two functors $F:\kk^\calH\ra \kk^\calG$ and $G:\kk^\calG\ra \kk^\calH$;
        \item a natural transformation $\eta: id_{\kk^\calH}\ra GF$ (which is called unit), which is a family $(\eta_A)_{A\in \kk^\calH}$, such that $\eta_A\in \kk^\calH(A,GFA)$ for each separable $\calH$-\cst-algebra $A$, and for any $A_1,A_2\in \kk^\calH$ and $x\in \kk^\calH(A_1,A_2)$, we have $x\otimes \eta_{A_2}=\eta_{A_1}\otimes GF(x)$;
        \item a natural transformation $\epsilon: FG\ra id_{\kk^\calG}$ (which is called counit), which is a family $(\epsilon_B)_{B\in \kk^\calG}$, such that $\eta_B\in \kk^\calG(FGB,B)$ for each separable $\calG$-\cst-algebra $B$, and for any $B_1,B_2\in \kk^\calG$ and $y\in \kk^\calG(B_1,B_2)$, we have $FG(y)\otimes \epsilon_{B_2}=\epsilon_{B_1}\otimes y$,
    \end{enumerate}
    such that the following diagrams commutes (known as triangle identities).
    \[\xymatrix{
    F \ar[r]^{id_F\star \eta} \ar[dr]_{id_F} & FGF \ar[d]^{\epsilon\star id_F}\\
    & F
    } \quad \xymatrix{
    G \ar[r]^{\eta\star id_G} \ar[dr]_{id_G} & GFG \ar[d]^{id_G\star \epsilon}\\
    & G
    }
    \]

\begin{prop}\label{adjunction in KK}
    Use the same notation as above, for any $A\in \kk^\calH$, $B\in \kk^\calG$, we have a bijection
    \[\kk^\calG(FA,B)\ra \kk^\calH(A,GB), x\mapsto \eta_A\otimes G(x).\]
\end{prop}
\begin{proof}
    The description of an internal adjunction follows easily from the definition. Now let $p$ be the map $\kk^\calG(FA,B)\ra \kk^\calH(A,GB), x\mapsto \eta_A\otimes G(x)$, let $q$ be the map $\kk^\calH(A,GB)\ra \kk^\calG(FA,B), y\mapsto F(y)\otimes \epsilon_B$. The two triangle identities implies that
    \[F(\eta_A)\otimes \epsilon_{FA}=1_{FA},\quad \eta_{GB}\otimes G(\epsilon_B)=1_{GB},\]
    Then use the two formulas and naturalities of $\eta, \epsilon$, for any $x\in \kk^\calG(FA,B)$ and $y\in \kk^\calH(A,GB)$,
    \[q(p(x)) = F(\eta_A)\otimes FG(x)\otimes \epsilon_B= F(\eta_A)\otimes \epsilon_{FA}\otimes x=x.\]
    \[p(q(y)) = \eta_A\otimes GF(y)\otimes G(\epsilon_B) = y\otimes \eta_{GB}\otimes G(\epsilon_B)=y.\]
    Hence, $p,q$ are inverse to each other.
\end{proof}

Following the definition, we can easily describe the internal adjunctions in $\frgr$.

\begin{prop}
    An internal adjunction $\Omega\dashv \Lambda$ in $\frgr$ is a quadruple $(\Omega,\Lambda,\eta, \epsilon)$ consisting of
    \begin{enumerate}
        \item two second countable \'etale groupoids $\calG,\calH\in B_0$,
        \item two second countable locally compact Hausdorff correspondences $\Omega:\calG\leftarrow\calH$ and $\Lambda:\calH\leftarrow\calG$,
        \item an $\calH,\calH$-equivariant continuous map $\eta:\calH\ra \Lambda\times_{\calG}\Omega$, 
        \item a $\calG,\calG$-equivariant continuous map $\epsilon: \Omega\times_{\calH} \Lambda\ra \calG$,
    \end{enumerate}
    such that the following two diagrams commute (known as triangle identities):
    \[\xymatrix{
    \Omega\circ \calH \ar[r]^{\hspace{-0.5cm}[id,\eta]} \ar[ddrr]_{r_{\Omega}} & \Omega\circ (\Lambda\circ \Omega)\ar[r]^{a_{\Omega,\Lambda,\Omega}\inv} & (\Omega\circ \Lambda)\circ \Omega \ar[d]^{[\epsilon,id]}\\
    & & \calG\circ \Omega \ar[d]^{l_{\Omega}}\\
    & & \Omega  
    }
    \xymatrix{
    \calH\circ \Lambda\ar[r]^{\hspace{-0.5cm}[\eta,id]} \ar[ddrr]_{l_{\Lambda}}& (\Lambda\circ \Omega)\circ \Lambda\ar[r]^{a_{\Lambda,\Omega,\Lambda}}  & \Lambda\circ (\Omega\circ \Lambda)\ar[d]^{[id,\epsilon]}\\
    & & \Lambda\circ \calG\ar[d]^{r_{\Lambda}}\\
    & & \Lambda 
    }\]
\end{prop}

\begin{prop}\label{adjointness of induction functors}
    If $(\Omega,\Lambda,\eta,\epsilon)$ is an internal adjunction in $\frgr$, then $\ind_{\Omega}\dashv \ind_\Lambda$ is an internal adjunction in $\mathfrak{KK}$.
\end{prop}
\begin{proof} As proved in theorem \ref{pseudofunctor}, $(\ind,\varphi,\psi)$ is a pseudofunctor from $\frgr$ to $\mathfrak{KK}$. Suppose that $\Omega$ is a correspondence $\calG\leftarrow \calH$ and $\Lambda$ is a correspondence $\calH\leftarrow \calG$. By proposition 6.1.7 of \cite{johnson20212}, since pseudofunctor preserves adjunction, $\ind_{\Omega}$ is left adjoint to $\ind_{\Lambda}$, with unit given by the composite
\[\bar\eta:id_{\kk^\calH}\xrightarrow{\psi_\calH\inv}\ind_{\calH}\xrightarrow{\ind_\eta}\ind_{\Omega\circ\Lambda}\xrightarrow{\varphi_{\Omega,\Lambda}\inv}\ind_{\Omega}\ind_{\Lambda},\]
and counit given by the composite
\[\bar\epsilon:\ind_{\Lambda}\ind_{\Omega}\xrightarrow{\varphi_{\Lambda,\Omega}}\ind_{\Lambda\circ \Omega}\xrightarrow{\ind_\epsilon}\ind_{\calG}\xrightarrow{\psi_\calG}id_{\kk^\calG}.\]
\end{proof}

\subsection{Induction-restriction adjunction}

The following internal adjunction in $\frgr$ can be seen as the origin of the induction-restriction adjunction.

\begin{prop}\label{relatively clopen subgrpd become example of adjunction}
    Let $\calG$ be a second countable \'etale groupoid, $\calH$ be a relatively clopen subgroupoid. Suppose that $\Omega=\calG_{\calH\units}$, equipped with left action of $\calG$ and right action of $\calH$ by multiplication, $\Lambda=\calG^{\calH\units}$ equipped with left action of $\calH$ and right action of $\calG$ by multiplication. Define
    \[\eta:\calH\ra \Lambda\circ \Omega, h\mapsto [r(h),h]=[h,s(h)],\]
    \[\epsilon:\Omega\circ \Lambda\ra \calG, [\omega,\lambda]\mapsto \omega\lambda.\]
    Then $\Omega,\Lambda$ are well-defined 1-cells in $\frgr$, $\eta,\epsilon$ are well-defined 2-cells in $\frgr$, and $(\Omega,\Lambda,\eta,\epsilon)$ is an internal adjunction in $\frgr$.
\end{prop}
\begin{proof} By lemma \ref{relative clopen grpd has natural correspondece}, $\Omega$ is a well-defined correspondence. The groupoid $\calG$ acts on $\Lambda$ by multiplication, so this action is free. Since $\Lambda$ is open in $\calG$, $s_{\calG}|_{\Lambda}$ is a local homeomorphism, so the action of $\calG$ is \'etale. For any compact $K\subseteq \Lambda$, $\{\gamma\in \calG:K\cap K\gamma\neq \emptyset\}=K\inv K$ is compact, so the action of $\calG$ is proper. In conclusion $\Lambda$ is also a well-defined correspondence. It is easy to see that $\eta,\epsilon$ are well-defined bi-equivariant continuous maps.

Check the two triangle identities. For any $[\omega,h]\in \Omega\circ \calH$,
\begin{align*}
    (l_{\Omega}\circ [\epsilon,id]\circ a\inv \circ [id,\eta])([\omega,h]) & = (l_{\Omega}\circ [\epsilon,id]\circ a\inv)([\omega,[h,s(h)]]) \\
    & = (l_{\Omega}\circ [\epsilon,id])([[\omega,h],s(h)])\\
    & = l_{\Omega}([\omega h, s(h)])\\
    & =\omega h= r_{\Omega}([\omega,h]).
\end{align*}
So the first triangle commutes. For any $[h,\lambda]\in \calH\circ \Lambda$,
\begin{align*}
    (r_{\Lambda}\circ [id,\epsilon]\circ a\circ [\eta,id])([h,\lambda]) & = (r_{\Lambda}\circ [id,\epsilon]\circ a)([[h,s(h)],\lambda])\\
    & = (r_{\Lambda}\circ [id,\epsilon])([h,[s(h),\lambda]])\\
    & = r_{\Lambda}([h,\lambda])\\
    & = h\lambda = l_{\Lambda}([h,\lambda]).
\end{align*}
So the second triangle commutes.\end{proof}

\begin{lem}\label{res-ind equiv}
    Use the same notation as above, let $res^\calG_\calH:\kk^\calG\ra\kk^\calH, A\mapsto A|_\calH=A|_{\calH\units}$ be the restriction functor (induced by the strict morphism $\calH\hookrightarrow\calG$). Then there is a natural isomorphism $\upsilon:res^\calG_\calH\ra \ind_\Lambda$, constitute of $\calH$-equivariant *-isomorphisms \[\upsilon_A:A|_\calH\ra \ind_{\Lambda}A, \upsilon_A(a)(\lambda)=\alpha_{\lambda\inv}(a(r(\lambda))),\forall a\in A|_{\calH},\lambda\in \Lambda\]
    for $(A,\alpha)\in \kk^\calG$.
\end{lem}
\begin{proof}
    For any $\lambda\in \Lambda$, $a\in A|_\calH$, $\alpha_{\lambda\inv}(a(r(\lambda)))\in A_{s(\lambda)}=A_{\sigma_\Lambda(\lambda)}$. Lemma \ref{continuity of groupoid action} implies that $\upsilon_A(a)$ is a continuous bounded section in $\Gamma_b(\Lambda,\sigma_\Lambda^*\ca A)$. So for any $\gamma \in \calG$ such that $\lambda$ and $\gamma\inv$ are composable,
    \[\upsilon_A(a)(\lambda\gamma\inv)=\alpha_{\gamma\lambda\inv}(a(r(\lambda \gamma\inv)))=\alpha_\gamma(\upsilon_A(a)(\lambda)).\]
    And $\Lambda/\calG$ can be canonically identified with $\calH\units$ by $\bar r:\lambda\calG\mapsto r(\lambda)$. The map $\lambda\calG\mapsto \mathbb R, \lambda\calG\mapsto \|\upsilon_A(\lambda)\|=\|a(r(\lambda)\|$ is identified with the map $\calH\units\ra\mathbb R, x\mapsto \|a(x)\|$, which vanishes at infinity. In conclusion $\upsilon_A(a)$ is a well-defined element in $\ind_\Lambda A$. It is easy to check that $\upsilon_A$ is an $\calH$-equivariant *-isomorphism.

    For any $\calG$-equivariant *-homomorphism $\phi:A\ra B$, where $(A,\alpha), (B,\beta)\in \kk^\calG$, for any $a\in A|_\calH$ and $\lambda\in A$,
    \begin{align*}
        \upsilon_B(\phi(a))(\lambda) & =\beta_{\lambda\inv}(\phi(a)(r(\lambda)))\\
        & = \beta_{\lambda\inv}\circ \phi_{r(\lambda)}(a(r(\lambda)))\\
        & =\phi_{r(\lambda)}\circ \alpha_{\lambda\inv}(a(r(\lambda)))\\
        & = (\ind_\Lambda \phi)(\upsilon_A(a))(\lambda).
    \end{align*}
    That is, the following diagram of $\calH$-equivariant *-homomorphisms commutes.
    \[
    \xymatrix{
    A|_\calH \ar[r]^{\upsilon_A}\ar[d]_{res^\calG_\calH(\phi)} &\ind_\Lambda A \ar[d]^{\ind_\Lambda \phi}\\
    B|_\calH \ar[r]^{\upsilon_B} & \ind_\Lambda B
    }
    \]
    So $\upsilon$ is a natural transformation for categories whose morphisms are equivariant *-homomorphisms. Using the decomposition property (theorem A 2.2 of \cite{lafforgue2007k}), $\upsilon$ is therefore a natural transformation for $\kk$-categories.
\end{proof}

\begin{corr}\label{induction-restriction adj}
    Use the same notation as above, $\ind_\Omega:\kk^\calH\ra \kk^\calG$ is left adjoint to $res^\calG_\calH:\kk^\calG\ra \kk^\calH$. 
\end{corr}

The compression isomorphism in theorem \ref{compression isomorphism} can be even reconstructed using these data. We need no longer the properness of $\calH$, and indeed the properness of $\Omega\rtimes\calH$ suffices.

\begin{thm}
    Let $\calG$ be a second countable \'etale groupoid, $\calH$ be a relatively clopen subgroupoid, $\Omega:\calG\leftarrow \calH$ be $\calG_{\calH\units}$. If $A$ is a separable $\calH$-\cst-algebra and $B$ is a separable $\calG$-\cst-algebra, then
    \[comp_\calH^\calG:\kk^\calG(\ind_\Omega A, B)\ra \kk^\calH(A,B|_\calH)\]
    is an isomorphism.
\end{thm}
\begin{proof}
    Let $\Lambda=\calG^{\calH\units}$. By proposition \ref{adjointness of induction functors} and proposition \ref{relatively clopen subgrpd become example of adjunction}, $\ind_\Omega$ is left adjoint to $\ind_\Lambda$, therefore by proposition \ref{adjunction in KK} the following composite is an isomorphism
    \[\kk^\calG(\ind_\Omega A,B)\xrightarrow{\ind_\Lambda}\kk^{\calH}(\ind_\Lambda\ind_\Omega A,\ind_\Lambda B)\xrightarrow{\bar\eta_A\otimes-}\kk^\calH(A,\ind_\Lambda B),\]
    where $\bar\eta_A=[\psi_\calH\inv]\otimes\ind_{\eta,A}\otimes[\varphi_{\Lambda,\Omega,A}\inv] \in \kk^\calH(A,\ind_\Lambda\ind_\Omega A)$.

    If we can show that the following diagram is commutative, then the compression map will be equivalent to this isomorphism.
    \[\xymatrix{
    \kk^\calG(\ind_\Omega A,B) \ar[r]^{\hspace{-0.5cm}\ind_\Lambda} \ar[dr]_{res^\calG_\calH} & \kk^{\calH}(\ind_\Lambda\ind_\Omega A,\ind_\Lambda B) \ar[r]^{\hspace{0.8cm}\bar \eta_A\otimes -} \ar[d]^{[\upsilon_{\ind_\Omega A}]\otimes -\otimes [\upsilon_B\inv]} & \kk^\calH(A,\ind_\Lambda B)\ar[d]^{-\otimes [\upsilon_B\inv]}\\
    & \kk^\calH(\ind_{\calG^{\calH\units}_{\calH\units}}A, B|_{\calH}) \ar[r]_{\hspace{0.8cm}[i_A]\otimes- } & \kk^\calH(A,B|_\calH)
    }\]
    And the left triangle commutes because of lemma \ref{res-ind equiv}. Hence, it suffices to show that the right square commutes, for which we just need to prove that $[i_A]\otimes[\upsilon_{\ind_\Omega A}]=\bar\eta_A$.

    Notice that $\Lambda\circ \Omega$ is $\calH,\calH$-homeomorphic to $\calG^{\calH\units}_{\calH\units}$ through the map $f: [\lambda,\omega]\mapsto \lambda \omega$. Let $\iota:\calH\ra \calG^{\calH\units}_{\calH\units}$ be the inclusion, then $f\circ \eta=\iota$. Therefore, $\ind_{\iota,A}=\ind_{\eta,A}\otimes \ind_{f,A}$.

    The map $\iota$ is an open inclusion, $\ind_{\iota,A}$ is induced by the $\calG$-equivariant inclusion $\ind_{\calH}A\ra\ind_{\calG^{\calH\units}_{\calH\units}}A$ as mentioned in remark \ref{special case of ind 2-cells}. So $[i_A]=[\psi_\calH]\inv\otimes \ind_{\iota,A}$. And $\ind_{\calG^{\calH\units}_{\calH\units}}A$ is identified with $(\ind_\Omega A)|_\calH$,  $\upsilon_{\ind_\Omega A}:\ind_{\calG^{\calH\units}_{\calH\units}}A\ra \ind_\Lambda \ind_\Omega A$ is the map $\upsilon_{\ind_\Omega A}(\eta)(\lambda)(\omega)=[\lambda\inv.\eta|_{\calG^{s(\lambda)}_{\calH\units}}](\omega)=\eta(\lambda \omega)$. Therefore, $(f\inv)^*\circ\upsilon_{\ind_\Omega A}=\varphi_{\Lambda,\Omega,A}\inv$. So $\ind_{f,A}\otimes [\upsilon_{\ind_\Omega A}]=[\varphi_{\Lambda,\Omega,A}\inv]$.

    In conclusion,
    \begin{align*}
        [i_A]\otimes[\upsilon_{\ind_\Omega A}] & = [\psi_\calH\inv]\otimes \ind_{\iota,A}\otimes [\upsilon_{\ind_\Omega A}]\\
        & = [\psi_\calH\inv]\otimes \ind_{\eta,A}\otimes \ind_{f,A}\otimes [\upsilon_{\ind_\Omega A}]\\
        & = [\psi_\calH\inv]\otimes \ind_{\eta,A}\otimes [\varphi_{\Lambda,\Omega,A}\inv] = \bar\eta_A.
    \end{align*}
    We proved our claim.
\end{proof}

\begin{corr}
    Use the same notation as above, for any $n\in \mathbb Z$
    \[comp^\calG_\calH:\kk_n^\calG(\ind_\Omega A, B)\xrightarrow{res^\calG_\calH} \kk_n^\calH(\ind_{\calG^{\calH\units}_{\calH\units}}A,B|_\calH)\xrightarrow{i_A^*}\kk_n^\calH(A,B|_\calH)\]
    is also an isomorphism.
\end{corr}
\begin{proof}
    We can identify $\ind_\Omega A\otimes C_0(\mathbb R)^{\otimes m}$ with $\ind_\Omega(A\otimes C_0(\mathbb R)^{\otimes m})$ as $\calG$-\cst-algebras for any $m\geqslant0$.
\end{proof}

\subsection{Restriction-induction adjunction}

    Based on the same idea, we can prove that in some case the restriction functor is left adjoint to the induction functor (therefore it is an ambidextrous adjunction $res^\calG_\calH\dashv \ind^\calG_\calH\dashv res^\calG_\calH$).
    
    \begin{prop}
        Let $\calG$ be a second countable \'etale groupoid, $\calH$ be a relatively clopen subgroupoid. If the map $\tilde r:\calG_{\calH\units}/\calH\ra\calG\units, \omega\calH\mapsto r(\omega)$ is proper, then we will have an isomorphism
    \[\kk^{\calG}(A, \ind_{\calG_{\calH\units}} B)\cong \kk^\calH(A|_\calH,B)\]
    for all $A\in \kk^\calG$, $B\in \kk^\calH$.
    \end{prop}
      
    Sketch of proof: consider the sub-bicategory $\frgr'$ of $\frgr$ such that, for $\calG_1,\calG_2\in B_0$ and $\Omega_1,\Omega_2\in \frgr(\calG_1,\calG_2)$, a $\calG_2,\calG_1$-equivariant map $f:\Omega_1\ra \Omega_2$ is a 2-cell of $\frgr'$ if and only if $\bar f:\Omega_1/\calG_1\ra \Omega_2/\calG_1, \omega\calG_1\mapsto f(\omega)\calG_1$ is a proper map.
    
    Then for such a 2-cell $f$ in $\frgr'$ and $A\in \kk^{\calG_1}$, there is a natural $\calG_2$-equivariant *-homomorphism
    \[\ind_{f,A}':\ind_{\Omega_2}A\ra \ind_{\Omega_1}A, \xi\mapsto \xi\circ f.\]

    Let $\ind'$ be the contravariant local functor that sends every correspondence $\Omega$ to $\ind_{\Omega}$ and sends every 2-cell $f$ in $\frgr'$ to $\ind_{f}'$. Then $(\ind',\varphi,\psi)$ constitute a pseudofunctor from the \textbf{co-bicategory} of $\mathfrak{Gr'}$ to $\mathfrak{KK}$. (The co-bicategory of $\mathfrak{Gr'}$ is a bicategory obtained by reversing all directions of 1-cells in $\mathfrak{Gr'}$, see \cite[Definition 2.6.3]{johnson20212}.)

    Now let $\calH$ be a relatively clopen subgroupoid of $\calG$, $\Omega=\calG_{\calH\units}$, $\Lambda=\calG^{\calH\units}$ and assume that $\tilde r$ is proper. Let $\eta,\epsilon$ be defined in the same way as in proposition \ref{relatively clopen subgrpd become example of adjunction}. Since $\eta:\calH\ra \Lambda\circ \Omega$ is a closed inclusion, it is a 2-cell in $\frgr'$. It is easy to see that $(\Omega\circ \Lambda)/\calG\ra \Omega/\calH, [\omega,\lambda]\calG\mapsto \omega\calH$ and $\calG/\calG\ra \calG\units, \gamma\calG\mapsto r(\gamma)$ are homeomorphisms, and for the map $\epsilon:\Omega\circ \Lambda\ra \calG$, $\bar\epsilon:(\Omega\circ \Lambda)/\calG\ra \calG/\calG$ is equivalent to $\tilde r$. Hence, if $\tilde r$ is proper, $\epsilon$ is a 2-cell in $\frgr'$. Similar to proposition \ref{relatively clopen subgrpd become example of adjunction}, $(\Omega,\Lambda,\eta,\epsilon)$ is an internal adjunction in $\mathfrak{Gr}'$, which implies that $res^\calG_\calH\cong\ind_\Lambda$ is left adjoint to $\ind_\Omega$ (the adjointness is reversed because the local functors here are contravariant.)

\section{Going-down principle and applications}

In this section we develop a general framework that allow us to reduce the problems about the topological K-theory of an \'etale groupoid to the same problems for all its proper open subgroupoids.

\subsection{Going-down functors}
Let $\calG$ be a second countable \'etale groupoid. We denote by $\ca S(\calG)$ the set of all proper open subgroupoids of $\calG$ and $\calG$ itself. We denote by $\ca C(\calG)$ the category whose objects are separable commutative proper $\calG$-\cst-algebras and whose morphisms are $\calG$-equivariant *-homomorphisms. For every $A\in \ca C(\calG)$, there exists a second countable locally compact Hausdorff proper $\calG$-space $Z$ such that $A\cong C_0(Z)$.
\begin{defn}
    Let $\calG$ be a second countable \'etale groupoid. A going-down functor for $\calG$ is a collection of $\mathbb Z$-graded functors $F=(F^n_\calH)_{\calH\in \ca S(\calG)}$, where every $F^n_\calH$ is a contravariant additive functor from $\ca C( \calG)$ to the category $\mathrm{Ab}$ of Abelian groups, such that the following axioms are satisfied:
    \begin{enumerate}
        \item \textbf{Cohomological axioms}: For every $\calH\in \ca S(\calG)$, $n\in \mathbb Z$,
        \begin{enumerate}
            \item the functor $F^n_\calH$ is homotopy invariant;
            \item the functor $F^n_\calH$ is half-exact. That is, a short exact sequence $0\ra I\ra A\ra A/I\ra 0$ in $\ca C(\calG)$ induces that
            \[F^n_\calH(A/I)\ra F^n_\calH(A)\ra F^n_\calH(I)\]
            is exact in the middle;
            \item (suspension) we denote by $\Sigma_\calH: \ca C(\calH)\ra \ca C(\ca H), A\mapsto A\otimes C_0(\mathbb R)$ the suspension functor, where $\calH$ acts trivially on the factor $C_0(\mathbb R)$. Then there is a natural isomorphism 
            \[S^n_\calH: F^{n+1}_\calH\ra F^n_\calH\circ \Sigma_\calH.\]
            
        \end{enumerate}
        \item \textbf{Induction axiom}: for any $\calH_1, \calH_2\in \ca S(\calG)$ such that $\calH_1\subseteq\calH_2$ and $n\in \mathbb Z$, there is a natural isomorphism $I^{\calH_2}_{\calH_1}(n):F^n_{\calH_1}\ra F^n_{\calH_2}\circ \ind_{\calH_1}^{\calH_2}$, compatible with suspension. Here $\ind_{\calH_1}^{\calH_2}:\ca C(\calH_1)\ra \ca C (\calH_2)$ is the induction functor that send $C_0(Z)$ to $\ind_\Omega C_0(Z)\cong C_0(\Omega\times_{\calH_1} Z)$, where $\Omega=({\calH_2})_{\calH_1\units}$ is the correspondence $\calH_2\leftarrow \calH_1$ as mentioned in example \ref{subgroupoid induction functor}.
    \end{enumerate}
\end{defn}

We remark that there is a natural isomorphism $\ind^{\calH_2}_{\calH_1}\circ \Sigma_{\calH_1}\ra \Sigma_{\calH_2}\circ \ind^{\calH_2}_{\calH_1}$, since for every $\calH_1$-space $Z$, there is canonically an $\calH_2$-equivariant homeomorphism $\Omega\times_{\calH_1}(Z\times\mathbb R)\cong (\Omega\times_{\calH_1}Z)\times \mathbb R$.

If $F$ is a going-down functor for $\calG$, we define
\[F^n(\ca G):=\varinjlim_{Z\subseteq \E\calG}F^n_\calG(C_0(Z)),\]
where $\E\calG$ is a model of classifying space for proper actions of $\calG$, $Z$ runs over all $\calG$-compact subspaces of $\E\calG$. The space $\E\calG$ has the following universal property: for any locally compact Hausdorff proper $\calG$- space $Z$, there exists a $\calG$-equivariant continuous map $Z\ra \E\calG$ which is unique up to $\calG$-homotopy. Therefore, $\E\calG$ is unique up to $\calG$-homotopic equivalence. Recall that all $\calG$-equivariant continuous maps between $\calG$-compact spaces are proper by lemma \ref{maps between G-compact sets are proper}.

\begin{defn}
    Let $F, G$ be two going-down functors for a second countable \'etale groupoid $\calG$. A going-down transformation $\Lambda:F\ra G$ is a collection $(\Lambda^n_\calH)_{\calH\in \ca S(\calG)}$ of natural transformations $\Lambda^n_\calH\in \mathrm{Nat}(F^n_\calH, G^n_\calH)$ that is compatible with suspension and induction, i.e. for any elements $\calH, \calH_1\subseteq\calH_2$ of $\ca S(\calG)$ and $n\in \mathbb Z$,  the following diagrams of natural transformations commute.
    \[\xymatrix{
    F^n_{\calH_1} \ar[r]^{\hspace{-0.8cm}I^{\calH_2}_{\calH_1}(n)} \ar[d]_{\Lambda^n_{\calH_1}} & F^n_{\calH_2}\circ \ind_{\calH_1}^{\calH_2} \ar[d]^{\Lambda^n_{\calH_2}\star id} \\
    G^n_{\calH_1} \ar[r]^{\hspace{-0.8cm}I^{\calH_2}_{\calH_1}(n)} & G^n_{\calH_2}\circ \ind_{\calH_1}^{\calH_2}
    }\quad
    \xymatrix{
    F^{n+1}_\calH \ar[r]^{S_\calH^n} \ar[d]^{\Lambda^{n+1}_\calH} & F^n_\calH \circ \Sigma_\calH\ar[d]^{\Lambda^n_\calH\star id} \\
    G^{n+1}_\calH \ar[r]^{S_\calH^n} & G^n_\calH \circ \Sigma_\calH
    }\]
\end{defn}

For a going-down transformation $\Lambda$, we define $\Lambda^n(\calG):F^n(\calG)\ra G^n(\calG)$ as the inductive limit of $\Lambda^n_\calG(C_0(Z)): F^n_\calG(C_0(Z))\ra G^n_\calG(C_0(Z))$ over all $\calG$-compact subsets $Z$ of $\E\calG$.

\begin{lem}\label{going-down functor long exactness}
    Let $F$ be a going-down functor for a second countable \'etale groupoid $\calG$, $\calH\in \ca S(\calG)$. For every short exact sequence
    \[0\ra I\xrightarrow{i} A\xrightarrow{q} A/I\ra 0\]
    is $\ca C(\calH)$, there are natural boundary maps $\partial_n : F^n_\calH(I)\ra F^{n+1}_\calH(A/I)$, such that there is a long exact sequence
    \[\cdots \ra F^n_\calH(A/I)\xrightarrow{F^n_\calH(q)} F^n_\calH(A)\xrightarrow{F^n_\calH(i)} F^n_\calH(I)\xrightarrow{\partial_n}F^{n+1}_\calH(A/I)\ra \cdots.\]

    Moreover, if $\Lambda:F\ra G$ is a going-down transformation, then the following diagram commutes.
    \[
    \xymatrix{
    \cdots \ra F^n_\calH(A/I) \ar[r] \ar[d]^{\Lambda^n_\calH(A/I)}  & F^n_\calH(A) \ar[r] \ar[d]^{\Lambda^n_\calH(A)} & F^n_\calH (I) \ar[r]^{\hspace{-0.5cm}\partial_n} \ar[d]^{\Lambda^n_\calH(I)} & F^{n+1}_\calH(A/I)\ra \cdots \ar[d]^{\Lambda^{n+1}_\calH(A/I)} \\
    \cdots \ra G^n_\calH(A/I) \ar[r] & G^n_\calH(A) \ar[r] & G^n_\calH(I) \ar[r]^{\hspace{-0.5cm}\partial_n} & G^{n+1}_\calH(A/I) \ra \cdots
    }\]
\end{lem}
\begin{proof}
The exactness can be proved by following the standard homotopy argument as in \cite[Section 21.4]{blackadar1998k}. Now in order to prove the commutativity of the diagram, we recall the construction of the boundary maps. Let $C_q=\{(a,f)\in A\oplus C_0([0,1),A/I):q(a)=f(0)\}$, $e: I\ra C_q, i\mapsto (i,0)$ is an $\calH$-equivariant *-homomorphism. We have an exact sequence in $\ca C(\calH)$,
\[0\ra I\xrightarrow{e} C_q\ra C_0([0,1),A/I)\ra 0.\]

Since $C_0([0,1),A/I)$ is contractible, we can prove that $F^n_\calH(e):F^n_\calH(C_q)\ra F^n_\calH(I)$ is an isomorphism by using the same method in \cite[Proposition 21.4.1]{blackadar1998k}, 

Define $j:\Sigma_\calH(A/I)\cong C_0((0,1),A/I)\ra C_q, f\mapsto (0,f)$, which is a well-defined $\calH$-equivariant *-homomorphism. The boundary map $\partial_n:F^n_\calH(I)\ra F^{n+1}_\calH(A/I)$ is defined as the composite $S^n_\calH(A/I)\inv\circ F^n_\calH(j)\circ F^n_\calH(e)\inv$.
Now in the diagram, the left two squares commute because of naturality of $\Lambda^n_\calH$. The right square commutes because of naturality of $\Lambda^n_\calH$ and compatibility of $\Lambda$ with suspension.
\end{proof}

\begin{corr}[Mayer-Vietoris] \label{going-down MV-seq}
Let $F$ be a going-down functor for a second countable \'etale groupoid $\calG$, $\calH\in \ca S(\calG)$. For a pullback diagram
\[\xymatrix{
A \ar[r]^{g_1}\ar[d]^{g_2} & A_1\ar[d]^{f_1}\\
A_2 \ar[r]^{f_2} & B
}
\]
in $\ca C(\calH)$ such that $f_1,f_2$ are surjective and $A=\{(a_1,a_2)\in A_1\oplus A_2:f_1(a_1)=f_2(a_2)\}$, there is a long-exact sequence
\[\cdots\ra F^n_\calH(B)\ra F^n_\calH(A_1)\oplus F^n_\calH(A_2) \ra F^n_\calH(A)\ra  F^{n+1}_\calH(B)\ra\cdots\]
If $\Lambda: F\ra G$ is a going-down transformation, then the following diagram commutes.
\[\xymatrix{
\cdots \ra   F^n_\calH(B)\ar[r]\ar[d] &  F^n_\calH(A_1)\oplus  F^n_\calH(A_2) \ar[r]\ar[d] &  F^n_\calH(A) \ar[r]\ar[d] &  F^{n+1}_\calH(B)\ra \cdots \ar[d] \\
\cdots \ra  G^n_\calH(B) \ar[r] & G^n_\calH(A_1)\oplus G^n_\calH(A_2) \ar[r] & G^n_\calH(A) \ar[r] & G^{n+1}_\calH(B) \ra \cdots
}
\]
\end{corr}

The following two lemmas can be seen as analog of proposition \ref{Rips complexes G-homotopy sufficient 1} and proposition \ref{Rips complexes G-homotopy sufficient 2}.

\begin{lem}\label{property of EG}
    Let $X$ be a local compact Hausdorff proper cocompact $\calG$-space, then there exists a $\calG$-compact subset $Z$ of $\E\calG$, such that there exists a $\calG$-equivariant continuous map $f:X\ra Z$.
\end{lem}
\begin{proof}
    By the universal property, there exists $\calG$-equivariant continuous map $g:Z\ra \E\calG$. Let $K$ be a compact subset of $Z$ such that $Z=\calG K$, then $g(X)=\calG g(K)$ is a $\calG$-compact subset of $\E\calG$. We can take $Z=g(X)$.
\end{proof}

\begin{lem}\label{cofinality in cocompact subsets of EG}
    Let $X$ be a local compact Hausdorff proper cocompact $\calG$-space, $Z_1$ and $Z_2$ be two $\calG$-compact subsets of $\E \calG$ and $f:X\ra Z_1$, $g: X\ra Z_2$ be two $\calG$-equivariant continuous maps. Then there exists a $\calG$-compact subset $Z$ of $\E \calG$ such that $Z\supseteq Z_1\cup Z_2$, and the following diagram commutes up to $\calG$-homotopy.
    \[\xymatrix{
    X \ar[r]^{f} \ar[d]^{g} & Z_1 \ar@{^{(}->}[d] \\
    Z_2 \ar@{^{(}->}[r] & Z
    }\]
\end{lem}
\begin{proof}
    Let $\iota_1:Z_1\hookrightarrow \E\calG$ and $\iota_2:Z_2\hookrightarrow \E\calG$ be inclusions. By the universal property of $\E\calG$, $\iota_1\circ f$ is $\calG$-homotopic to $\iota_2\circ g$. Let $F:X\times [0,1]\ra \E\calG$ be the $\calG$-homotopy such that $F|_{X\times\{0\}}=\iota_1\circ f$ and $F|_{X\times\{1\}}=\iota_2\circ g$. Then $F(X\times [0,1])$ is $\calG$-compact subset of $\E\calG$. We can take $Z=F(X\times [0,1])\cup Z_1\cup Z_2$.
\end{proof}

The following result with proposition \ref{Rips complexes G-homotopy sufficient 1}, proposition \ref{Rips complexes G-homotopy sufficient 2} implies that the Rips complexes will be a sufficient family of proper $\calG$-compact $\calG$-spaces.

\begin{prop}\label{enough to run over rips complex}
    Let $F,G$ be going-down functors for a second countable \'etale groupoid $\calG$ and $\Lambda:F\ra G$ be a going-down transformation. Then $F^n(\calG)$ is isomorphic to $\varinjlim_{K\subseteq\calG} F^n_\calG(C_0(P_K(\calG))) $, $G^n(\calG)$ is isomorphic to $\varinjlim_{K\subseteq\calG} G^n_\calG(C_0(P_K(\calG))) $ and $\Lambda^n(\calG)$ is equivalent to $\varinjlim_{K\subseteq \calG} \Lambda^n_\calG(C_0(P_K(\calG)))$ (here $K$ runs over compact subsets of $\calG$).
\end{prop}
\begin{proof}
    By \cite[Proposition 11.4]{tu1999conjecture-bc}, there exists a model of $\E \calG$, which is second countable locally compact Hausdorff. Before all for convenience of writing, when $Z_1\subseteq Z_2$ are $\calG$-compact subsets of $\E\calG$, we use $\iota_{Z_1,Z_2}$ to denote the inclusion $Z_1\hookrightarrow Z_2$.
    
    Firstly for a $\calG$-compact subset $Z$ of $\E\calG$, by proposition \ref{Rips complexes G-homotopy sufficient 1}, there exists a compact $K_0$ of $\calG$ and a $\calG$-equivariant continuous map $f:Z\ra P_{K_0}(\calG)$. We define $\psi_{Z,F^n}$ as the composite
    \[F^n_\calG(C_0(Z))\xrightarrow{F^n_\calG(f^*)} F^n_\calG(C_0(P_{K_0}(\calG)))\ra \varinjlim_{K\subseteq \calG} F^n_\calG(C_0(P_K(\calG))),\]
    by proposition \ref{Rips complexes G-homotopy sufficient 2} and $\calG$-homotopy invariance of $F^n_\calG$, the map $\psi_{Z,F^n}$ does not depend on the choice of $K_0$ and $f$.

    If $Z_1\subseteq Z_2$ are two $\calG$-compact subsets of $\E\calG$, $K_1, K_2$ are compact subsets of $\calG$ and $f_1:Z_1\ra P_{K_1}(\calG)$ and  $f_2:Z_2\ra P_{K_2}(\calG)$ are $\calG$-equivariant continuous maps, then by proposition \ref{Rips complexes G-homotopy sufficient 2}, there exists $K_3\supseteq K_1\cup K_2$ such that the following diagram commutes up to $\calG$-homotopy.
    \[\xymatrix{
    X_1 \ar[r]^{f_1} \ar[d]_{\iota_{Z_1,Z_2}} & P_{K_1}(\calG) \ar[r]^{\iota_{K_1,K_3}} & P_{K_3}(\calG) \\
    X_2 \ar[r]_{f_2} & P_{K_2}(\calG) \ar[ur]_{\iota_{K_2,K_3}}
    }
    \]
    This implies that $\psi_{Z_1,F^n}\circ F^n_\calG(\iota^*_{Z_1,Z_2})=\psi_{Z_2,F^n}$. Therefore, we can define the map
    \[\psi_{F^n}:=\varinjlim_{Z\subseteq\E\calG}\psi_{Z,F^n}:\varinjlim_{Z\subseteq \E\calG}F^n_\calG(C_0(Z))\ra \varinjlim_{K\subseteq \calG} F^n_\calG(C_0(P_K(\calG))).\]

    For any compact subset $K$ of $\calG$, since $P_k(\calG)$ is a proper $\calG$-compact $\calG$-space, by lemma \ref{property of EG}, there exists a $\calG$-compact subset $Z_0$ of $\E\calG$ and a $\calG$-equivariant continuous map $g:P_K(\calG)\ra Z_0$. We define $\phi_{K,F^n}$ as the composite
    \[F^n_\calG(C_0(P_K(\calG)))\xrightarrow{F^n_\calG(g^*)} F^n_\calG(C_0(Z_0))\ra \varinjlim_{Z\subseteq \E\calG}F^n_\calG(C_0(Z)).\]
    Similarly, use proposition \ref{cofinality in cocompact subsets of EG}, the map $\phi_{K,F^n}$ does not depend on choice of $Z_0$ and $g$, and $\phi_{K_1,F^n}\circ F^n_\calG(\iota_{K_1,K_2}^*)=\phi_{K_2,F^n}$ for any $K_1\subseteq K_2$ compact subsets of $\calG$. We define
    \[\phi_{F^n}:=\varinjlim_{K\subseteq \calG}\phi_{K,F^n}:\varinjlim_{K\subseteq \calG} F^n_\calG(C_0(P_K(\calG)))\ra \varinjlim_{Z\subseteq \E\calG}F^n_\calG(C_0(Z)).\]
    
    Claim: $\psi_{F^n}$ are $\phi_{F^n}$ are inverse to each other. Let $Z$ be any $\calG$-compact subset of $\E\calG$, let $K$ be a compact subset of $\calG$ and $f:Z\ra P_K(\calG)$ be a $\calG$-equivariant continuous map, and let $Z'$ be a $\calG$-compact subset of $\E$, which contains $Z$ and $g:P_K(\calG)\ra Z'$ be a $\calG$-equivariant continuous map. By lemma \ref{cofinality in cocompact subsets of EG}, there exists a $\calG$-compact subset $Z''$ of $\E\calG$ such that $Z''\supseteq Z'$ and the following diagram commutes up to $\calG$-homotopy.
    \[\xymatrix{
    Z \ar[r]^{f} \ar[d]_{\iota_{Z,Z'}} & P_K(\calG) \ar[r]^{g} & Z' \ar[d]^{\iota_{Z',Z''}}\\
    Z' \ar[rr]_{\iota_{Z',Z''}} & & Z''
    }\]
    And this implies that $\phi_{F^n}\circ \psi_{F^n}=id$. Similarly, using proposition \ref{Rips complexes G-homotopy sufficient 2} we can prove that $\psi_{F^n}\circ \phi_{F^n}=id$. We can define
    \[\psi_{G^n}:\varinjlim_{Z\subseteq \E\calG}G^n_\calG(C_0(Z))\ra \varinjlim_{K\subseteq \calG} G^n_\calG(C_0(P_K(\calG)))\]
    \[\phi_{G^n}:\varinjlim_{K\subseteq \calG} G^n_\calG(C_0(P_K(\calG)))\ra \varinjlim_{Z\subseteq \E\calG}G^n_\calG(C_0(Z))\]
    in the same way, and they are inverse to each other. Since $\Lambda^n_\calG:F^n_\calG\ra G^n_\calG$ is a natural transformation, we have the following commutative diagram.
    \[\xymatrix{
    \varinjlim_{Z\subseteq \E\calG}F^n_\calG(C_0(Z)) \ar[r]^{\psi_{F^n}} \ar[d]_{\varinjlim_{K\subseteq \calG} \Lambda^n_\calG(C_0(P_K(\calG)))} & \varinjlim_{K\subseteq \calG} F^n_\calG(C_0(P_K(\calG))) \ar[d]^{\Lambda^n(\calG)} \\
    \varinjlim_{Z\subseteq \E\calG}G^n_\calG(C_0(Z)) \ar[r]_{\psi_{G^n}} & \varinjlim_{K\subseteq \calG} G^n_\calG(C_0(P_K(\calG)))
    }\]
    In this diagram $\psi_{F^n}$ and $\phi_{G^n}$ are isomorphisms, hence the two vertical arrows are equivalent.
\end{proof}

\begin{rem}
    Proposition \ref{Rips complexes G-homotopy sufficient 1}, proposition \ref{Rips complexes G-homotopy sufficient 2}, lemma \ref{property of EG} and lemma \ref{cofinality in cocompact subsets of EG} imply that the inclusion functor from the category of Rips complexes of $\calG$ to the category of $\calG$-compact subspaces of $\E\calG$ is final. The proof here is a standard argument. (See \cite[section IX.3, Theorem 1]{MacLane}.)
\end{rem}

\subsection{Going-down principle}

In this section we will prove our main theorem as following.

\begin{thm}\label{going-down principle}
    Let $ F, G$ be two going-down functors for a second countable \'etale groupoid $\calG$, and $\Lambda: F\ra G$ be a going-down transformation. Suppose that for all proper open subgroupoids $\calH\subseteq \calG$ and for all $n\in \mathbb Z$,
    \[
    \Lambda^n_\calH(C_0(\calH\units)):  F^n_\calH(C_0(\calH\units))\ra G^n_\calH(C_0(\calH\units))
    \]
    is an isomorphism. 
    
    Then for all $n\in \mathbb Z$, $\Lambda^n(\calG): F^n(\calG)\ra G^n(\calG)$ is an isomorphism.
\end{thm}

This is a generalization of \cite[Theorem 4.4.6]{bonicke2018going} and \cite[Theorem 4.6]{bonicke2019going} from ample groupoid case.

Our strategy is approximating $\E \calG$ by the geometric realizations $P_K(\calG)$ of Rips complexes, so it suffices to deal with all finite dimensional proper $\calG$-compact $\calG$-simplicial complex with hypotheses $(H_1)$ and $(H_2)$. Firstly we will deal with 0-dimensional cases.

\begin{lem}\label{invariant subsets}
    Let $\calH$ be an \'etale groupoid.
    \begin{enumerate}
        \item If $A$ is an $\calH$-invariant subset of $\calH\units$, then $\overline{A}$ is an $\calH$-invariant closed subset of $\calH$.
        \item If $B$ is an $\calH$-invariant locally closed subset of $\calH\units$, then there exists an $\calH$-invariant open subset $U$ and an $\calH$-invariant closed subset $C$ such that $B=U\cap C$.
    \end{enumerate}
\end{lem}
\begin{proof}
    (1) It suffices to show that, if $h\in \calH_{\overline{A}}$, then $r(h)\in \overline{A}$. Let $W$ be an open bisection of $h$ in $\calH$. Since $s(h)\in \overline{A}$, there exists a net $(a_\lambda)_\lambda$ in $A$ such that $a_\lambda\ra s(h)$. Then we have $r(h)=\lim_\lambda r(s|_W\inv(a_\lambda))\in \overline{A}$.

    (2) Since $B$ is locally closed, there exists an open $V$ such that $B=V\cap \overline{B}$. Let $U=r(s\inv(V))$, $U$ is therefore an $\calH$-invariant open containing $V$ and $\overline{B}$ is an $\calH$-invariant closed subset by the previous result. Claim: $B=U\cap \overline{B}$.

    If $x\in U\cap \overline{B}$, there exists $h\in \calH_V$ such that $r(h)=x$. By the invariance of $\overline{B}$, $s(h)\in V\cap \overline{B}=B$. Then by invariance of $B$, $x=r(h)\in B$. So we proved that $U\cap \overline{B}\subseteq B$. Another inclusion is trivial and we proved our claim.
\end{proof}

\begin{lem}\label{goind down small 0-dim case}
    Under the same assumptions of theorem \ref{going-down principle}, for any proper open subgroupoid $\calH\subseteq \calG$ and any $\calH$-invariant locally closed subset $S$ of $\calH\units$,
    \[\Lambda_\calH^n(C_0(S)): F_\calH^n(C_0(S))\ra G_\calH^n(C_0(S))\]
    is an isomorphism.
\end{lem}
\begin{proof}
    Firstly, $S$ is a locally compact Hausdorff proper $\calH$-space, so $C_0(S)$ is a well-defined element of $\mathcal C(\calG)$. When $S$ is an open of $\calH\units$, let $\tilde\calH=\calH^S_S$, which is a proper open subgroupoid of $\calH$, therefore $\Lambda^n_{\tilde\calH}(C_0(S))$ is an isomorphism. Since $\calH_S=\tilde\calH$, $\ind^\calH_{\tilde\calH}C_0(S)\cong C_0(S)$ as an $\calH$-\cst-algebra, hence by the compatibility of $\Lambda$ with induction, the following diagram commutes.
    \[\xymatrix{
     F^n_{\calH}(C_0(S)) \ar[r]^{\Lambda^n_\calH(S)}  & G^n_{\calH}(C_0(S))\\
     F^n_{\tilde\calH}(C_0(S)) \ar[r]^{\Lambda^n_{\tilde\calH}(S)}\ar[u]^{I^\calH_{\tilde\calH}(n)} & G^n_{\tilde\calH}(C_0(S))\ar[u]_{I^\calH_{\tilde\calH}(n)}
    }\]
    By definition the two vertical maps $I^\calH_{\tilde\calH}(n)$ are also isomorphisms. So $\Lambda^n_\calH(C_0(S))$ is an isomorphism.

    Now if $S$ is closed subset of $\calH\units$, then $\calH\units\setminus S$ is an $\calH$-invariant open subset of $\calH\units$. Then $\Lambda^n_\calH(C_0(\calH\units))$ and $\Lambda^n_\calH(C_0(\calH\units\setminus S))$ are isomorphisms. Consider the following short exact sequence in $\ca C(\calH)$,
    \[0\ra C_0(\calH\units\setminus S)\ra C_0(\calH\units)\ra C_0(S)\ra 0.\]
    By lemma \ref{going-down functor long exactness}, it induces the following commutative diagram, where the two rows are exact.
    \begin{equation}\label{diagram 4}
        \xymatrix@C=5em@R=3em{
     F^{n-1}_{\calH}(C_0(\calH\units)) \ar[r]^{\Lambda^{n-1}_\calH(\calH\units)}\ar[d] & G^{n-1}_{\calH}(C_0(\calH\units))\ar[d]\\
     F^{n-1}_{\calH}(C_0(\calH\units\setminus S)) \ar[r]^{\Lambda^{n-1}_\calH(\calH\units\setminus S)} \ar[d]& G^{n-1}_{\calH}(C_0(\calH\units\setminus S))\ar[d]\\
     F^n_\calH(C_0(S)) \ar[r]^{\Lambda^{n}_\calH(S)} \ar[d] & G^n_{\calH}(C_0(S))\ar[d]\\
     F^n_\calH(C_0(\calH\units)) \ar[r]^{\Lambda^{n}_\calH(\calH\units)}\ar[d] & G^n_{\calH}(C_0(\calH\units))\ar[d]\\
     F^n_\calH(C_0(\calH\units\setminus S)) \ar[r]^{\Lambda^{n}_\calH(\calH\units\setminus S)} &G^n_{\calH}(C_0(\calH\units\setminus S))
    }
    \end{equation}
    In this diagram all horizontal arrows except the middle one are isomorphisms. By 5-lemma, this implies that $\Lambda^n_\calH(C_0(S))$ is an isomorphism.

    Now if $S$ is locally closed subset of $\calH\units$, by lemma \ref{invariant subsets}, we can assume that $S=V\cap \overline{S}$, where $V$ is an $\calH$-invariant open of $\calH\units$, $\overline{S}$ is an $\calH$-invariant closed subset of $\calH$. Now let $\tilde\calH=\calH^V_V$, which is an proper open subgroupoid of $\calH$. By the previous result, since $S$ is an $\tilde\calH$-invariant closed subset of $V$, $\Lambda^n_{\tilde\calH}(C_0(S))$ is an isomorphism. Again $C_0(S)\cong \ind^\calH_{\tilde\calH}C_0(S)$ as an $\calH$-\cst-algebra and using the compatibility of $\Lambda$ with induction, $\Lambda_{\calH}^n(C_0(S))$ is equivalent to $\Lambda^n_{\tilde\calH}(C_0(S))$ and hence an isomorphism.
\end{proof}

\begin{prop}\label{going-down 0-dim}
    Under the same assumptions of theorem \ref{going-down principle}, if $Y$ is a locally compact Hausdorff space with proper \'etale $\calG$-action, $Z\subseteq Y$ is a $\calG$-compact subset of $Y$, then the map
    \[\Lambda^n_\calG(C_0(Z)):  F^n_\calG(C_0(Z))\ra G^n_\calG(C_0(Z))\]
    is an isomorphism.
\end{prop}
\begin{proof}
    Let $\rho:Y\ra \calG\units$ be the anchor map, which is a local homeomorphism. By the local structure of proper actions of \'etale groupoids (proposition \ref{local structure of proper cocompact action of étale groupoid}), for every $z\in Z$, there is an open neighborhood $U_z$ of $z$ in $Y$, and a proper open subgroupoid $\calH(z)$ in form of $F_z\ltimes \calH(z)\units$, where $F_z$ is the finite group $\{\gamma\in \calG: \gamma z =z\}$, $\calH(z)\units$ is an open neighborhood of $\rho(z)$, such that $\rho|_{U_z}$ is homeomorphism onto $\calH(z)\units$, $U_z$ is $\calH(z)$-invariant, and when $\gamma \in \calG\backslash \calH(z)$, $\gamma U_z\cap U_z=\emptyset$. Let $V_z$ be an $\calH(z)$-invariant relatively compact open subset of $U_z$ such that $V_z\subseteq \overline{V_z}\subseteq U_z$. By lemma \ref{invariant subsets}, $\overline{V_z}$ is also an $\calH(z)$-invariant closed subset of $U_z$.
    
    Since $Z$ is $\calG$-compact, there are finitely many points $z_1,\cdots, z_m$ such that $\calG V_{z_1},\cdots, \calG V_{z_m}$ cover $Z$. Let $U_i=U_{z_i}$, $V_i=V_{z_i}$ and $\calH_i=\calH(z_i)$. Let $S_i=\overline{V_i}\cap Z$. Then $S_i$ is an $\calH_i$-invariant compact subset of $Y$, and
    \[Z=\cup_{i=1}^m\calG S_i\cong \cup_{i=1}^m\calG \times_{\calH_i} S_i.\]
    
    We will prove the statement by induction on $m$. 

    When $m=1$, say $Z=\calG S_1$. By lemma \ref{goind down small 0-dim case}, $\Lambda_{\calH_1}^n(C(S_1))$ is an isomorphism.
    Then using the compatibility of $\Lambda$ with induction, the following diagram commutes.
    \[\xymatrix{
     F^n_{\calH_1}(C(S_1)) \ar[r]^{\Lambda^n_{\calH_1}(S_1)} \ar[d]_{I^\calG_{\calH_1}(n)} & G^n_{\calH_1}(C(S_1))\ar[d]^{I^\calG_{\calH_1}(n)} \\
     F^n_\calG(C_0(\calG S_1)) \ar[r]^{\Lambda^n_\calG(\calG S_1)} & G^n_\calG(C_0(\calG S_1))
    }\]
    Therefore, $\Lambda_\calG^n(\calG S_1)$ is also an isomorphism. We proved the case $m=1$.

    Now assume the case $m-1$ is proved. Let $Z'=\cup_{i=2}^m\calG S_i$,  $S'=S_1\cap Z'$. Since for each $1\leqslant i\leqslant m$, $\calG(V_i\cap Z)$ is an open subset of $Z$ that is contained in $\calG S_i$, we have $Z=\inter_Z(\calG S_1)\cup \inter_Z(Z')$ and $\calG S_1\cap Z'=\calG S'$. Using corollary \ref{going-down MV-seq}, we have a commutative diagram where the two rows are exact sequences.

    {\small
\begin{equation*}
\xymatrix@C=5em@R=3em{
 F^n_\calG(C_0(\calG S')) \ar[r]^{\Lambda^n_\calG(\calG S')} \ar[d] & G^n_\calG(C_0(\calG S')) \ar[d] \\
 F^n_\calG(C_0(Z')) \oplus  F^n_\calG(C_0(\calG S_1)) \ar[r] \ar[d] & G^n_\calG(C_0(Z')) \oplus G^n_\calG(C_0(\calG S_1)) \ar[d] \\
 F^n_\calG(C_0(Z)) \ar[r]^{\Lambda^n_\calG(Z)} \ar[d] & G^n_\calG(C_0(Z)) \ar[d] \\
 F^{n+1}_\calG(C_0(\calG S')) \ar[r]^{\Lambda^{n+1}_\calG(\calG S')} \ar[d] & G^{n+1}_\calG(C_0(\calG S')) \ar[d] \\
 F^{n+1}_\calG(C_0(Z')) \oplus  F^{n+1}_\calG(C_0(\calG S_1)) \ar[r] & G^{n+1}_\calG(C_0(Z')) \oplus G^{n+1}_\calG(C_0(\calG S_1))
}
\end{equation*}
}

    By assumption, the second and the fifth horizontal arrows are isomorphisms. And $S'$ is an $\calH_1$-invariant compact subsets of $S_1$, $\calG S'$ is again same as the case $m=1$, therefore the first and the fourth horizontal arrows are also isomorphisms. The five lemma implies that the middle horizontal arrow is an isomorphism. That is, $\Lambda^n_\calG(Z)$ is an isomorphism for any $n\in \mathbb Z$.
\end{proof}

\begin{prop}\label{going-down sim complex}
    Under the same assumptions of theorem \ref{going-down principle}, if $(Y,\Delta)$ is a finite dimensional proper $\calG$-compact $\calG$-simplicial complex that has hypotheses $(H_1)$ and $(H_2)$, then
    \[\Lambda^n_\calG(C_0(|\Delta|)):  F^n_\calG(C_0(|\Delta|))\ra G^n_\calG(C_0(|\Delta|))\]
    is an isomorphism.
\end{prop}
\begin{proof}
    By proposition \ref{bary center subdivision invariance} and proposition \ref{promence property of barycenter subdivision}, after replacing $(Y,\Delta)$ by its barycenter subdivision, we can assume that $(Y,\Delta)$ is typed. Now we apply induction on dimension of $\Delta$.

    When $\Delta$ has dimension 0, by proposition \ref{AmBm}, $|\Delta|$ is $\calG$-equivariantly homeomorphic to a $\calG$-compact subset of $Y$. In this case the statement is already proved in proposition \ref{going-down 0-dim}.

    Now assume that $\Delta$ has dimension $m$ and the dimension $m-1$ case is proved. Since $\Delta$ is typed, by proposition \ref{typed sim com}, $U:=|\Delta|\setminus|\Delta^{m-1}|$ is $\calG$-equivariantly homeomorphic to $center(m,\Delta)\times \mathbb R^m$. While by corollary \ref{center is closed in some etale space}, $center(m,\Delta)$ is $\calG$-compact subset of some \'etale $\calG$-space. By  proposition \ref{going-down 0-dim}, $\Lambda^*_\calG(center(m,\Delta))$ is an isomorphism. And $\Lambda$ is compatible with suspension, hence $\Lambda^*_\calG(U)$ is also an isomorphism.

    Apply lemma \ref{going-down functor long exactness} to the short exact sequence in $\ca C(\calG)$
    \[0\ra C_0(U)\ra C_0(|\Delta|)\ra C_0(|\Delta^{m-1}|)\ra 0,\]
     we have a commutative diagram where the two rows are exact sequences.

    {\small
\begin{equation}
\xymatrix@C=6em@R=3em{
 F^{n-1}_\calG(C_0(U)) \ar[r]^{\Lambda^{n-1}_\calG(U)} \ar[d] & G^{n-1}_\calG(C_0(U)) \ar[d] \\
 F^{n}_\calG(C_0(|\Delta^{m-1}|)) \ar[r]^{\Lambda^n_\calG(|\Delta^{m-1}|)} \ar[d] & G^n_\calG(C_0(|\Delta^{m-1}|)) \ar[d] \\
 F^n_\calG(C_0(|\Delta|)) \ar[r]^{\Lambda^n_\calG(|\Delta|)} \ar[d] & G^n_\calG(C_0(|\Delta|)) \ar[d] \\
 F^{n}_\calG(C_0(U)) \ar[r]^{\Lambda^{n}_\calG(U)} \ar[d] & G^{n}_\calG(C_0(U)) \ar[d] \\
 F^{n+1}_\calG(C_0(|\Delta^{m-1}|)) \ar[r]^{\Lambda^{n+1}_\calG(|\Delta^{m-1}|)} & G^{n+1}_\calG(C_0(|\Delta^{m-1}|))
}
\end{equation}
}

We have seen that the first and the fourth horizontal arrows are isomorphisms. The second and the fifth arrows are also isomorphisms by assumption. The five lemma implies that $\Lambda^n_\calG(|\Delta|)$ is an isomorphism.
\end{proof}

\begin{proof}[Proof of theorem \ref{going-down principle}]
    Theorem \ref{going-down principle} is implied by proposition \ref{enough to run over rips complex} and proposition \ref{going-down sim complex}.
\end{proof}

\subsection{Baum--Connes conjecture for \'etale groupoids that are strongly amenable at infinity}
In this section we generalize the main result of \cite{bonicke2020going} for all second countable \'etale groupoids. 
It was also firstly proved in \cite{bonicke2024categorical} via a categorical approach.

\begin{lem}\label{Suspension in equivariant kk}
    Let $\calG$ be a second countable \'etale groupoid and $A,B$ be separable $\calG$-\cst-algebras. Let $\kk_n^\calG(A,B):=\kk_0^\calG(A\otimes C_0(\mathbb R)^{\otimes {|n|}},B)$ for any $n\neq 0$. Then for any $n\in \mathbb Z$, there is a natural isomorphism 
    \[S^n_{\calG}:\kk^\calG_{n+1}(-,A)\ra \kk^\calG_{n}(-\otimes C_0(\mathbb R),A).\]
\end{lem}
\begin{proof}
    For $n\geqslant 0$, the functors on the two sides are equal, hence we can take $S^n_\calG$ as the identity natural transformation.

    For $n\leqslant -1$, we need to bring in the Bott periodicity. Let $\beta\in \kk_0(\mathbb C,C_0(\mathbb R)^{\otimes 2})$ and $\alpha\in \kk_0(C_0(\mathbb R)^{\otimes 2},\mathbb C)$ be the Bott elements (see \cite[6.7]{Skandalis91}), hence $\alpha\otimes \beta=1_{\mathbb C}$ and $\beta\otimes \alpha=1_{C_0(\mathbb R)^{\otimes 2}}$. Let $*$ be the groupoid of one point, so the trivial map $\calG\ra *$ is a strict morphism. Then let $\beta_\calG\in \kk_0^\calG(C_0(\calG\units),C_0(\calG\units)\otimes C_0(\mathbb R)^{\otimes 2})$ and $\alpha_\calG\in \kk^\calG_0(C_0(\calG\units)\otimes C_0(\mathbb R)^{\otimes 2},C_0(\calG\units))$ be respectively pullback of $\beta$ and $\alpha$ by this strict morphism. For any $D\in \ca C(\calG)$, let $\beta_{\calG,D}=\tau_D^\calG(\beta_\calG)\in \kk_0^\calG(D,D\otimes C_0(\mathbb R)^{\otimes 2})$, $\alpha_{\calG,D}=\tau_D^\calG(\alpha_\calG)\in \kk^\calG_0(D\otimes C_0(\mathbb R)^{\otimes 2},D)$. So $\alpha_{\calG,D}$ and $\beta_{\calG,D}$ are inverse to each other. For convenience of writing let $id_k=id_{C_0(\mathbb R)^{\otimes k}}$ for any $k\in \mathbb N_+$. Then for any morphism $f:D_1\ra D_2$ in $\ca C(\calG)$, by proposition \ref{functoriality of tau},
    \[(f\otimes id_2)_*(\beta_{\calG,D_1})=f^*(\beta_{\calG,D_2}).\]

    We define the natural isomorphism $S^n_\calG=(S^n_{\calG,D})_{D\in \ca C(\calG)}$ as,
    \[S^n_{\calG,D}:\kk_0^\calG(D\otimes C_0(\mathbb R)^{\otimes(-1-n)},A)\ra \kk_0^\calG(D\otimes C_0(\mathbb R)^{\otimes(1-n)},A),\]
    \[y\mapsto \beta_{\calG,D\otimes C_0(\mathbb R)^{\otimes(-1-n)}}\otimes y.\]

    For any morphism $f:D_1\ra D_2$ in $\ca C(\calG)$, $y\in \kk_0^\calG(D_2\otimes C_0(\mathbb R)^{\otimes(-1-n)},A)$, we have
    \begin{align*}
        & (f\otimes id_{1-n})^*(\beta_{\calG,D_2\otimes C_0(\mathbb R)^{\otimes(-1-n)}}\otimes y) \\
        & = (f\otimes id_{1-n})^*(\beta_{\calG,D_2\otimes C_0(\mathbb R)^{\otimes(-1-n)}})\otimes y\\
        & = \beta_{\calG,D_1\otimes C_0(\mathbb R)^{\otimes(1-n)}}\otimes [f\otimes id_{-1-n}]\otimes y\\
        & = \beta_{\calG,D_1\otimes C_0(\mathbb R)^{\otimes(1-n)}}\otimes (f\otimes id_{-1-n})^*(y).
    \end{align*}
    That is, the following diagram commutes.
    \[\xymatrix{
    \kk_0^\calG(D_2\otimes C_0(\mathbb R)^{\otimes(-1-n)},A)\ar[r]^{S^n_{\calG,D_2}} \ar[d]_{(f\otimes id_{-1-n})^*} &  \kk_0^\calG(D_2\otimes C_0(\mathbb R)^{\otimes(1-n)},A) \ar[d]^{(f\otimes id_{1-n})^*} \\
    \kk_0^\calG(D_1\otimes C_0(\mathbb R)^{\otimes(-1-n)},A)\ar[r]_{S^n_{\calG,D_1}} &  \kk_0^\calG(D_1\otimes C_0(\mathbb R)^{\otimes(1-n)},A)
    }
    \]
    So we proved that $S^n_\calG$ is a well-defined natural isomorphism.
\end{proof}

\begin{prop}\label{going-down functor example 1}
    Let $\calG$ be a second countable \'etale groupoid and $A$ be a separable $\calG$-\cst-algebra. Then the collection $(\kk^\calH_n(-,A|_\calH))_{\calH\in \ca S(\calG)}$ is a going-down functor, where $\kk^\calH_n(-,A|_\calH)$ is defined as $\kk_0^\calH(-\otimes C_0(\mathbb R)^{\otimes |n|}, A|_\calH)$.

    If $A,B$ are separable $\calG$-\cst-algebras and $x\in \kk^\calG(A,B)$, define
    \[\Lambda^n_\calH:=-\otimes res^\calG_\calH(x): \kk^\calH_n(-,A|_\calH)\ra \kk^\calH_n(-,B|_\calH),\]
    then $\Lambda = (\Lambda^n_\calH)_{\calH\in \ca S(\calG)}$ is a going-down transformation.
\end{prop}
\begin{proof}
    For the collection of functors $(\kk^\calH_n(-,A|_\calH))_{\calH\in \ca S(\calG)}$, the homotopy invariance is obtained from definition of groupoid equivariant KK-theory. The half-exactness is proved in \cite[Proposition 7.2, Lemma 7.7]{tu1999conjecture-n}. By lemma \ref{Suspension in equivariant kk}, suspension axiom is also valid.
    
    Using corollary \ref{induction-restriction adj}, for any two elements $\calH_1\subseteq \calH_2$ of $\ca S(\calG)$, there is natural isomorphism 
    \[\kk^{\calH_1}_*(-,A|_{\calH_1})\ra \kk^{\calH_2}_*(\ind^{\calH_2}_{\calH_1}(-),A|_{\calH_2}),\] 
    constitute of inverses of compression isomorphisms. So we proved the induction axiom. In conclusion  $(\kk^\calH_*(-,A|_\calH))_{\calH\in \ca S(\calG)}$ is a going-down functor.

    For any morphism $f:D_1\ra D_2$ in $\ca C(\calG)$, $y\in \kk^\calH_n(D_2,A|_\calH)$,
    \[f^*(y)\otimes res^\calG_\calH(x)=f^*(y\otimes res^\calG_\calH(x)),\]
    that is the diagram
    \[\xymatrix{
    \kk_n^\calH(D_2,A|_\calH) \ar[r]^{\Lambda^n_\calH(D_2)} \ar[d]_{f^*} & \kk_n^\calH(D_2,B|_\calH) \ar[d]^{f^*}\\
    \kk_n^\calH(D_1,A|_\calH) \ar[r]^{\Lambda^n_\calH(D_1)} & \kk_n^\calH(D_1,B|_\calH)
    }\]
    is commutative, therefore $\Lambda_\calH^n$ is a natural transformation.
    
    Clearly $\Lambda^n_\calH$ are compatible with suspension for $n\geqslant 0$. For $n\leqslant -1$, let the natural transformations $S^n_\calH$ be described as in lemma \ref{Suspension in equivariant kk}. Then the compatibility with suspension comes from the associativity of Kasparov product.
    
    Now any two elements $\calH_1\subseteq \calH_2$ of $\ca S(\calG)$, $D\in \ca C(\calH_1)$ and $y\in \kk^{\calH_2}(\ind^{\calH_2}_{\calH_1}D,A|_{\calH_2})$, 
    \begin{align*}
        comp^{\calH_2}_{\calH_1}(y)\otimes res^{\calG}_{\calH_1}(x) & = [i_D]\otimes res^{\calH_2}_{\calH_1}(y)\otimes res^{\calG}_{\calH_1}(x)\\
        & = [i_D]\otimes res^{\calH_2}_{\calH_1}(y)\otimes res^{\calH_2}_{\calH_1}(res^\calG_{\calH_2}(x))\\
        & = [i_D]\otimes res^{\calH_2}_{\calH_1}(y\otimes res^{\calG}_{\calH_2}(x))\\
        & =comp^{\calH_2}_{\calH_1}(y\otimes res^{\calG}_{\calH_2}(x)).
    \end{align*}
    Hence, $\Lambda$ is compatible with induction. (This is also \cite[Lemma 6.7]{bonicke2020going}.)
\end{proof}

\begin{thm}\label{going down-standard example result}
    \textnormal{\cite[Theorem 4.4]{bonicke2024categorical}}
    Let $\calG$ be a second countable \'etale groupoid, $A, B$ be separable $\calG$-\cst-algebras. Suppose that there is an element $x\in \kk_0^\calG(A,B)$ such that
    \[\kk^\calH_*(C_0(\calH\units), A|_\calH) \xrightarrow{-\otimes res_\calH^\calG (x)} \kk^\calH_*(C_0(\calH\units),B|_\calH) \]
    is an isomorphism for all proper open subgroupoid $\calH\subseteq \calG$. Then
    \[-\otimes x:K^{top}_*(\calG;A)\ra K^{top}_*(\calG; B)\]
    is an isomorphism.
\end{thm}
\begin{proof}
    Apply theorem \ref{going-down principle} to the going-down transformation defined in proposition \ref{going-down functor example 1}.
\end{proof}

About the amenability and amenability at infinity of \'etale groupoids the reference is \cite{anantharaman2000amenable} and \cite{anantharaman2016exact}. Recall that the amenability can be characterized as the following.

\begin{prop}\label{amenability}
    \textnormal{\cite[Proposition 3.2, Remark 3.4]{anantharaman2016exact}}
    Let $\calG$ be an \'etale groupoid. Then $\calG$ is amenable if and only if there exists a net $(\phi_i)_i$ of non-negative functions in $C_c(\calG)$ such that
    \begin{enumerate}
        \item for every $i$, $\phi_i|_{\calG\units}$ is uniformly bounded by 1;
        \item $\phi_i$ converges to 1 uniformly on every compact of $\calG$.
    \end{enumerate}
\end{prop}

Recall that a locally compact Hausdorff groupoid is called amenable at infinity, if there exists a locally compact Hausdorff $\calG$-space $Y$ with proper surjective anchor map $\rho$, such that $\calG\ltimes Y$ is amenable. If moreover $\rho$ admits a continuous section, we say that $\calG$ is strongly amenable at infinity. See \cite[Definition 4.1]{anantharaman2016exact}.

\begin{lem}\label{technical lemma for changing amenable space}
    \textnormal{\cite[Proposition 3.1.28]{lassagne2013k}}
    Let $\calG$ be a second countable \'etale groupoid, $Y$ be a second countable locally compact Hausdorff $\calG$-space with proper surjective anchor map $\rho$, and $\rho$ admits a continuous section. Let $(P(Y), \tilde\rho)$ be defined as in the section \ref{section-space of measures}. Then $\tilde \rho$ is a continuous surjective proper map that admits a continuous section, $P(Y)$ is a second countable locally compact Hausdorff space, and $\calG$ acts amenably on $P(Y)$.
\end{lem}
\begin{proof}
    By proposition \ref{first properties of P(Y)} and proposition \ref{P(Y) has continuous action}, $\tilde \rho$ is a continuous surjective proper map that admits a continuous section, $P(Y)$ is a second countable locally compact Hausdorff $\calG$-space. It suffices to check the amenability of $\calG\curvearrowright P(Y)$.
    
    Since $\calG\ltimes Y$ is amenable, by lemma \ref{amenability}, there exists a net $(\phi_i)_i$ of continuous positive definite functions in $C_c(\calG\ltimes Y)$ such that,
    \begin{enumerate}
        \item for every $i$, $\phi_i|_{(\calG\ltimes Y)\units}$ is uniformly bounded by 1;
        \item $\phi_i$ converges to 1 uniformly on every compact of $\calG\ltimes Y$.
    \end{enumerate}
    Now we define $\psi_i$ a function on $\calG\ltimes P(Y)$ that $\psi_i(\gamma,\mu)=\mu(\phi_i(\gamma,-))$. If $((\gamma_\lambda,\mu_\lambda))_\lambda$ is a net that converges to $(\gamma,\mu)$ in $\calG\ltimes Y$, $(\phi_i(\gamma_\lambda,-))_\lambda$ is then a net that converges to $\phi_i(\gamma,-)$ in $\ca A$. There exists $g\in C_c(Y)$ such that $g|_{Y_{s(\gamma)}}=\phi_i(\gamma,-)$. The two net $(\phi_i(\gamma_\lambda,-))_\lambda$ and $g|_{Y_{s(\gamma)}}$ converges to a same limit in $\ca A$, therefore $\phi_i(\gamma_\lambda,-)-g|_{Y_{s(\gamma)}}\xrightarrow{\|\cdot\|_\infty}0$. Hence, $\psi_i(\gamma_\lambda,\mu_\lambda)-\psi_i(\gamma,\mu)=\mu_\lambda(\phi_i(\gamma_\lambda,-))-\mu(\phi_i(\gamma,-))=\mu_\lambda(\phi_i(\gamma_\lambda,-)-g|_{Y_{s(\gamma)}})+(\mu_\lambda-\mu)(g)$ converges to zero. So every $\psi_i$ is a continuous function.
    
    If $K_1\times_{\calG\units}K_2\subseteq \calG\ltimes Y$ is a compact subset that contains the support of $\phi_i$, $K_1\times_{\calG\units}\tilde\rho\inv(\rho(K_2))$ is then a compact subset containing the support of $\psi_i$. For every $\mu\in P(Y)$, let $y=\tilde\rho(y)$ and $\gamma_1,\cdots, \gamma_n\in \calG^{\rho(y)}$, the $n\times n$ matrix $[\psi_i(\gamma_j\inv\gamma_k, \gamma_j\inv \mu)]_{j,k}=[\mu(\phi_i(\gamma_j\inv\gamma_k,\gamma_j\inv-))]_{j,k}$ is positive-definite since the matrix $[\phi_i(\gamma_j\inv\gamma_k,\gamma_j\inv-)]_{j,k}$ is. For any $\mu\in P(Y)$, $\psi_i(\tilde\rho(\mu),\mu)=\mu(\phi(\tilde\rho(\mu),-))$ is also uniformly bounded by 1. Finally, $\phi_i\ra 1$ uniformly on every compact in form of $K_1\times_{\calG\units}K_2$, and this implies that $\psi_i\ra 1$ uniformly on every compact in form of $K_1\times_{\calG\units}\tilde\rho\inv(\rho(K_2))$. In conclusion, we proved the amenability of $\calG\ltimes P(Y)$.
\end{proof}

\begin{prop}\label{space whitness strongly amenablity induces homotopy equivalence}
    \textnormal{{\cite[Proposition 8.2]{bonicke2020going}}}
    Let $\calG$ be a second countable \'etale groupoid, and let $Y$ be a second countable locally compact Hausdorff $\calG$-space with proper surjective anchor map $\rho$, such that $\calG\ltimes Y$ is amenable, $\rho$ admits a continuous section $s:\calG\units\ra Y$,  for every $u \in \calG\units$ the fiber $Y_u$ is a convex space and $\calG$ acts by affine transformations. Then for any proper open subgroupoid $\calH \subseteq \calG$, $\rho|_{Y_{\calH\units}}:Y_{\calH\units}=\rho\inv(\calH\units)\ra \calH\units$ admits an $\calH$-equivariant continuous section $\tilde s:\calH\units\ra Y_{\calH\units}$, such that $\tilde s\circ \rho$ is homotopic to $id_{Y_{\calH\units}}$.
\end{prop}

\begin{thm}
    \textnormal{\cite[Theorem 4.8]{bonicke2024categorical}}
    Let $\calG$ be a second countable \'etale groupoid which is strongly amenable at infinity. Then for any separable $\calG$-\cst-algebra $A$, the Baum--Connes assembly map
    \[\mu_{\calG,A}:K^{top}_*(\calG;A)\ra K_*(A\rtimes_r\calG)\]
    is split injective.
\end{thm}
\begin{proof}
    Firstly we need to show that there exists a $\calG$-space $(Y,\rho)$ which satisfies the conditions in proposition \ref{space whitness strongly amenablity induces homotopy equivalence}. By \cite[Proposition 4.8, Lemma 4.9]{anantharaman2016exact}, there exists a second countable locally compact Hausdorff $\calG$-space $Y_1$ with proper surjective anchor map $\rho_1$, such that $\calG\ltimes Y_1$ is amenable and $\rho_1$ admits a continuous section. Then apply lemma \ref{technical lemma for changing amenable space} to $(Y_1,\rho_1)$, let $Y=P(Y_1)$ and $\rho=\tilde\rho_1$, by lemma \ref{technical lemma for changing amenable space} $(Y,\rho)$ satisfies all conditions in proposition \ref{space whitness strongly amenablity induces homotopy equivalence}.

    Let $[\rho^*]\in \kk_0^\calG(C_0(\calG\units),C_0(Y))$ be the element induced by the proper continuous map $\rho: Y\ra \calG\units$. Let $x=\tau_A^\calG([\rho^*])\in \kk_0^\calG(A, A\otimes_{\calG\units}C_0(Y))$. By proposition \ref{space whitness strongly amenablity induces homotopy equivalence}, for every proper open subgroupoid $\calH\subseteq \calG$, $res^\calG_\calH(x)\in \kk_0^\calH(A|_\calH,A|_\calH\otimes_{\calH\units}C_0(Y_{\calH\units}))$ is invertible. Hence, theorem \ref{going down-standard example result} implies that
    \[-\otimes x: K^{top}_*(\calG;A)\ra K^{top}_*(\calG,A\otimes_{\calG\units}C_0(Y))\]
    is an isomorphism.

    By naturality of Baum--Connes assembly map, the following diagram commutes.
    \[
    \xymatrix{
    K^{top}_*(\calG;A) \ar[r]^{\mu_{\calG,A}} \ar[d]^{-\otimes x} & K_*(A\rtimes_r\calG) \ar[d]^{-\otimes j_r(x)} \\
    K^{top}_*(\calG; A\otimes_{\calG\units}C_0(Y)) \ar[r]_{\mu} & K_*((A\otimes_{\calG\units}C_0(Y))\rtimes_r\calG)
    }
    \]
    Here $j_r$ is Kasparov descent map, the bottom arrow $\mu$ is the Baum--Connes assembly map for $\calG$ with coefficients in $A\otimes_{\calG\units}C_0(Y)$. \cite[Lemma 4.1]{skandalis2002coarse} implies that $\mu$ is equivalent to $\mu_{\calG\ltimes Y,A}$, while $\mu_{\calG\ltimes Y,A}$ is an isomorphism since $\calG\ltimes Y$ is amenable by \cite[Theorem 0.1, Lemma 3.5]{tu1999conjecture-bc}. Now the left and the bottom arrows are isomorphisms, hence $\mu_{\calG,A}$ is split injective.
\end{proof}

\begin{rem}\label{remark of Skandalis}
    Georges Skandalis pointed it out that we can prove more directly the split injectivity following the proof of the countable discrete group case by Higson (see \cite[Theorem 3.2, Proposition 3.7]{Higson99}) without using the going-down principle. Use the same notation as above, let $Z$ be any proper $\calG$-compact locally compact Hausdorff $\calG$-space. If there exists a proper open subgroupoid $\calH$ and an $\calH$-invariant relatively compact open $V$ of $Z$ such that $Z\cong \calG_{\calH\units}\times_\calH \overline{V}$, then since the following diagram commutes,
    \[\xymatrix{
    \kk_*^\calG(C_0(Z),A) \ar[r]^{\hspace{-1cm} -\otimes x} \ar[d]_{comp^\calG_\calH} & \kk_*^\calG(C_0(Z),A\otimes_{\calG\units}C_0(Y)) \ar[d]^{comp^\calG_\calH}\\
    \kk_*^\calH(C(\overline{V}),A|_\calH) \ar[r]_{\hspace{-1cm} -\otimes res^\calG_\calH(x)} & \kk_*^\calH(C(\overline{V}),A|_\calH\otimes_{\calH\units}C_0(Y_{\calH\units}))
    }\]
    in which the two vertical arrows are compression isomorphisms (theorem \ref{compression isomorphism}), the bottom arrow is an isomorphism because of proposition \ref{space whitness strongly amenablity induces homotopy equivalence}, we can conclude that the top arrow is an isomorphism. By lemma \ref{local structure of proper cocompact action of étale groupoid}, a general proper $\calG$-compact space $Z$ is a finite union of closed subspaces of the type $\calG_{\calH\units}\times_\calH \overline{V}$. Then follow from the Mayer-Vietoris exact sequence and the five lemma,
    \[-\otimes x:\kk_*^\calG(C_0(Z),A) \ra \kk_*^\calG(C_0(Z),A\otimes_{\calG\units}C_0(Y))\]
    is an isomorphism for any proper $\calG$-compact $\calG$-space $Z$. Let $Z$ runs over all $\calG$-compact subspaces of $\E\calG$,
    \[-\otimes x: K_*^{top}(\calG;A)\ra K_*^{top}(\calG;A\otimes_{\calG\units}C_0(Y))\]
    is an isomorphism. The rest of the proof will be same as above.
\end{rem}

\subsection{Continuity of topological K-theory}

\begin{thm}
    \cite[Proposition 4.12]{bonicke2024categorical}
    Let $\calG$ be a second countable \'etale groupoid, $(A_n, \varphi_n)_n$ be an inductive sequence of separable $\calG$-\cst-algebras and $A=\varinjlim A_n$. Let $\psi_n: A_n \ra A$ be the canonical map, then
    \[\varinjlim\psi_{n,*}: \varinjlim K^{top}_*(\calG;A_n)\ra K^{top}_*(\calG;A)\]
    is an isomorphism.
\end{thm}
\begin{proof}
    The proof is very similar to \cite[Theorem 5.2]{bonicke2019going}. We have two going-down functors 
    \[(\varinjlim \kk_*^\calH(-,A_n|_\calH)_{\calH\in \ca S(\calG)},\quad (\kk_*^\calH(-,A|_\calH))_{\calH\in \ca S(\calG)},\]
    and
    \[\varinjlim_n K_*^{top}(\calG;A_n)\cong \varinjlim_{Y\subseteq \E \calG, \calG \text{-compact}}\varinjlim_n\kk^\calG_*(C_0(Y),A_n)\]
    (see \cite[section IX.8]{MacLane} for the commutativity of two colimits).

    For all $\calH\in \ca S(\calG)$ and all locally compact Hausdorff proper $\calH$-space $Y$, the homomorphisms
    \[\varinjlim(\psi_{n}|_\calH)_*:\varinjlim_n\kk^\calH_*(C_0(Y),A_n|_\calH)\ra \kk_*^\calH(C_0(Y),A|_\calH)\]
    constitute a going-down transformation. Hence, by theorem \ref{going-down principle} it suffices to show that when $\calH$ is a proper open subgroupoid and $Y=\calH\units$ this is an isomorphism. In this case, for every $n$, the following diagram commute,
    \[\xymatrix{
    \kk_*^\calH\left(C_0(\calH\units),\, A_n|_\calH\right)
    \ar[r]^(.50){(\psi_n|_\calH)_*}
    \ar[d]_{\mu_{\calH,A_n|_\calH}} &
    \kk_*^\calH\left(C_0(\calH\units),\, A|_\calH\right)
    \ar[d]^{\mu_{\calH,A|_\calH}} \\
    K_*\left(A_n|_{\calH}\rtimes \calH\right)
    \ar[r]^(.55){(\psi_n|_\calH \rtimes \calH)_*} &
    K_*\left(A|_\calH\rtimes \calH\right)
}\]

    The two horizontal arrows $\mu_{\calH,A_n|_\calH}$ and $\mu_{\calH,A|_\calH}$ are isomorphisms since $\calH$ is proper, hence $\varinjlim(\psi_{n}|_\calH)_*$ is equivalent to $\varinjlim ((\psi_n|_\calH)\rtimes\calH)_*$, which is an isomorphism because of continuity of K-theory and $A|_\calH\rtimes \calH=\varinjlim A_n|_\calH\rtimes \calH$ by \cite[Lemma 5.1]{bonicke2019going}.
\end{proof}

\begin{corr}
     Let $\calG$ be a second countable \'etale groupoid, $(A_n, \varphi_n)_n$ be an induction sequence of separable $\calG$-\cst-algebras and $A=\varinjlim A_n$. Assume that either $\calG$ is exact or all $\varphi_n$ are injective. Then if $\calG$ satisfies the Baum--Connes conjecture with coefficients in $A_n$ for all $n$, then $\calG$ satisfies the Baum--Connes conjecture with coefficient in $A$.
\end{corr}

\subsection{K\"unneth formula}
In this section we study the scope of validity of K\"unneth formula of $K^{top}_*(\calG;-)$, where $\calG$ is an second countable \'etale groupoid. Locally compact group cases are studied in \cite{chabert2004going}. Part of \'etale groupoid cases are studied in \cite{bonicke2019going}, but some results are restricted to ample groupoids. We are going to remedy this by using theorem \ref{going-down principle}.

Let $\calG$ be a second countable locally compact Hausdorff groupoid and $A$ be a separable exact $\calG$-\cst-algebra. For any \cst-algebra $B$, we define $\epsilon:K_*(B)\ra \kk_*^\calG(A, A\otimes B)$ as the composite
\[\kk_*(\mathbb C, B)\ra \kk_*^\calG(C_0(\calG\units),C_0(\calG\units)\otimes B)\xrightarrow{\tau^\calG_A}\kk_*^\calG(A,A\otimes B),\]
here $A \otimes B$ is the minimal tensor products, which is a well-defined $\calG$-\cst-algebra (see \cite[Theorem 4.1]{Lalonde15}), the map $\tau_A^\calG$ is same as in \cite[Proposition 6.6]{bonicke2019going}.  Let $\kk_*, \kk^\calG_*, K_*$ here be $\mathbb Z/2\mathbb Z$-graded, and the first map is induced from the canonical strict morphism from $\calG$ to a trivial groupoid of one point. Now for any proper $\calG$-space $Y$, define $\alpha_{Y}$ as the composite
\[\xymatrix{
\kk^\calG_*(C_0(Y),A)\otimes_{\Z} K_*(B) \ar[r]^{\hspace{-0.5cm}id\otimes_{\Z} \epsilon} \ar[dr]_{\alpha_{Y}} & \kk^\calG_*(C_0(Y),A)\otimes_{\Z} \kk^\calG_*(A,A\otimes B)\ar[d]^{\otimes} \\
& \kk^\calG_*(C_0(Y),A\otimes B).
}\]
And define $\alpha_\calG:K^{top}_*(\calG;A)\otimes_{\Z} K_*(B)\ra K^{top}_*(\calG,A\otimes B)$ by passing the limit over all $\calG$-compact subspaces $Y$ of $\E \calG$.

\begin{prop}
    \textnormal{\cite[Proposition 6.8]{bonicke2019going}}
    For a second countable locally compact Hausdorff groupoid $\calG$ and a separable exact $\calG$-\cst-algebra $A$, the following statements are equivalent.
    \begin{enumerate}
        \item The map $\alpha_\calG$ is an isomorphism for all $B$ with $K_*(B)$ free over $\mathbb Z$.
        \item For every \cst-algebra $B$, there exists a canonical homomorphism
\[\beta_\calG:K^{top}_*(\calG;A\otimes B)\ra \mathrm{Tor}(K^{top}_*(\calG;A),K_*(B))\]
such that 
\[0\ra K^{top}_*(\calG;A)\otimes_{\Z} K_*(B)\xrightarrow{\alpha_\calG} K^{top}_*(\calG,A\otimes B)\xrightarrow{\beta_\calG}\mathrm{Tor}(K^{top}_*(\calG;A),K_*(B))\ra 0\]
is an exact sequence (of $\mathbb Z/2\mathbb Z$-graded abelian groups)
    \end{enumerate}
\end{prop}

\begin{defn}
    We say that $A$ satisfies the $\calG$-K\"unneth formula if it satisfies one of the conditions above. We denote the class of all separable exact $\calG$-\cst-algebras that satisfies the $\calG$-K\"unneth formula by $\ca N_\calG$.
\end{defn}

When $\calG$ is a trivial groupoid of one point it is the class $\ca N$ as defined in \cite{chabert2004going}. Some stability properties of $\ca N_\calG$ are proved in \cite[Lemma 6.9]{bonicke2019going}.

For convenience of writing, when $ F^*$ is a $\Z/2\Z$-graded abelian group valued functor, for all $n\in \Z$ let $ F^n$ be $F^{n\bmod 2}$. Then $(F^n)_{n\in\mathbb Z}$ is naturally a $\Z$-graded $\mathrm{Ab}$-valued functor. We use the similar convention for natural transformations.

The K\"unneth formula and Baum--Connes conjecture for groupoids are connected by the following result.

\begin{prop}\label{connect Kunneth with BC}
    \textnormal{\cite[Proposition 6.10, Proposition 6.13]{bonicke2019going}}
    Let $\calG$ be a second locally compact Hausdorff groupoid with a Haar system, $A$ be a separable exact $\calG$-\cst-algebra. 
    \begin{enumerate}
        \item If for any \cst-algebra $B$, $\calG$ satisfies the Baum--Connes conjecture with coefficients in $A\otimes B$, then $A\in \ca N_\calG$ if and only if $A\rtimes_r\calG\in \ca N$.
        \item If $A\in \ca N_\calG$ and $\calG$ satisfies the Baum--Connes conjecture with coefficients in $A$, then $A\rtimes_r\calG\in \ca N$ if and only if $\calG$ satisfies the Baum--Connes conjecture with coefficients in $A\otimes B$ for any \cst-algebra $B$.
    \end{enumerate}
\end{prop}

\begin{thm}\label{going down class N}
    Let $\calG$ be a second countable \'etale groupoid, $A$ be a separable exact $\calG$-\cst-algebra. Assume that for all proper open subgroupoid $\calH\subseteq \calG$, $A|_\calH\rtimes \calH\in \ca N$. Then $A\in \ca N_\calG$.
\end{thm}
\begin{proof}
    Fix a \cst-algebra $B$ with $K_*(B)$ free. For every $\calH\in \ca S(\calG)$, define two contravariant functors $ F^*_\calH, \ca G^*_\calH$ from $\ca C(\calH)$ to the category of $\mathbb Z/2\mathbb Z$-graded abelian groups as
    \[ F^*_\calH=\kk_*^\calH(-, A|_\calH)\otimes_\Z K_*(B),\]
    \[G^*_\calH=\kk_*^\calH(-,A|_{\calH}\otimes B).\]
    
    Since $A|_\calH\otimes B\cong (A\otimes B)|_\calH$, and tensor product with $K_*(B)$ preserves exactness, by proposition \ref{going-down functor example 1}, $( F^n_\calH)_{\calH\in \ca S(\calG)}, (G^n_\calH)_{\calH\in \ca S(\calG)}$ are going-down functors.

    If $f:C_0(Y_1)\ra C_0(Y_2)$ is a morphism in $\ca C(\calH)$,  for any $x\in \kk_*^\calH(C_0(Y_2),A|_\calH)$ and $y\in K_*(B)$,
    \[f^*(\alpha_{Y_2}(x\otimes_\Z y))=f^*(x\otimes \epsilon(y))=f^*(x)\otimes \epsilon(y)=\alpha_{Y_1}(f^*(x)\otimes_\Z y).\]
    That is the following diagram commutes.
    \[\xymatrix{
    \kk_*^\calH(C_0(Y_2),A|_\calH)\otimes_\Z K_*(B) \ar[r]^{\alpha_{Y_2}} \ar[d]_{f^*\otimes id} & \kk_*^\calH(C_0(Y_2), A|_\calH\otimes B) \ar[d] \ar[d]_{f^*}\\
    \kk^\calH_*(C_0(Y_1),A|_\calH) \ar[r]^{\alpha_{Y_1}} & \kk_*^\calH(C_0(Y_1),A|_\calH\otimes B)
    }\]
    
    Therefore, the maps $\alpha_Y$ for all locally compact Hausdorff proper $\calH$-spaces $Y$ constitute a natural transformation $\Lambda^*_\calH: F^*_\calH\ra \ca G^*_\calH$. Let $H_1\subseteq H_2$ be two elements of $\ca S(\calG)$ and $Y$ is a locally compact Hausdorff proper $\calH_1$-space, $\Omega:=(\calH_2)_{\calH_1\units}$, for any $x\in \kk_*^{\calH_2}(C_0(\Omega\times_{\calH_1}Y),A|_{\calH_2})$, $y\in K_*(B)$,
    \begin{align*}
        \alpha_{Y}\circ (comp^{\calH_2}_{\calH_1}\otimes_{\mathbb Z} id)(x\otimes_{\mathbb Z}y) & = \alpha_Y(comp^{\calH_2}_{\calH_1}(x)\otimes_{\Z} y)\\
        & = comp^{\calH_2}_{\calH_1}(x)\otimes \epsilon(y)\\
        & = [i_{C_0(Y)}]\otimes res^{\calH_2}_{\calH_1}(x)\otimes \epsilon(y)\\
        & = [i_{C_0(Y)}]\otimes res^{\calH_2}_{\calH_1}(x\otimes \epsilon(y))\\
        & = comp^{\calH_2}_{\calH_1}(x\otimes \epsilon(y))\\
        & = comp^{\calH_2}_{\calH_1}\circ \alpha_{\Omega\times_{\calH_1}Y}(x\otimes_\Z y)
    \end{align*}
    So the following diagram commutes.
    \[\xymatrix{
    \kk^{\calH_1}_*(C_0(Y),A|_{\calH_1})\otimes_\Z K_*(B) \ar[d]^{\alpha_{Y}} & \kk^{\calH_2}_*(C_0(\Omega\times_{\calH_1}Y),A|_{\calH_2})\otimes_\Z K_*(B) \ar[d]^{\alpha_{\Omega\times_{\calH_1}Y}} \ar[l]^{\hspace{-0.5cm}comp^{\calH_2}_{\calH_1}\otimes_\Z id}  \\
    \kk_*^{\calH_1}(C_0(Y), A|_{\calH_1}\otimes B) & \kk_*^{\calH_2}(C_0(\Omega\times_{\calH_1}Y), A|_{\calH_1}\otimes B) \ar[l]^{comp^{\calH_2}_{\calH_1}}.
    }\]
    So $\Lambda=(\Lambda^n_\calH)_{\calH\in \ca S(\calG)}$ is compatible with induction. Clearly $\Lambda$ is compatible with suspension. In conclusion $\Lambda$ is a going-down transformation.

    By proposition \ref{connect Kunneth with BC}, for all proper open subgroupoids $\calH$, $A|_{\calH}\rtimes \calH\in \ca N$ if and only if $A|_\calH\in \ca N_{\calH}$. This implies that $\alpha_{\calH\units}=\alpha_\calH$ is an isomorphism. Then by theorem \ref{going-down principle}, $\alpha_\calG$ is an isomorphism, that is $A\in \ca N_\calG$.
\end{proof}

\begin{corr}\label{crossed prod in N}
    Let $\calG$ be a second countable \'etale groupoid and $A$ be a separable exact $\calG$-\cst-algebra. If $A$ satisfies the following assumption:
    \begin{enumerate}
        \item for all proper open subgroupoid $\calH\subseteq \calG$, $A|_\calH\rtimes \calH\in \ca N$;
        \item for all separable \cst-algebra $B$, $\calG$ satisfies the Baum--Connes conjecture with coefficient in $A\otimes B$.
    \end{enumerate}
    Then $A\rtimes_r \calG\in \ca N$.
\end{corr}
\begin{proof}
    Using theorem \ref{going down class N} and proposition \ref{connect Kunneth with BC}.
\end{proof}

\begin{corr}
    Let $\calG$ be a second countable \'etale groupoid, $\Sigma$ be a twist over $\calG$. If $\calG$ satisfies the Baum--Connes conjecture with coefficients, then $C_r^*(\calG;\Sigma)\in \ca N$.
\end{corr}
\begin{proof}
    By \cite[Proposition 5.1]{vanErpWilliams}, there exists a Hilbert $C_0(\calG\units)$-module $E$ and an action of $\calG$ on $\calK(E)$, such that $C_r^*(\calG;\Sigma)$ is Morita equivalent to $\calK(E)\rtimes_r\calG$. For all $x\in \calG\units$, $\calK(E)_x\cong \calK(E_x)$ is of type I, then by \cite[Proposition 10.3]{tu1999conjecture-bc}, for each proper open subgroupoid $\calH$ of $\calG$, $\calK(E)|_\calH\rtimes\calH$ is of type I and hence belongs to $\ca N$ (see \cite[22.3.4, Theorem 23.1.3]{blackadar1998k}). Then corollary \ref{crossed prod in N} implies that $\calK(E)\rtimes_r\calG\in \ca N$. So we have $C_r^*(\calG;\Sigma)\in \ca N$ because of the stability of $\ca N$ under Morita equivalence.
\end{proof}

\addcontentsline{toc}{section}{References}
\bibliography{sample}{}
\bibliographystyle{plain}

\end{document}